\documentclass[10pt]{article}
\usepackage{amsmath,stackengine}
\usepackage{amssymb}
\usepackage{titlesec}
\usepackage[T1]{fontenc}
\usepackage{lmodern}
\usepackage{graphicx}
\usepackage{titling}
\usepackage[english]{babel}
\usepackage[margin=2.5cm]{geometry}
\usepackage{csquotes}
\usepackage{lipsum}
\usepackage[english]{babel}
\usepackage{amsthm}
\usepackage{amssymb}
\usepackage{bm}
\usepackage{hyperref}
\usepackage[backend=biber]{biblatex}
\usepackage{mathrsfs}
\usepackage{color}
\usepackage{mathtools,amsfonts,amssymb,setspace}

\usepackage{enumitem}
\usepackage{subfigure}
\usepackage[font=small,labelfont=bf]{caption}
\newtheorem{theorem}{Theorem}[section]

\newtheorem{lemma}[theorem]{Lemma}
\theoremstyle{definition}
\newtheorem{definition}{Definition}[section]
\newtheorem{prop}[theorem]{Proposition}
\theoremstyle{definition}

\theoremstyle{definition}
\newtheorem{remark}{Remark}

\newcommand{\subjclass}[1]{\small	\par\noindent\textbf{Mathematics Subject Classification (2020): }#1\par}
\newcommand{\keywords}[1]{\small	\par\noindent\textbf{Keywords: }#1\par}
\DeclareMathAlphabet\mathbfcal{OMS}{cmsy}{b}{n}
\newcommand{\RomanNumeralCaps}[1]{\MakeUppercase{\romannumeral #1}}

\definecolor{green}{rgb}{0.1,0.62,0.0}
\definecolor{blue}{rgb}{0.0, 0.0, 1.0}
\definecolor{deepblue}{rgb}{0.0,0.0,0.7}
\definecolor{red}{rgb}{1.0, 0.0, 0.0}

\font\myfont=cmbx12 at 11pt
\font\myfond=cmr12 at 10pt
\title{\myfont POINTWISE TRACKING OPTIMAL CONTROL PROBLEM FOR CAHN-HILLIARD NAVIER-STOKES SYSTEM }
\author{\myfond SHEETAL DHARMATTI \thanks{sheetal@iisertvm.ac.in \; School of Mathematics, Indian Institute of Science Education and Research, Thiruvananthapuram, Maruthamala PO, Vithura, Thiruvananthapuram, Kerala, 695551, INDIA.}  \;  and  \;  GREESHMA K \thanks{ greeshmak21@iisertvm.ac.in \; School of Mathematics, Indian Institute of Science Education and Research, Thiruvananthapuram, Maruthamala PO, Vithura, Thiruvananthapuram, Kerala, 695551, INDIA. }
} 
\date{}

\numberwithin{equation}{section}
\thanksmarkseries{arabic}

\begin{document}

\maketitle

\begin{abstract}
    We study a pointwise tracking optimal control problem for the two-dimensional local Cahn–Hilliard Navier–Stokes system, which models the evolution of two immiscible, incompressible fluids. The source term in the Cahn–Hilliard equation acts as a control, and the cost functional measures the deviation of the phase variable from desired values at a finite set of spatial points over time. This setting reflects realistic applications where only a limited number of sensors are available. We also study a variant of the above pointwise tracking control problem where the cost is incorporated with a terminal time pointwise tracking term. The main mathematical difficulty arises from the low regularity of the cost functional due to the pointwise evaluation of the state variables. We prove the existence of strong solutions, establish the existence of an optimal control, and the differentiability of the control-to-state mapping. We define the adjoint system using a transposition method to characterise optimal control. Moreover, a first-order necessary optimality condition is derived in terms of the adjoint for both problems. Furthermore, we prove that our analysis can be extended to the case of singular potentials.
\end{abstract}

\setcounter{section}{1}
\section*{Introduction}

In many practical problems, it is difficult to apply control or collect data in the entire domain. Instead, a control is often applied at a finite number of points, and/or observations are taken only at specific points. This motivates studying the optimal control problems where the control is applied at, or the data is observed at, a finite number of spatial points.
These problems are particularly relevant in applications such as fluid flow control and temperature regulation, where the use of pointwise sensors is more realistic and cost-effective \cite{ZLJ, ZLJ3}. However, the mathematical analysis is complex, since pointwise data evaluation requires higher regularity of solutions. In this work, we consider two pointwise tracking optimal control problems for the coupled Cahn-Hilliard Navier-Stokes system. One for pointwise tracking at almost all times, and the second, where, along with pointwise tracking at almost all times, pointwise tracking at the terminal time is also included.

To describe the problem at hand, let $\Omega \subseteq \mathbb{R}^2$ be an open bounded domain with a smooth boundary $\partial \Omega$. Consider a finite subset of $\Omega$, say, $\mathcal{D} = \{ x_1, x_2,..., x_k\}$. Define $\Phi^i(t): [0,T] \rightarrow \mathbb{R}$, the desired trajectory at point $x_i$, for $1 \leq i\leq k$ and $\textbf{u}_d$ denote the desired velocity field on $\Omega\times [0,T]$. We first study the following optimal control problem: minimise a tracking type cost functional given by,
\begin{align}\label{eqJ0}
    \mathcal{J}_1(\varphi,\textbf{u}, U) &= \frac{1}{2}\sum\limits_{i=1}^{k}\int\limits_{0}^{T}(\varphi(x_i,t)-\Phi^i(t))^2 dt
     +\frac{1}{2}\int\limits_{0}^{T}\int_\Omega(\textbf{u}-\textbf{u}_d)^2 dxdt + \frac{1}{2}\int_\Omega(\textbf{u}(x,T)-\textbf{u}_d(x,T))^2 dx
    + \frac{1}{2}\int\limits_{0}^{T}\int_\Omega|U|^2 dxdt
\end{align}
over an admissible set of controls, where $(\varphi, \textbf{u})$ solve the following CHNS system  derived in \cite{AMW, GPV, LKW}.
\begin{align}
\varphi'+\textbf{u}\cdot\nabla\varphi &=\ \nabla\cdot( m(\varphi)\nabla\mu) + U,     & \text{in}\;\; \Omega\times [0,T],\label{eq8}
\\\mu &=\ -\Delta\varphi + F '\left(\varphi\right),    &\text{in}\;\; \Omega\times [0,T],\label{eq9}
\\\textbf{u}'-2\nabla\cdot\left(\nu\left(\varphi\right)D\textbf{u} \right) + (\textbf{u}\cdot\nabla) \textbf{u} + \nabla \pi &=\ \mu\nabla\varphi ,   & \text{in}\;\; \Omega\times [0,T],\label{eq10}
\\ \nabla\cdot \textbf{u}  & =  0, & \text{in}  \;\; \Omega\times [0,T].\label{eq11}
\end{align}
with initial and boundary conditions,
\begin{align}
\frac{\partial \varphi}{\partial n} = \frac{\partial \mu}{\partial n} &= 0 , \text{ on } \partial\Omega\times [0,T], \label{eq12}
\\\textbf{u} &=\textbf{0},  \text{ on }    \partial\Omega\times [0,T],\label{eq13}
\\ \varphi\left(x , 0 \right) =\varphi_{0}, \textbf{u}(x,0) &= \textbf{u}_0 \text{ in }  \Omega.\label{eq14}
\end{align}
Note that $D\textbf{u}$ denotes the symmetric gradient $(\nabla\textbf{u}+ \nabla^{T}\textbf{u})/2$. Here, $\mu$ denotes the chemical potential, and the above system is local. We get the nonlocal system by replacing \eqref{eq10} by $\mu = - J \ast\varphi + F '\left(\varphi\right)$. In this work, we consider a regular potential of polynomially controlled growth and a constant mobility. For example, consider the double-well potential given by,
 \begin{equation}
     F(s) = \frac{(1-s^2)^2}{4}, \hspace{.5cm} s\in [-1,1].
 \end{equation}
 Here, the control is applied as an external mass source term, $U$. 
 In the first problem, we consider the cost functional ${\mathcal{J}_1}$ in which the desired state at pre-decided points $x_i$ should be close to  $\Phi^i(t) = \varphi^i$ for $t\in [0, T] \; a. e. $, where $\varphi^i \in \mathbb{R}$ is a constant and hence $\Phi^i(t)$ is  a  constant trajectory.

Our second (more involved) problem is a variant of the first problem, in which we retain the same admissible control set and the state system, while the objective functional is modified to include pointwise tracking at the terminal time.  
We consider the cost functional,
\begin{align}\label{J1}
    \mathcal{J}_2(\varphi,\textbf{u}, U) &= \frac{1}{2}\sum\limits_{i=1}^{k}\int\limits_{0}^{T}(\varphi(x_i,t)-\Phi^i(t))^2 dt
    + \frac{1}{2}\sum\limits_{i=1}^{k}(\varphi(x_i,T)-\Phi^i(T))^2 +\frac{1}{2}\int\limits_{0}^{T}\int_\Omega(\textbf{u}-\textbf{u}_d)^2 dxdt \nonumber\\ &\hspace{.5cm}+ \frac{1}{2}\int_\Omega(\textbf{u}(x,T)-\textbf{u}_d(x,T))^2 dx
    + \frac{1}{2}\int\limits_{0}^{T}\int_\Omega|U|^2 dxdt
\end{align}
Due to the additional term in $\mathcal{J}_2$, we need to study higher regularity of the solution as well as differentiability properties of the control to state operator in a strong space.

 The above cost functional are physically relevant and practically applicable, and have been studied in the literature for some stationary systems. Most of the papers which study the pointwise tracking optimal control are based on numerical approaches \cite{AFO, BND, CAC, LWN, FOE, FED, ZLJ4}. In a recent work by G. Peralta \cite{PRG3}, an optimal control problem for the stationary Stokes equation is studied where the objective functional involves a pointwise evaluation of velocity, stress and pressure. However, as far as our knowledge of evolutionary systems, the pointwise tracking type optimal control problems have not been studied yet. 

Nevertheless, the optimal control problem for evolution equations, where the control is applied only at a finite number of spatial points, has been studied in the literature for the Boussinesq type system \cite{JRP, PJP}.
Later, K. Kunisch and E. Casas studied a similar problem for parabolic equations and Navier-Stokes equations in a more general setting where the control applied is a measurable function \cite{CAK, CAK2, CKK}. Furthermore, G. Peralta \cite{PRG1, PRG2} extended a similar analysis to the viscous Cahn-Hilliard-Oberbeck-Boussinesq phase-field system, and in particular, these results apply to the CHNS system. Distributed optimal control problems for CHNS systems are studied in  \cite{BSM, TMD, FGS} and references therein and boundary control problems in  \cite{MTD}.



In this work, we consider the pointwise tracking optimal control problem of a local CHNS system with a regular potential. As mentioned earlier, the pointwise tracking type optimal control problem for evolutionary equations has not been studied in the literature. The novelty of our work lies in treating evolutionary systems with strong nonlinear coupling, which raises new analytical challenges and opportunities. The major difficulty arises from the low regularity of the cost functionals $\mathcal{J}_1$ and $\mathcal{J}_2$, which involve point evaluations of the phase variable $\varphi$. To justify the differentiability of the control to state operator and derive first-order necessary conditions, a strong solution of the system is required. Moreover, the associated adjoint system contains a measure valued source term arising from the Dirac evaluations at the observation points. To handle this, we adopt a transposition method, defining the solutions of  the adjoint state 
proving well-posedness using duality arguments.
The optimal tracking in the terminal time problem is even more challenging due to a low regular term in the terminal time in the cost. Increasing the regularity of parameters of the system enables us to deduce a higher regular solution and establish differentiability results for the control to state operator, which in turn are needed for obtaining the optimality conditions. We use the transposition method to define the solution of the adjoint system, and then we characterise the optimal control in terms of the adjoint state. 


 Later on, we study the problem for the local system with a singular potential. A primary example of a singular potential which is practically relevant is the logarithmic potential defined by,
\begin{equation}\label{eq15}
F_{\log}(s) = (1+s)\log(1+s) + (1-s)\log(1-s), \hspace{.5cm} s\in(-1,1),
\end{equation}
which is widely used in phase-field models. To analyse the optimal control problem, we require the existence of a unique strong solution. The main difficulty in establishing higher regularity arises from the blow-up of derivatives of the potential. To overcome this, we adopt the methods developed in \cite{AMT, JNG, JWH}. In particular, by exploiting the separation property established in \cite{JNG}, and by restricting $F$ to compact subsets of (-1, 1), we can avoid the blow up to obtain solutions with higher regularity. Once this higher regularity is established, the remaining proof can be done on similar lines of proof for a regular potential. 

In the following section, we discuss the mathematical setup, the notations used, and some preliminary results on the well-posedness of the system. In Section 3, we study the pointwise tracking optimal control problem for the cost functional given by \eqref{eqJ0}. After formally stating the problem, we prove the existence of an optimal control. To characterise the optimal control, we first investigate the well-posedness results of the linear system using standard techniques. We also establish higher regularity of the solution of the linearised variable, which is essential for defining the solution of the adjoint system by the transposition method. Furthermore, we derive a necessary condition for optimality in terms of the adjoint variable. In section 4, we study the pointwise tracking in the terminal time problem corresponding to \eqref{J1}. We study the existence of an optimal control and establish a differentiability result for the control to state operator in a strong space. Regularity estimates for the differentiability of the control to state operator are derived in the appendix. Moreover, we deduce the optimality conditions using the adjoint system. The analysis of the problem with a singular potential is presented in Section 5, and we end the paper with concluding remarks and outline some future directions.

\section{Notations and Preliminary results}

Let $\Omega \subseteq \mathbb{R}^2$ be an open, bounded, connected domain with a smooth boundary $\partial\Omega$. We denote $\Omega\times [0,T]$ and $\partial\Omega\times [0,T]$ by $Q$ and $\Sigma$ respectively. Let $H$ denote the set of square integrable functions on $\Omega$ and $\|.\|$ , $(.,.)$ be the norm and inner product in the space $H$ respectively. We denote the collections of functions in $H $, whose first order weak derivative is also in $H$ by $V$. The norm and inner product in this space is denoted by $\|.\|_{V}$ and $(.,.)$ respectively. 
The inner product in $V$ is defined by,
\begin{align*}
    (f,g) = (\nabla f, \nabla g), \hspace{.5cm} \text{ for } f,g\in V.
\end{align*}
We denote by bold font $\textbf{H}$ and $\textbf{V}$ the spaces $H\times H$ and $V \times V$.
Let the $\mathcal{D}(\Omega)$ represent the set of $C^\infty$ functions on $\Omega$ that are divergence free and with a compact support. The closure of $\mathcal{D}(\Omega)$ in $\textbf{H}$ and $\textbf{V}$ is denoted by $\mathbb{G}_{div}$ and $\mathbb{V}_{div}$ respectively. From \cite{TEM}, a characterisation of above divergence free spaces follows. 
\begin{align*}
    \mathbb{G}_{div} &= \{ \textbf{u}\in L^2(\Omega;\mathbb{R}^2): \text{div}(\textbf{u})=0, \textbf{u} \cdot \textbf{n}|_{\partial \Omega}=0\},\\
 \mathbb{V}_{div} &= \{ \textbf{u}\in H^1_0(\Omega;\mathbb{R}^2): \text{div}(\textbf{u})=0\}.
\end{align*}
Here $\|.\|$, $\|.\|_{\mathbb{V}_{div}}$ denotes the norm in $\mathbb{G}_{div}$ and $\mathbb{V}_{div}$ respectively. Further, $C^{0,\gamma}(\Omega)$ denotes the holder continuous function on $\Omega$. Furthermore, define,
\begin{equation*}
H^2_N = \{ f\in H^2(\Omega) : \frac{\partial f}{\partial \textbf{n}}  = 0 \text{ on } \partial\Omega\}.    
\end{equation*}
Further, we define the Stokes operator $\mathbfcal{S} : \textbf{D}(\mathbfcal{S})= \mathbb{V}_{div}\cap \textbf{H}^2$ onto the space $\mathbb{G}_{div}$ by,
\begin{equation*}
    (\mathbfcal{S}\textbf{u}, \textbf{w}) = (\nabla\textbf{v}, \nabla\textbf{w}), \forall \textbf{w} \in \mathbb{V}_{div}.
\end{equation*}

Now we will introduce a few inequalities that will be used frequently in following analysis. 
\\\\\textit{Gagliardo-Nirenberg Inequality:} Let $1\leq q \leq +\infty$, $0\leq j< m$, $1\leq r\leq +\infty$, $p\geq 1$ and $\theta \in [0,1]$ satisfies, 
\begin{equation*}
    \frac{1}{p} = \frac{j}{n} + \theta \big( \frac{1}{r}- \frac{m}{n}\big) + \frac{(1-\theta)}{q},  \hspace{.5cm} \frac{j}{m} \leq \theta \leq 1, 
\end{equation*}
Then,
 \begin{equation}\label{eq1.1b}
    \|D^j u\|_{L^p(\Omega)} \leq C\|D^m u\|_{L^r(\Omega)} ^\theta\| u\|_{L^q(\Omega)}^{(1-\theta)} + \|u\|_{L^s(\Omega)},
\end{equation}
 holds true for any $u\in L^q(\mathbb{R}^n)$ with $D^mu\in L^r(\mathbb{R}^n)$ where $s$ is arbitrary with the following exception \cite{BRZ}.
 \begin{itemize}
     \item If $r>1$, $m-j-\frac{n}{r} \geq 0$, then we need the additional condition $\frac{j}{m} \leq \theta < 1$.\label{c2}
\end{itemize}
 The constant $C$ depends only on $n, m, j, q, r$ and $\theta$. 
Equivalently, if we choose $s = \min\{p, q\}$ then \eqref{eq1.1b} is equivalent to the following \cite{BRZ}. 
\begin{equation}\label{eq1.1c}
    \|D^j u\|_{L^p(\Omega)} \leq C\|u\|_{W^{m,r}(\Omega)} ^\theta\| u\|_{L^q(\Omega)}^{(1-\theta)}.
\end{equation}
\\\\\textit{Agmon's Inequality:}
Let $u\in H^2(\Omega)$ then the following inequality holds \cite{TEM}.
\begin{equation}
    \|u\|_{L^\infty(\Omega)} \leq \|u\|_{H^2(\Omega)}^{1/2}\|u\|^{1/2}.
\end{equation}
\\We set the following assumptions on the mobility, the viscosity, external forcing term and the potential.
\begin{enumerate}[font={\bfseries},label={A\arabic*.}]
    \item[{[A1]}] The mobility, $m$ and viscosity, $\eta$ are continuous and satisfy,
    \begin{align*}
     m_0\leq m(s)\leq m_1,\hspace{.5cm} \eta_0\leq \eta(s)\leq \eta_1, \hspace{.5cm}\forall s\in [-1,1], 
    \end{align*}
    where $m_0$, $m_1$, $\eta_0$, $\eta_1$ are positive constants.
    \item[{[A2]}] The potential, $F\in C^2(\mathbb{R})$ is non-negative and satisfies the following,
    \begin{align}\label{F1}
        |F''(s)|\leq C_0(1+|s|^r), \hspace{.5cm}|F'(s)|\leq C_1F(s)+ C_2, \hspace{.5cm} F(s) \geq C_3|s|^2-C_4,\,\, \forall s\in \mathbb{R}
    \end{align}
    for $r\in [0,\infty)$ with positive constants $C_0$, $C_1$, $C_2$, $C_3$ and $C_4$ that are independent of $s$. 
\end{enumerate}

To investigate the optimal control problem (OCP), it is essential first to establish the well-posedness of the system \eqref{eq8}–\eqref{eq14}. Moreover, we need to examine the existence of strong solutions to ensure that the cost functional, which involves pointwise information, is well-defined. The existence results and the continuous dependence of the solution on the source term follow directly from the results established in \cite{LKW} for the CHNS system with chemotaxis and mass source term. Though we will be studying the control problem for dimension 2 only,  the existence results stated below holds true for  dimensions two and three.

\begin{prop}[\textit{Existence of a weak solution to the local CHNS system}]\label{prop1}
    Let $d\in \{ 2,3 \}$, $\varphi_0 \in V$ and $\textbf{u}_0 \in \mathbb{G}_{div}$. The external source term, $U\in L^2(0,T;H)$. Assume \textbf{[A1]},\textbf{[A2]} holds. Then there exist a weak solution $(\varphi, \textbf{u})$ to the system \eqref{eq8}-\eqref{eq14} that satisfies,
    \begin{align*}
    \varphi &\in L^\infty(0,T; V)\cap L^2(0,T;H^3)\cap H^1(0,T;V'),\\
    \mu &\in L^2(0,T;V),\\
    \textbf{u}&\in L^\infty(0,T;\mathbb{G}_{div})\cap L^2(0,T;\mathbb{V}_{div}) \cap H^1(0,T;\mathbb{V}_{div}').
    \end{align*}
And a weak formulation,    
    \begin{align}
      \langle\varphi', \psi\rangle+ (\textbf{u}\cdot\nabla\varphi, \psi)+ ( m(\varphi)\nabla(-\Delta\varphi+F'(\varphi)), \nabla\psi) &= (U, \psi), \label{eqw1}\\    
      \langle\textbf{u}', \textbf{v}\rangle+ 2(\left(\nu\left(\varphi\right)D\textbf{u}, D\textbf{v} \right) + ((\textbf{u}\cdot\nabla) \textbf{u}, \textbf{v}) &= (\mu\nabla\varphi, \textbf{v}) ,\label{eqw3}
    \end{align}
for all $\psi \in V$ and $\textbf{v}\in \mathbb{V}_{div}$. Furthermore,
    \begin{align*}
    \varphi\left(x , 0 \right) &=\varphi_{0}(x),\hspace{.5cm} \textbf{u}(x,0)= \textbf{u}_0(x), \text{ a.e in} \; \; \Omega.\end{align*} 
\end{prop}

\begin{remark}
    Note that in \cite[Theorem 1]{LKW} the assumption is that the source term $U \in L^2(0,T; L^\infty(\Omega))$. We can replace it with a comparatively weaker assumption,  $U \in L^2(0,T; H)$ and still  obtain the same result. 
\end{remark}

We derive strong solution results for the system under consideration, before deriving the continuous dependence estimates.  We set the following additional assumptions to achieve the goal.

\begin{enumerate}[font={\bfseries},label={A\arabic*.}]
    \item[{[A3]}] The mobility, $m \in C_b^2(\mathbb{R})$ and viscosity, $\eta \in C_b^1(\mathbb{R})$ with bounded derivatives. 
    \item[{[A4]}] The potential, $F\in C^3(\mathbb{R})$ and for some $r \in [1, \infty)$, satisfies
    \begin{align}\label{eqF4}
        |F'''(s)| &\leq C_0(1+ |s|^{r-1}),
    \end{align}
   for all $s\in \mathbb{R}$ and $C_0$ is a constant independent of $s$.
    \item[{[A5]}] The source term, $U \in L^2(0,T;V\cap L^\infty(\Omega))$.
\end{enumerate}
\medskip
\begin{prop}[\textit{Existence of a Strong Solution to the Local CHNS system}]\label{prop2}
Let $d =2$, assumptions \textbf{[A1]}- \textbf{[A5]}
   be satisfied, $\varphi_0 \in H^3(\Omega)$ and $\textbf{u}_0 \in \mathbb{V}_{div}$. Then there exist a strong solution $(\varphi, \textbf{u})$ of the system \eqref{eq8}-\eqref{eq14} which satisfies,
    \begin{align*}
    \varphi &\in L^\infty(0,T; H^3)\cap L^2(0,T;H^4)\cap H^1(0,T;V) \cap C([0,T];C^{0,\gamma}(\Bar{\Omega})),\\
    \mu &\in L^\infty(0,T;V)\cap L^2(0,T;H^3),\\
    \textbf{u}&\in L^\infty(0,T;\mathbb{V}_{div})\cap L^2(0,T;\textbf{H}^2) \cap H^1(0,T;\mathbb{G}_{div}).
    \end{align*}
Moreover,    
   \begin{align*}
    \varphi\left(x , 0 \right) &=\varphi_{0}(x),\hspace{.5cm} \textbf{u}(x,0)= \textbf{u}_0(x), \text{ a.e in} \; \; \Omega.\end{align*} 
\end{prop}
The proof of the above proposition follows by similar arguments as in  \cite[Theorem 2]{LKW}. To prove the H\"older continuity of $\varphi$ in variable $x$ is obtained by an application of Aubin-Lions lemma, $L^\infty(0,T;H^2)\cap H'(0,T;H) \hookrightarrow \hookrightarrow C([0,T];C^{0,\gamma}(\Bar{\Omega}))$.
\medskip

\begin{prop}[Continuous dependence of the  strong solution on initial data and mass source term]\label{prop3}
Let $(\varphi_i, \textbf{u}_i)$, $i=1,2$ be two strong solutions of the system \eqref{eq8}-\eqref{eq14} corresponding to initial data $(\varphi_{0i}, \textbf{u}_{0i})$ and source term $U_i$. Assume $m$ and $\eta$ are Lipschitz continuous. Additionally, $F$ satisfy, 
\begin{align}\label{F5}
    |F'(s_1)-F'(s_2)| &\leq C_5(1+|s_1|^r+|s_2|^r)|s_1-s_2|,\hspace{.5cm} \forall s_1, s_2\in \mathbb{R}.
\end{align}
for some $r>0$ and a constant $C_5>0$. Then there exists a constant $C>0$ such that the following estimate holds true.
\begin{align}\label{cd1}
    \|\varphi_1-\varphi_2\|_{L^\infty(0,T;H)\cap L^2(0,T;H^2)}^2 &+\|\textbf{u}_1-\textbf{u}_2\|_{L^\infty(0,T;\mathbb{G}_{div})\cap L^2(0,T;\mathbb{V}_{div})}^2  + \|\mu_1-\mu_2\|_{L^2(0,T;H)}^2 \nonumber\\&\leq C\big( \|\textbf{u}_{01}-\textbf{u}_{02}\|^2 +\|\varphi_{01}-\varphi_{02}\|^2 +\|U_1-U_2\|_{L^2(0,T;H)}^2 \big).
\end{align}
Additionally, we have the following continuous dependence estimates.
\begin{align}\label{cd2}
    \|\varphi_1-\varphi_2\|_{L^\infty(0,T;V)\cap L^2(0,T;H^3)}^2 + \|\varphi_1'-\varphi_2'\|_{L^2(0,T;V')}^2 \leq C\big( \|\textbf{u}_{01}-\textbf{u}_{02}\|^2 +\|\varphi_{01}-\varphi_{02}\|_{V}^2 +\|U_1-U_2\|_{L^2(0,T;H)}^2 \big).
\end{align}
Furthermore, if we assume $F\in C^4(\mathbb{R})$ then,
\begin{align}\label{cd3}
     \|\varphi_1-\varphi_2\|_{L^\infty(0,T;H^2)\cap L^2(0,T; H^4)}^2  \leq C\big( \|\textbf{u}_{01}-\textbf{u}_{02}\|^2 +\|\varphi_{01}-\varphi_{02}\|_{H^2}^2 +\|U_1-U_2\|_{L^2(0,T;H)}^2 \big).
\end{align}
\end{prop}
\begin{proof}
    The first part of the proposition follows by similar arguments as in \cite[Theorem 3]{LKW}. We will prove estimates \eqref{cd2} and \eqref{cd3}. Let us  denote $\varphi = \varphi_1-\varphi_2$, $\textbf{u} = \textbf{u}_1-\textbf{u}_2$, $\varphi_{0}=\varphi_{01}-\varphi_{02}$ and $\textbf{u}_0 = \textbf{u}_{01}-\textbf{u}_{02}$. Then $(\varphi, \textbf{u})$ solves the difference system,
\begin{align}
    \varphi'+\textbf{u}_1\cdot\nabla\varphi +\textbf{u}\cdot\nabla\varphi_2 &= \nabla\cdot\big(m(\varphi_1)\nabla\Tilde{\mu}\big) + \nabla\cdot\big((m(\varphi_1)-m(\varphi_2))\nabla\mu_2\big)+ U_1-U_2,\label{eq1.7}\\
    \Tilde{\mu} &= -\Delta\varphi+ F'(\varphi_1)-F'(\varphi_2),\label{eq1.8}\\
    \textbf{u}'+(\textbf{u}_1\cdot\nabla)\textbf{u}+(\textbf{u}\cdot\nabla)\textbf{u}_2 &+ \nabla\tilde{q} -2\nabla\cdot(\eta(\varphi_1)D\textbf{u})-2\nabla\cdot((\eta(\varphi_1)-\eta(\varphi_2))D\textbf{u}_2) = \mu_1\nabla\varphi + \Tilde{\mu}\nabla\varphi_2,\label{eq1.9}
\end{align}
\\where $\mu_i =(-\Delta\varphi_i+ F'(\varphi_i))$, $i=1,2$. First we test \eqref{eq1.7} with $-\Delta\varphi$, which yields,
\begin{align}
  \frac{1}{2}\frac{d}{dt}\|\nabla\varphi\|^2 + (\textbf{u}\cdot\nabla\varphi_2, -\Delta\varphi) &= -\big(m(\varphi_1)\nabla\Tilde{\mu}, \nabla(-\Delta\varphi)\big) - \big((m(\varphi_1)-m(\varphi_2))\nabla\mu_2, \nabla(-\Delta\varphi)\big)\nonumber \\&\hspace{.5cm}+ (U_1-U_2,-\Delta\varphi)  
\end{align}
Estimating the terms one by one, using H\"olders, Gagliardo-Nirenberg, Ladyzhenskaya, Poincar\'e, Young's inequalities, and assumption [A1], we obtain,
\begin{align*}
   |(\textbf{u}\cdot\nabla\varphi_2, -\Delta\varphi)| &\leq \|\textbf{u}\|_{L^4}\|\nabla\varphi_2\|_{L^4}\|\Delta\varphi\| 
   \leq \|\nabla\textbf{u}\|\|\varphi_2\|_{H^2}\|\Delta\varphi\| \\
   &\leq \|\varphi_2\|_{H^2}^2\|\nabla\textbf{u}\|^2 + C\|\Delta\varphi\|^2.\\
    \big(m(\varphi_1)\nabla(-\Delta\varphi), \nabla(-\Delta\varphi)\big) &\geq m_0\|\nabla^3\varphi\|^2.\\
    |\big(m(\varphi_1)\nabla(F'(\varphi_1)-F'(\varphi_2)), \nabla(-\Delta\varphi)\big)| &=  |\big(m(\varphi_1)((F''(\varphi_1)-F''(\varphi_2))\nabla\varphi_1 + F''(\varphi_2)\nabla\varphi), \nabla(-\Delta\varphi)\big)| \\
    &\leq m_1\|F'''(\Hat{\varphi})\|_{L^\infty}\|\varphi\|_{L^4}\|\nabla\varphi_1\|_{L^4}\|\nabla^3\varphi\| + m_1\|F''(\varphi_2)\|_{L^\infty}\|\nabla\varphi\|\|\nabla^3\varphi\|,  \\
    &\leq \frac{m_0}{4}\|\nabla^3\varphi\|^2 + C\|\varphi_1\|_{H^2}^2\|\varphi\|\|\varphi\|_{V} +  C\|\nabla\varphi\|^2 \\
    &\leq \frac{m_0}{4}\|\nabla^3\varphi\|^2 + C( 1+ \|\varphi_1\|_{H^2}^2)\|\varphi\|_{V}^2
\end{align*}
The last inequality is obtained by applying assumptions on the potential \eqref{F1}, \eqref{eqF4} and the uniform global bound for the strong solution. Furthermore, we use Agmon's inequality and Sobolev embeddings to estimate the following. 
\begin{align*}
    |\big((m(\varphi_1)-m(\varphi_2))\nabla\mu_2, \nabla(-\Delta\varphi)\big)| & \leq C\|\varphi\|_{L^\infty}\|\nabla\mu_2\|\|\nabla^3\varphi\|\\
    &\leq C\|\varphi\|_{H^2}^{1/2}\|\varphi\|^{1/2}\|\nabla\mu_2\|\|\nabla^3\varphi\|\\
    &\leq \frac{m_0}{4}\|\nabla^3\varphi\|^2 + C\|\nabla\mu_2\|^2\|\varphi\|_{H^2}\|\varphi\| \\&\leq \frac{m_0}{4}\|\nabla^3\varphi\|^2 + C\|\nabla\mu_2\|^2\|\varphi\|_{H^2}^2 ,\\
     |(U_1-U_2,-\Delta\varphi)| &\leq \|U_1-U_2\| \|\Delta\varphi\| \leq \|U_1-U_2\|^2 + C\|\Delta\varphi\|^2.
\end{align*}
By combining all above estimates, we deduce that,
\begin{align*}
  \frac{1}{2}\frac{d}{dt}\|\nabla\varphi\|^2 +  \frac{m_0}{2}\|\nabla^3\varphi\|^2 &\leq \|\|\varphi_2\|_{H^2}^2\|\nabla\textbf{u}\|^2 + C(1+\|\varphi_1\|_{H^2}^2 )\|\varphi\|_{V}^2  + C( 1+\|\nabla\mu_2\|^2)\|\varphi\|_{H^2}^2 + \|U_1-U_2\|^2 
\end{align*}
By exploiting Gronwall's lemma and the estimate \eqref{cd1} we get,
\begin{align*}
     \sup_{t\in[0,T]}\|\varphi\|_{V}^2 +\int\limits_{0}^{T}\|\varphi\|_{H^3}^2 \leq C\Big( \|\textbf{u}_{01}-\textbf{u}_{02}\|^2 +\|\varphi_{01}-\varphi_{02}\|_{V}^2 +\int\limits_{0}^{T}\|U_1-U_2\|^2 \Big).
\end{align*}
Note that, the $H^3$ estimate of $\varphi$ in the above inequality is obtained by the following observation: for the elliptic estimate $\|\varphi\|_{H^2} \leq C(\|\Delta\varphi\| + \|\varphi\|)$ and by Poincar\'e $\|\Delta\varphi\| \leq C\|\nabla(\Delta\varphi)\|$. By comparison we have, $\|\nabla^3\varphi\| \leq \|\nabla(\Delta\varphi)\|$ for $d=2$, consequently, $\|\varphi\|_{H^3} \leq C(\|\nabla(\Delta\varphi)\| + \|\varphi\|)$.
\\Moreover, consider the following weak formulation obtained by testing \eqref{eq1.7} with $\chi \in V$.
\begin{align*}
    \langle\varphi', \chi\rangle + (\textbf{u}_1\cdot\nabla\varphi, \chi) +(\textbf{u}\cdot\nabla\varphi_2, \chi) &= -\big(m(\varphi_1)\nabla\Tilde{\mu}, \nabla\chi \big) - \big((m(\varphi_1)-m(\varphi_2))\nabla\mu_2, \nabla\chi \big)+ (U_1-U_2, \chi)
\end{align*}
\begin{align*}
    |\langle\varphi', \chi\rangle| &\leq C\Big[\|\textbf{u}_1\|\|\varphi\| + \|\textbf{u}\|\|\varphi_2\| +  \|\nabla\Tilde{\mu}\| + \|\varphi\|\|\nabla\mu_2\| \Big]\|\nabla\chi\| + \|U_1-U_2\|\|\chi\|
\end{align*}
Note that $\Bar{\mu} = -\Delta \varphi + F'(\varphi^h) -F'(\Bar{\varphi})$ and $\|\Bar{\mu}\| \leq C( \|\nabla^3\varphi\| + \|\varphi\|_{V} + \|\varphi_1\|_{H^2}\|\varphi\|_{V}$ by similar estimates. Further, Integrating  above inequality over $[0,T]$ and applying the uniform estimates for the strong solution, we get, 
\begin{equation*}
    \|\varphi\|_{H^1(0,T;V')} \leq C\big( \|\textbf{u}_{01}-\textbf{u}_{02}\|^2 +\|\varphi_{01}-\varphi_{02}\|_{V}^2 +\|U_1-U_2\|_{L^2(0,T;H)}^2 \big).
\end{equation*}
For further regularity, we assume $F \in C^4(\mathbb{R})$. Test \eqref{eq1.7} with $\Delta^2\varphi$, we get,
\begin{align}\label{eq1.11}
  \frac{1}{2}\frac{d}{dt}\|\Delta\varphi\|^2 + (\textbf{u}_1\cdot\nabla\varphi, \Delta^2\varphi)+ (\textbf{u}\cdot\nabla\varphi_2, \Delta^2\varphi) &= \big(\nabla\cdot(m(\varphi_1)\nabla(-\Delta\varphi+ F'(\varphi_1)-F'(\varphi_2)),  \Delta^2\varphi\big) \nonumber\\&+ \big(\nabla\cdot((m(\varphi_1)-m(\varphi_2))\nabla\mu_2),  \Delta^2\varphi\big)+ (U_1-U_2, \Delta^2\varphi)  
\end{align}
By estimating each term individually through H\"older's, Gagliardo-Nirenberg, Young's and Poincar\'e inequalities, we obtain,
\begin{align*}
     |(\textbf{u}\cdot\nabla\varphi_2, \Delta^2\varphi)| &\leq \|\textbf{u}\|_{L^4}\|\nabla\varphi_2\|_{L^4}\|\Delta^2\varphi\| \leq \|\nabla\textbf{u}\|\|\varphi_2\|_{H^2}\|\Delta^2\varphi\|\\
     &\leq \epsilon\|\Delta^2\varphi\|^2 + C\|\varphi_2\|_{H^2}^2\|\nabla\textbf{u}\|^2
\end{align*}
Observe that,
\begin{align*}
    \big(\nabla\cdot(m(\varphi_1)\nabla(-\Delta\varphi+ F'(\varphi_1)-F'(\varphi_2))), \Delta^2\varphi\big) &= \big(m'(\varphi_1)\nabla\varphi_1\nabla(-\Delta\varphi+ F'(\varphi_1)-F'(\varphi_2)), \Delta^2\varphi\big) \\&\hspace{.5cm}+ \big(m(\varphi_1)\Delta(-\Delta\varphi+ F'(\varphi_1)-F'(\varphi_2)), \Delta^2\varphi\big),
\end{align*}
Each term on the right-hand side is estimated using H\"older's, generalised Gagliardo-Nirenberg, Sobolev and Young's inequalities, as detailed below. 
\begin{align*}
    |\big(m'(\varphi_1)\nabla\varphi_1\nabla(-\Delta\varphi), \Delta^2\varphi\big)| &\leq C\|\nabla\varphi_1\|_{L^\infty}\|\nabla^3\varphi\|\|\Delta^2\varphi\| \leq C\|\varphi_1\|_{H^3}\|\nabla^3\varphi\|\|\Delta^2\varphi\|\\& \leq \epsilon\|\Delta^2\varphi\|^2 + C\|\varphi_1\|_{H^3}^2\|\nabla^3\varphi\|^2 \\
    |\big(m'(\varphi_1)\nabla\varphi_1\nabla(F'(\varphi_1)-F'(\varphi_2)), \Delta^2\varphi\big)| &=|\big(m'(\varphi_1)\nabla\varphi_1(F'''(\Hat{\varphi})\varphi\nabla\varphi_1 + F''(\varphi_2)\nabla\varphi), \Delta^2\varphi\big)| 
    \\&\leq C\|\nabla\varphi_1\|_{L^8}^2\|\varphi\|_{L^4}\|\Delta^2\varphi\| + C\|\nabla\varphi_1\|_{L^4}\|\nabla\varphi\|_{L^4}\|\Delta^2\varphi\|\\
    &\leq C\|\varphi_1\|_{H^2}^2\|\varphi\|_{V}\|\Delta^2\varphi\| + \|\varphi_1\|_{H^2}\|\varphi\|_{H^2}\|\Delta^2\varphi\|\\
     &\leq 2\epsilon\|\Delta^2\varphi\|^2 + C\|\varphi_1\|_{H^2}^4\|\varphi\|_{V}^2 + \|\varphi_1\|_{H^2}^2\|\varphi\|_{H^2}^2 
\end{align*}
The second last inequality is obtained by applying Sobolev embeddings. 
Note that,  
\begin{equation*}
|\big(m(\varphi_1)\Delta^2\varphi, \Delta^2\varphi\big)| \geq m_0\|\Delta^2\varphi\|^2
\end{equation*}
and, 
\begin{align*}
\big(m(\varphi_1)\Delta(F'(\varphi_1)&-F'(\varphi_2)), \Delta^2\varphi\big)| \leq |\big(m(\varphi_1)F^{(4)}(\Hat{\varphi})\varphi(\nabla\varphi_1)^2, \Delta^2\varphi\big)|+  |\big(m(\varphi_1)F'''(\varphi_2)\nabla\varphi_1\nabla\varphi, \Delta^2\varphi\big)| \\&\hspace{.5cm}+ |\big(m(\varphi_1)F'''(\Tilde{\varphi})\varphi\Delta\varphi_1, \Delta^2\varphi\big)|+ |\big(m(\varphi_1)F'''(\varphi_2)\nabla\varphi_2\nabla\varphi, \Delta^2\varphi\big)| + |\big(m(\varphi_1)F''(\varphi_2)\Delta\varphi, \Delta^2\varphi\big)|
\end{align*}
Estimating terms on right hand side,
\begin{align*} 
    |\big(m(\varphi_1)F^{(4)}(\Hat{\varphi})(\nabla\varphi_1)^2\varphi, \Delta^2\varphi\big)| &\leq C\|\nabla\varphi_1\|_{L^8}^2\|\varphi\|_{L^4}\|\Delta^2\varphi\| 
    \leq C\|\varphi_1\|_{H^2}^{2}\|\varphi\|_{V}\|\Delta^2\varphi\| \\&\leq  \epsilon\|\Delta^2\varphi\|^2 + C\|\varphi_1\|_{H^2}^{4}\|\varphi\|_{V}^2 \\
    |\big(m(\varphi_1)F'''(\varphi_2)\nabla\varphi_1\nabla\varphi, \Delta^2\varphi\big)| &\leq C\|\nabla\varphi_1\|_{L^4}\|\nabla\varphi\|_{L^4}\|\Delta^2\varphi\| \leq C\|\varphi_1\|_{H^2}\|\varphi\|_{H^2}\|\Delta^2\varphi\|\\
    &\leq \epsilon\|\Delta^2\varphi\|^2 + C\|\varphi_1\|_{H^2}^2\|\varphi\|_{H^2}^2 \\
    |\big(m(\varphi_1)F'''(\Tilde{\varphi})\varphi\Delta\varphi_1, \Delta^2\varphi\big)| &\leq \|\varphi\|_{L^\infty}\|\Delta\varphi_1\|\|\Delta^2\varphi\| \\
    &\leq \epsilon\|\Delta^2\varphi\|^2 + \|\Delta\varphi_1\|^2\|\varphi\|_{H^2}^2\\
    |\big(m(\varphi_1)F'''(\varphi_2)\nabla\varphi_2\nabla\varphi, \Delta^2\varphi\big)| &\leq C\|\nabla\varphi_2\|_{L^4}\|\nabla\varphi\|_{L^4}\|\Delta^2\varphi\| \leq C\|\varphi_2\|_{H^2}\|\varphi\|_{H^2}\|\Delta^2\varphi\|\\
    &\leq \epsilon\|\Delta^2\varphi\|^2 + C\|\varphi_2\|_{H^2}^2\|\varphi\|_{H^2}^2\\
    |\big(m(\varphi_1)F''(\varphi_2)\Delta\varphi, \Delta^2\varphi\big)| &\leq C\|\Delta\varphi\|\| \Delta^2\varphi\| \leq \epsilon\|\Delta^2\varphi\|^2+ C\|\varphi\|_{H^2}^2
\end{align*}
We have used H\"olders, Gagliardo-Nirenberg, Young's, sobolev inequalities and \eqref{F1}, \eqref{eqF4}, in above calculations.
Furthermore, the last term of \eqref{eq1.11} can be estimated as follows.
\begin{align*}
   |\big(\nabla\cdot((m(\varphi_1)-m(\varphi_2))\nabla\mu_2), \Delta^2\varphi\big)| 
   &= |\big(\nabla(m(\varphi_1)-m(\varphi_2))\nabla\mu_2 + (m(\varphi_1)-m(\varphi_2))\Delta\mu_2 , \Delta^2\varphi\big)| \\&\leq  |\big(m''(\Hat{\varphi})\varphi\nabla\varphi_1\nabla\mu_2, \Delta^2\varphi\big)|+ |\big(m'(\varphi_2)\nabla\varphi\nabla\mu_2, \Delta^2\varphi\big)|+ |\big(m'(\Hat{\varphi})\varphi\Delta\mu_2, \Delta^2\varphi\big)|
\end{align*}
\begin{align*}
    |\big(m''(\Hat{\varphi})\varphi\nabla\varphi_1\nabla\mu_2, \Delta^2\varphi\big)| &\leq \|\nabla\varphi_1\|_{L^8}\|\varphi\|_{L^8}\|\nabla\mu_2\|_{L^4}\|\Delta^2\varphi\| \leq C\|\varphi_1\|_{H^2}\|\varphi\|_{V}\|\mu_2\|_{H^2}\|\Delta^2\varphi\|\\
    &\leq \epsilon\|\Delta^2\varphi\|^2 + C\|\varphi_1\|_{H^2}^2\|\mu_2\|_{H^2}^2\|\varphi\|_{V}^2\\
    |\big(m'(\varphi_2)\nabla\varphi\nabla\mu_2, \Delta^2\varphi\big)| &\leq C\|\nabla\varphi\|_{L^4}\|\nabla\mu_2\|_{L^4}\|\Delta^2\varphi\| \leq C\|\varphi\|_{H^2}\|\mu_2\|_{H^2}\|\Delta^2\varphi\|\\
     &\leq \epsilon\|\Delta^2\varphi\|^2 + \|\mu_2\|_{H^2}^2\|\varphi\|_{H^2}^2
\end{align*}
\begin{align*}
     |\big(m'(\Hat{\varphi})\varphi\Delta\mu_2, \Delta^2\varphi\big)| &\leq \|\varphi\|_{L^\infty}\|\Delta\mu_2\|\| \Delta^2\varphi\| \leq C\|\varphi\|_{H^2}\|\mu_2\|_{H^2}\|\Delta^2\varphi\|\\ &\leq \epsilon\|\Delta^2\varphi\|^2 + C\|\mu_2\|_{H^2}^2\|\varphi\|_{H^2}^2
    \\ |(U_1-U_2, \Delta^2\varphi)| &\leq \|U_1-U_2\|\|\Delta^2\varphi\| \leq \epsilon\|\Delta^2\varphi\|^2 + C\|U_1-U_2\|^2
\end{align*}
By combining all the above estimates in \eqref{eq1.11}, and choosing $\epsilon= \frac{m_0}{28}$, we get,
\begin{align*}
  \frac{1}{2}\frac{d}{dt}\|\Delta\varphi\|^2 +  \frac{m_0}{2}\|\Delta^2\varphi\|^2 &\leq \|\varphi_2\|_{H^2}^2\|\nabla\textbf{u}\|^2 + C\|U_1-U_2\|^2 + C\big( 1+ \|\varphi_1\|_{H^2}^4 + \|\varphi_1\|_{H^2}^2\|\mu_2\|_{H^2}^2 \big)\|\varphi\|_{V}^2 \\ &\hspace{.5cm}+ C\big( 1+ \|\varphi_1\|_{H^2}^2 + \|\varphi_2\|_{H^2}^2 + \|\mu_2\|_{H^2}^2 \big)\|\varphi\|_{H^2}^2  
\end{align*}
Note that the term $\|\Delta\varphi\|^2$ on the R.H.S. is obtained from $\|\varphi\|_{H^2(\Omega)}^2$, by applying the elliptic estimate $\|\varphi\|_{H^2}^2 \leq C(\|\Delta\varphi\|^2 + \|\varphi\|^2)$.
By exploiting the Gronwall's lemma and uniform estimates derived for weak and strong solutions, we get,
\begin{align*}
     \sup_{t\in[0,T]}\|\varphi_1-\varphi_2\|_{H^2}^2 +\|\varphi_1-\varphi_2\|_{L^2(0,T;H^4)}^2 \leq C\Big( \|\textbf{u}_{01}-\textbf{u}_{02}\|^2 +\|\varphi_{01}-\varphi_{02}\|_{H^2}^2 + \|U_1-U_2\|_{L^2(0,T;H)}^2 \Big).
\end{align*}
\end{proof}
We will see in the following sections that a higher regular stability estimate is useful, particularly to study the differentiability properties of the control to state operator.

\section{Pointwise Tracking Optimal Control}
 In this section our aim is to study optimal control problem for \eqref{eq8} -\eqref{eq14} by minimizing the cost functional  $\mathcal{J}_1$, given by,
\begin{align}\label{eqJ1}
    \mathcal{J}_1(\varphi,\textbf{u}, U) &= \frac{1}{2}\sum\limits_{i=1}^{k}\int\limits_{0}^{T}(\varphi(x_i,t)-\Phi^i(t))^2 dt +\frac{1}{2}\int\limits_{0}^{T}\int_\Omega(\Bar{\textbf{u}}-\textbf{u}_d)^2 dxdt 
    + \frac{1}{2}\int_\Omega(\Bar{\textbf{u}}(x,T)-\textbf{u}_d(x,T))^2 dx + \frac{1}{2}\int\limits_{0}^{T}\int_\Omega|U|^2 dxdt 
\end{align}
over an appropriate control set. Let us  restrict the control set to be bounded. For we define an admissible control set:
\begin{equation}
\mathcal{U}_{ad}:= \{ U\in L^2(0,T;V\cap L^\infty(\Omega)): U_1(x,t) \leq U(x,t) \leq U_2(x,t), \text{ for a.e. } (x,t)\in \Omega\times(0,T)\},
\end{equation}
where  the functions $U_1$, $U_2$ are fixed and are almost everywhere bounded.
The optimal control problem we wish to study is 
$$\min_{U\in \mathcal{U}_{ad}}\{ \mathcal{J}_1(\varphi,\textbf{u},U) | \; (\varphi, \textbf{u}) \text{ is the strong solution of \eqref{eq8} -\eqref{eq14} with control U } \}\hspace{1cm} \textbf{[OCP\RomanNumeralCaps{1}]}$$  

 Note that $\{ \Phi^i(t)\}_{1\leq i\leq k}$ denotes the desired concentration at prescribed points $\{ x_i\}_{i=1}^{k}$ and $\textbf{u}_d$ is the desired velocity field. We will consider $\Phi^i(t)= \varphi^i, 1\leq i\leq k$; a constant trajectory  with $\varphi^i \in \mathbb{R}$. By Proposition \ref{prop2}, proposition \ref{prop3} there exist a unique strong solution and hence $\varphi \in C([0,T];C^{0,\gamma}(\Bar{\Omega}))$. Thus $\mathcal{J}_1$  is well defined.
 Furthermore, we define spaces,
\begin{align*}
 \mathcal{U} &= L^2(0,T; V \cap L^\infty(\Omega)), \\
 \mathcal{V} &= [L^\infty(0,T;H^2)\cap L^2(0,T;H^4) \cap H^1(0,T;V')] \times [L^\infty(0,T;\mathbb{G}_{div}) \cap L^2(0,T;\mathbb{V}_{div}) \cap H^1(0,T;\mathbb{V}_{div}')], \\
 \mathcal{Z} &= [L^\infty(0,T;H)\cap L^2(0,T;H^2)] \times [L^\infty(0,T;\mathbb{G}_{div}) \cap L^2(0,T;\mathbb{V}_{div}) \cap H^1(0,T;\mathbb{V}_{div}')].
\end{align*}
We define a map $\mathcal{S}$ from $\mathcal{U}$ to the solution space, $\mathcal{V}$ by $$U \mapsto (\varphi_{U}, \textbf{u}_{U}),$$ where $(\varphi_{U}, \textbf{u}_{U})$ denotes  the unique strong solution of the system \eqref{eq8}-\eqref{eq14}. By proposition \ref{prop2} and proposition \ref{prop3} the above map is well-defined. Using this notation we can rewrite the problem \textbf{[OCP\RomanNumeralCaps{1}]}
as follows.
$$\min_{U\in \mathcal{U}_{ad}}\{ \mathcal{J}_1(\mathcal{S}(U),U) \}$$

In the following subsection, we establish the existence of an optimal control. The proof is done by considering a minimising sequence and using the properties of objective functional, $\mathcal{J}_1$ and the admissible control set, $\mathcal{U}_{ad}$.

\subsection{Existence of an optimal control}
\begin{theorem}[Existence of an optimal control]\label{theorem1}
  Let $\varphi_0 \in H^3(\Omega)$, $\textbf{u}_0 \in \mathbb{V}_{div}$ and assumptions \emph{\textbf{[A1]}- \textbf{[A5]}} be satisfied. In addition, $F$ satisfies \eqref{F5}. Then the problem, \emph{\textbf{[OCP\RomanNumeralCaps{1}]}} admits a solution.  
\end{theorem}
\begin{proof}
Since $\mathcal{J}_1$ is bounded below, we define $$j= \inf_{U\in \mathcal{U}_{ad}} \mathcal{J}_1(\varphi, \textbf{u}, U).$$ 
Let $(\varphi_n, \textbf{u}_n, U_n)$ be a minimising sequence of $\mathcal{J}_1$. i.e. $j = \underset{n\rightarrow \infty}{\lim} \mathcal{J}_1(\varphi_n, \textbf{u}_n, U_n)$. Since $\mathcal{U}_{ad}$ is a closed, bounded, convex set; there exist a relabelled sequence $U_n \rightharpoonup \Bar{U}$ in $L^2(0,T; V)$ and $\Bar{U}\in \mathcal{U}_{ad}$. Using the uniform boundedness of $U_n$ and the a priori estimates enable us to extract a weakly convergent subsequence of $(\varphi_n, \textbf{u}_n)$. Let $(\Bar{\varphi}, \Bar{u})$ be the weak limit of $(\varphi_n, \textbf{u}_n)$. Then using the continuous dependence estimate, we get $(\Bar{\varphi}, \Bar{u})$ solves the system \eqref{eq8}-\eqref{eq14} corresponding to the source term $\Bar{U}$.
The only thing left to prove is that $j$ is the minimum value of $\mathcal{J}_1$. Consider the cost functional,
\begin{align}
    \mathcal{J}_1(\varphi,\textbf{u}, U) &= \underbrace{\frac{1}{2}\int\limits_{0}^{T}\sum\limits_{i=1}^{k}(\varphi(x_i,t)-\varphi^i(t))^2}_{\mathcal{J}_{a}} +\underbrace{\frac{1}{2}\int\limits_{0}^{T}\int_\Omega(\Bar{\textbf{u}}-\textbf{u}_d)^2 +\frac{1}{2}\int_\Omega(\Bar{\textbf{u}}(x,T)-\textbf{u}_d(x,T))^2
    + \frac{1}{2}\int\limits_{0}^{T}\|U\|^2}_{\mathcal{J}_{b}}
\end{align}
Note that, $\mathcal{J}_a$ denotes the first term and $\mathcal{J}_b$ denotes the sum of remaining terms on R.H.S. of the above equation. Using the weak convergence of $\varphi_n$ in $L^2(0,T;H^2(\Omega))$ and using the embedding $L^2(0,T;H^2(\Omega)) \hookrightarrow L^2(0,T;C(\Bar{\Omega}))$, we have,  $\underset{k\rightarrow \infty}{\lim} \mathcal{J}_a(\varphi_k, \textbf{u}_k, U_k) = \mathcal{J}_a(\Bar{\varphi}, \Bar{\textbf{u}}, \Bar{U})$.
Further by weak sequential lower semi-continuity of $\mathcal{J}_b$, $\mathcal{J}_b(\Bar{\varphi}, \Bar{\textbf{u}}, \Bar{U}) \leq \underset{k\rightarrow \infty}{\lim} \mathcal{J}_b(\varphi_k, \textbf{u}_k, U_k)$.
Thus, 
\begin{equation*}
    j\leq \mathcal{J}_1(\Bar{\varphi}, \Bar{\textbf{u}}, \Bar{U}) \leq \lim_{k\rightarrow \infty} \mathcal{J}_1(\varphi_k, \textbf{u}_k, U_k) = j
\end{equation*}
Hence the existence of an optimal control.
\end{proof}

We call $\bar{U}$ as an optimal control and $(\bar {\varphi}, \bar{\textbf{u} })$
as the corresponding optimal state.

\subsection{Differentiability of the control to state operator}

Our analysis in this section entails studying the differentiability of the control to state operator $\mathcal{S}$, which is essential for deriving the optimality conditions. First we will introduce the linearization of the CHNS system \eqref{eq8}-\eqref{eq14} around the state $(\Bar{\varphi}, \Bar{\textbf{u}})$, which is a strong solution corresponding to the optimal control $\Bar{U}$. Later, we prove that the Fr\'etchet derivative of $\mathcal{S}$ at  $\Bar{U}$ is the solution of the linearization of \eqref{eq8}-\eqref{eq14} around $\mathcal{S}(\Bar{U})= (\Bar \varphi, \Bar {\textbf{u}}) $. 
\\
Consider the linearization of \eqref{eq8}-\eqref{eq14} around  $(\Bar{\varphi}, \Bar{\textbf{u}})$ to obtain a  system  in $(\psi, \textbf{w})$ given by
\begin{align}
    \psi' + \Bar{\textbf{u}}\cdot\nabla\psi + \textbf{w}\cdot\nabla\Bar{\varphi} &= \nabla\cdot\big(m(\Bar{\varphi})\nabla(-\Delta \psi + F''(\Bar{\varphi})\psi)\big)+ \nabla\cdot\big(m'(\Bar{\varphi})\psi\nabla(-\Delta \Bar{\varphi}+ F'(\Bar{\varphi}))\big) + h, &\text{ in } Q, \label{eq2.3}\\
    \textbf{w}'+ (\Bar{\textbf{u}}\cdot\nabla)\textbf{w} + (\textbf{w}\cdot\nabla)\Bar{\textbf{u}} & - \nabla\cdot(2\eta(\Bar{\varphi})D\textbf{w})- \nabla\cdot(2\eta'(\Bar{\varphi})\psi D\Bar{\textbf{u}}) + \nabla p^\ast = (-\Delta\Bar{\varphi})\nabla\psi + (-\Delta\psi)\nabla\Bar{\varphi}, &\text{ in } Q,\label{eq2.4}\\
    \nabla\cdot \textbf{w} &= 0, &\text{ in } Q,\label{eq2.5}\\
    \frac{\partial\psi}{\partial \textbf{n}} &= 0, \frac{\partial(\Delta\psi)}{\partial \textbf{n}} = 0, &\text{ on } \Sigma,\label{eq2.6}\\ \textbf{w} &= \textbf{0}, &\text{ on } \Sigma, \label{eq2.7}\\
    \psi(x,0) &= \psi_0(x), \hspace{.5cm}\textbf{w}(x,0) = \textbf{w}_0(x). &\text{ in } \Omega,\label{eq2.8}
\end{align}
where, $ h = U-\bar{U}$.
\begin{theorem}[\textit{Existence of weak solution to the linearized system}]\label{theorem2}
    Let $\psi_0\in V$, $\textbf{w}_0\in \mathbb{G}_{div}$ and $ h\in L^2(V\cap L^\infty(\Omega))$. Assume \emph{\textbf{[A1]}}-\emph{\textbf{[A5]}} are satisfied. Then there exists a weak solution $(\psi, \textbf{w})$ of the system \eqref{eq2.3}-\eqref{eq2.8} of following regularity.
    \begin{align*}
    \psi &\in L^\infty(0,T;V) \cap L^2(0,T; H^3(\Omega)) \cap H^1(0,T;V'),\\
    \textbf{w} &\in L^\infty(0,T;\mathbb{G}_{div}) \cap L^2(0,T; \mathbb{V}_{div})\cap H^1(0,T; \mathbb{V}_{div}').
    \end{align*}
\end{theorem}
\begin{proof}
We prove the existence of a weak solution using the Galerkin approximation method. Let $\{ \zeta_i\}\subseteq H^2_{N}$ denote the eigenvectors of the Neumann-Laplacian operator; these form an orthonormal basis of $H$ and an orthogonal basis of $H_N^2$ where $H^2_N$ denotes the functions in $H^2$ satisfying Neumann boundary conditions. Similarly, the set of eigenvectors of the Stokes operator denoted by $\{\bm{\nu}_i\} \subseteq \textbf{H}^2\cap\mathbb{V}_{div}$ forms an orthonormal basis of $\mathbb{G}_{div}$ as well as an orthogonal basis of $\mathbb{V}_{div}$. Define $\Psi_n = \mathrm{ span}\{ \zeta_1, \zeta_2,..., \zeta_n \}$ and  $\mathcal{V}_n = \mathrm{ span}\{ \bm{\nu}_1, \bm{\nu}_2,..., \bm{\nu}_n \}$. Further, $P_n$, $\mathbb{P}_n$ denote the projection of $H$ and $\mathbb{G}_{div}$ onto the finite dimensional spaces $\Psi_n$ and $\mathcal{V}_n$ respectively. Now we define, 
\begin{align*}
    \psi_n = \sum\limits_{i=1}^{n}a_i^n(t)\zeta_i,\hspace{.5cm} \textbf{w}_n = \sum\limits_{i=1}^{n}b_i^n(t)\bm{\nu}_i
\end{align*}
Consider the  weak formulation for the finite dimensional problem corresponding to system \eqref{eq2.3}-\eqref{eq2.8}. For each fixed $t$ we multiply $\psi_n $  by $ \xi_i$ and $\textbf{w}_n $ by $\bm{\nu}_i$ to get
\begin{align}
    \langle\psi_n'(t),\zeta_i\rangle + (\Bar{\textbf{u}}\cdot\nabla\psi_n,\zeta_i)  + (\textbf{w}_n\cdot\nabla\Bar{\varphi},\zeta_i) + \big(m(\Bar{\varphi})\nabla(-\Delta \psi_n + F''(\Bar{\varphi})\psi_n), \nabla\zeta_i\big)& \nonumber\\+ \big(m'(\Bar{\varphi})\psi_n\nabla(-\Delta \Bar{\varphi}+ F'(\Bar{\varphi})), \nabla\zeta_i\big) &= (h(t),\zeta_i),\label{eq2.9}\\
    \langle\textbf{w}_n'(t), \bm{\nu}_i\rangle+ ((\Bar{\textbf{u}}\cdot\nabla)\textbf{w}_n,\bm{\nu}_i) + ((\textbf{w}_n\cdot\nabla)\Bar{\textbf{u}},\bm{\nu}_i) + (2\eta(\Bar{\varphi})D\textbf{w}_n, D\bm{\nu}_i)+ (2\eta'(\Bar{\varphi})\psi_n D\Bar{\textbf{u}}, \bm{\nu}_i)&\nonumber \\+ ((\Delta\Bar{\varphi})\nabla\psi_n, \bm{\nu}_i) + ((\Delta\psi_n)\nabla\Bar{\varphi}, \bm{\nu}_i) &=0,\label{eq2.10}\\
    \psi_n(x,0) &= P_n(\psi_0(x)),\label{eq2.11} \\
    \textbf{w}(x,0) &= \mathbb{P}_n(\textbf{w}_0(x)).\label{eq2.12}
\end{align}
for all $1\leq i\leq n$. The above system reduces to a system of ODE in variables $a_i^n(t)$ and $b_i^n(t)$. Using the Lipschitz continuity of $m, F, \eta$ and regularity of $(\bar{\varphi}, \bar{\textbf{u}})$ in the time variable, we can apply Cauchy Picard's theorem, to prove that this system admits a solution. Consequently, we obtain a solution, $(\psi_n, \textbf{w}_n)$ of the finite dimensional problem on $[0,t]$, for some $t\leq T$. To extend this solution to the entire interval $[0,T]$ we derive a priori estimates for the solution of finite dimensional problem \eqref{eq2.9}-\eqref{eq2.12}.
\\We multiply \eqref{eq2.9} by $a_i^n(t)$ and take the summation over $1\leq i\leq n$ to obtain,
\begin{align}\label{eq213}
    \frac{1}{2}\frac{d}{dt}\|\psi_n\|^2 + (\textbf{w}_n\cdot\nabla\Bar{\varphi},\psi_n) + \big(m(\Bar{\varphi})\nabla(-\Delta \psi_n + F''(\Bar{\varphi})\psi_n), \nabla\psi_n\big)& \nonumber\\+ \big(m'(\Bar{\varphi})\psi_n\nabla(-\Delta \Bar{\varphi}+ F'(\Bar{\varphi})), \nabla\psi_n\big) &= ( h,\psi_n).
\end{align}
Estimating each term  by applying H\"olders, Ladyzhenskaya, Gagliardo Nirenberg and Young's inequalities, we get,
\begin{align*}
    |(\textbf{w}_n\cdot\nabla\Bar{\varphi},\psi_n)| &\leq \|\textbf{w}_n\|_{L^4}\|\nabla\Bar{\varphi}\|_{L^4}\|\psi_n\|\\ 
    &\leq \|\nabla\textbf{w}_n\|\|\nabla\Bar{\varphi}\|^{1/2}\|\Bar{\varphi}\|_{H^2}^{1/2}\|\psi_n\|\\
    &\leq \epsilon\|\nabla\textbf{w}_n\|^2 + C\|\nabla\Bar{\varphi}\|\|\Bar{\varphi}\|_{H^2}\|\psi_n\|^2
\end{align*}
Further, estimating,
\begin{align*}
   \big(m(\Bar{\varphi})\nabla(-\Delta \psi_n + F''(\Bar{\varphi})\psi_n), \nabla\psi_n\big) &= -\big(m(\Bar{\varphi})(-\Delta \psi_n + F''(\Bar{\varphi})\psi_n), \Delta\psi_n\big)-\big(m'(\Bar{\varphi})\nabla\Bar{\varphi}(-\Delta \psi_n + F''(\Bar{\varphi})\psi_n), \nabla\psi_n\big),
\end{align*}
using \textbf{[A1]} and standard inequalities,  we obtain,
\begin{align*}
   \big(m(\Bar{\varphi})\Delta \psi_n, \Delta\psi_n\big) &\geq m_0\|\Delta\psi_n\|^2,\\
   |\big(m(\Bar{\varphi})F''(\Bar{\varphi})\psi_n, \Delta\psi_n\big)| &\leq m_1||F''(\Bar{\varphi})||_{L^\infty}\|\psi_n\|\|\Delta\psi_n\|\\
   &\leq m_1C_0(1+\|\Bar{\varphi}\|_{L^\infty}^r)\|\psi_n\|\|\Delta\psi_n\| \\
   &\leq \epsilon\|\Delta\psi_n\|^2 + C(1+\|\Bar{\varphi}\|_{L^\infty}^{2r})\|\psi_n\|^2,\\
   |\big(m'(\Bar{\varphi})\nabla\Bar{\varphi}(-\Delta \psi_n), \nabla\psi_n\big)| &\leq \|m'\|_{L^\infty}\|\nabla\Bar{\varphi}\|_{L^4}\|\Delta\psi_n\|\|\nabla\psi_n\|_{L^4}\\
   &\leq C \|\nabla\Bar{\varphi}\|^{1/2}\|\Bar{\varphi}\|_{H^2}^{1/2}\|\Delta\psi_n\|\|\psi_n\|_{H^2}^{3/4}\|\psi_n\|^{1/4} \\
   &\leq \epsilon\|\Delta\psi_n\|^2 + C \|\nabla\Bar{\varphi}\|\|\Bar{\varphi}\|_{H^2}\|\psi_n\|_{H^2}^{3/2}\|\psi_n\|^{1/2}\\
   &\leq 2\epsilon\|\psi_n\|_{H^2}^2 + C \|\nabla\Bar{\varphi}\|^2\|\Bar{\varphi}\|_{H^2}^2\|\psi_n\|_{H^2}\|\psi_n\|\\
   &\leq 3\epsilon\|\psi_n\|_{H^2}^2 + C \|\nabla\Bar{\varphi}\|^4\|\Bar{\varphi}\|_{H^2}^4\|\psi_n\|^2.
\end{align*}
Moreover, by exploiting \eqref{F1}, \eqref{eqF4} it follows that, 
\begin{align*}
   |\big(m'(\Bar{\varphi})\nabla\Bar{\varphi}F''(\Bar{\varphi})\psi_n, \nabla\psi_n\big)| &\leq \|m'\|_{L^\infty}\|\nabla\Bar{\varphi}\|_{L^4}\|F''(\Bar{\varphi})\|_{L^\infty}\|\psi_n\|\|\nabla\psi_n\|_{L^4}\\
   &\leq C\|\nabla\Bar{\varphi}\|^{1/2}\|\Bar{\varphi}\|_{H^2}^{1/2}(1+\|\Bar{\varphi}\|_{L^\infty}^r)\|\psi_n\|\|\psi_n\|_{H^2}^{3/4}\|\psi_n\|^{1/4}  \\
   &\leq C_1\|\nabla\Bar{\varphi}\|\|\Bar{\varphi}\|_{H^2}(1+\|\Bar{\varphi}\|_{L^\infty}^{2r})\|\psi_n\|^2 + C_2\|\psi_n\|_{H^2}^{3/2}\|\psi_n\|^{1/2}  \\
   &\leq \epsilon\|\psi_n\|_{H^2}^2 + C_1\|\nabla\Bar{\varphi}\|\|\Bar{\varphi}\|_{H^2}(1+\|\Bar{\varphi}\|_{L^\infty}^{2r})\|\psi_n\|^2 + C_2\|\psi_n\|_{H^2}\|\psi_n\|  \\
   &\leq 2\epsilon\|\psi_n\|_{H^2}^2 + C_1\|\nabla\Bar{\varphi}\|\|\Bar{\varphi}\|_{H^2}(1+\|\Bar{\varphi}\|_{L^\infty}^{2r})\|\psi_n\|^2 + C_2\|\psi_n\|^2,
\end{align*}
Indeed the last term on R.H.S. of \eqref{eq213} can be estimated as follows,
\begin{align*}
    |\big(m'(\Bar{\varphi})\psi_n\nabla(-\Delta \Bar{\varphi})), \nabla\psi_n\big)| &\leq \|m'\|_{L^\infty}\|\psi_n\|_{L^4}\|\nabla^3\Bar{\varphi}\|\|\nabla\psi_n\|_{L^4} \\
    &\leq C\|\nabla^3\Bar{\varphi}\|\|\psi_n\|_{H^2}^{1/4}\|\psi_n\|^{3/4}\|\psi_n\|_{H^2}^{3/4}\|\psi_n\|^{1/4}\\
    &\leq C\|\nabla^3\Bar{\varphi}\|\|\psi_n\|_{H^2}\|\psi_n\|\\
    &\leq \epsilon \|\psi_n\|_{H^2}^2 + C\|\Bar{\varphi}\|_{H^3}^2\|\psi_n\|^2,\\
   |\big(m'(\Bar{\varphi})\psi_n F''(\Bar{\varphi})\nabla\Bar{\varphi}, \nabla\psi_n\big)| &\leq \|m'\|_{L^\infty}\|\psi_n\|\|F''(\Bar{\varphi})\|_{L^\infty}\|\nabla\Bar{\varphi}\|_{L^4}\|\nabla\psi_n\|_{L^4}\\
   &\leq C(1+\|\Bar{\varphi}\|_{L^\infty}^{r})\|\nabla\Bar{\varphi}\|^{1/2}\|\Bar{\varphi}\|_{H^2}^{1/2}\|\psi_n\|\|\psi_n\|_{H^2}^{3/4}\|\psi_n\|^{1/4}\\
   &\leq  C(1+\|\Bar{\varphi}\|_{L^\infty}^{2r})\|\nabla\Bar{\varphi}\|\|\Bar{\varphi}\|_{H^2}\|\psi_n\|^2 + C_1\|\psi_n\|_{H^2}^{3/2}\|\psi_n\|^{1/2}\\
   &\leq \epsilon\|\psi_n\|_{H^2}^2 + C(1+\|\Bar{\varphi}\|_{L^\infty}^{2r})\|\nabla\Bar{\varphi}\|\|\Bar{\varphi}\|_{H^2}\|\psi_n\|^2 + C_1\|\psi_n\|_{H^2}\|\psi_n\|\\
   &\leq 2\epsilon\|\psi_n\|_{H^2}^2 + C(1+\|\Bar{\varphi}\|_{L^\infty}^{2r})\|\nabla\Bar{\varphi}\|\|\Bar{\varphi}\|_{H^2}\|\psi_n\|^2 + C_1\|\psi_n\|^2,\\
   |(h, \psi_n)| &\leq \|h\|\|\psi_n\| \leq C(\|h\|^2+ \|\psi_n\|^2).
\end{align*}
Combining all above estimates we deduce,
\begin{align}
     \frac{1}{2}\frac{d}{dt}\|\psi_n\|^2 &+ m_0\|\Delta\psi_n\|^2 \leq \epsilon\|\nabla\textbf{w}_n\|^2 + 9\epsilon\|\psi_n\|_{H^2}^2 + C\|h\|^2\nonumber \\
     &+ C\|\psi_n\|^2\Big( 1+\|\Bar{\varphi}\|_{L^\infty}^{2r} + (1+\|\Bar{\varphi}\|_{L^\infty}^{2r})\|\nabla\Bar{\varphi}\|\|\Bar{\varphi}\|_{H^2} + \|\Bar{\varphi}\|_{H^3}^2+ \|\nabla\Bar{\varphi}\|^4\|\Bar{\varphi}\|_{H^2}^4   \Big).\label{eq2.15}
\end{align}
\\Similarly, multiply \eqref{eq2.10} by $b_i^n(t)$ and take the summation over $1\leq i\leq n$. We obtain,
\begin{align}
     \frac{1}{2}\frac{d}{dt}\|\textbf{w}_n\|^2 + 2\|\sqrt{\eta(\Bar{\varphi})}D\textbf{w}_n\|^2 + ((\textbf{w}_n\cdot\nabla)\Bar{\textbf{u}},\textbf{w}_n) + 2(\eta'(\Bar{\varphi})\psi_n D\Bar{\textbf{u}}, \textbf{w}_n) &\nonumber \\+ ((\Delta\Bar{\varphi})\nabla\psi_n, \textbf{w}_n) + ((\Delta\psi_n)\nabla\Bar{\varphi}, \textbf{w}_n) &=0,
\end{align}
By applying H\"olders, Gagliardo-Nirenberg, Poincar\'e, Ladyzhenskaya, Young's inequalities we will estimate each term of the above inequality.
\begin{align*}
    |((\textbf{w}_n\cdot\nabla)\Bar{\textbf{u}},\textbf{w}_n)| &\leq \|\nabla\Bar{\textbf{u}}\|\|\textbf{w}_n\|_{L^4}^2 \leq \|\nabla\Bar{\textbf{u}}\|\|\textbf{w}_n\|\|\nabla\textbf{w}_n\|\\
    &\leq \epsilon\|\nabla\textbf{w}_n\|^2 + C\|\nabla\Bar{\textbf{u}}\|^2\|\textbf{w}_n\|^2,\\
    |(\eta'(\Bar{\varphi})\psi_n D\Bar{\textbf{u}}, \textbf{w}_n)| &\leq \|\eta'\|_{L^\infty}\|\nabla\Bar{\textbf{u}}\|\|\psi_n\|_{L^4}  \|\textbf{w}_n\|_{L^4}\\
    &\leq C\|\nabla\Bar{\textbf{u}}\|\|\psi_n\|_{H^2}^{1/4}\|\psi_n\|^{3/4} \|\nabla\textbf{w}_n\|\\ 
    &\leq \epsilon\|\nabla\textbf{w}_n\|^2 + C\|\nabla\Bar{\textbf{u}}\|^2\|\psi_n\|_{H^2}^{1/2}\|\psi_n\|^{3/2}\\
    &\leq \epsilon\|\nabla\textbf{w}_n\|^2 + C\|\nabla\Bar{\textbf{u}}\|^4\|\psi_n\|^2+ \|\psi_n\|_{H^2}\|\psi_n\|\\
    &\leq \epsilon\|\nabla\textbf{w}_n\|^2 + \epsilon\|\psi_n\|_{H^2}^2 + C(1+\|\nabla\Bar{\textbf{u}}\|^4)\|\psi_n\|^2, \\
    |((\Delta\Bar{\varphi})\nabla\psi_n, \textbf{w}_n)| &\leq \|\Delta\Bar{\varphi}\|\|\nabla\psi_n\|_{L^4}\| \textbf{w}_n\|_{L^4}\\
    &\leq \|\Delta\Bar{\varphi}\|\|\psi_n\|_{H^2}^{3/4}\|\psi_n\|^{1/4}\|\nabla\textbf{w}_n\|\\
     &\leq \epsilon\|\nabla\textbf{w}_n\|^2+ C\|\Delta\Bar{\varphi}\|^2\|\psi_n\|_{H^2}^{3/2}\|\psi_n\|^{1/2}\\
     &\leq \epsilon\|\nabla\textbf{w}_n\|^2+\epsilon\|\psi_n\|_{H^2}^2 + C\|\Delta\Bar{\varphi}\|^4\|\psi_n\|_{H^2}\|\psi_n\|\\
     &\leq \epsilon\|\nabla\textbf{w}_n\|^2+ 2\epsilon\|\psi_n\|_{H^2}^2 + C\|\Delta\Bar{\varphi}\|^8\|\psi_n\|^2,\\
     |((\Delta\psi_n)\nabla\Bar{\varphi}, \textbf{w}_n)| &\leq \|\Delta\psi_n\|\|\nabla\Bar{\varphi}\|_{L^4}\| \textbf{w}_n\|_{L^4}\\
     &\leq \|\Delta\psi_n\|\|\nabla\Bar{\varphi}\|^{1/2}\|\Bar{\varphi}\|_{H^2}^{1/2}\|\textbf{w}_n\|^{1/2}\|\nabla\textbf{w}_n\|^{1/2}\\
     &\leq \epsilon\|\Delta\psi_n\|^2+ C\|\nabla\Bar{\varphi}\|\|\Bar{\varphi}\|_{H^2}\|\textbf{w}_n\|\|\nabla\textbf{w}_n\|\\
     &\leq \epsilon\|\nabla\textbf{w}_n\|^2+ \epsilon\|\psi_n\|_{H^2}^2+ C\|\nabla\Bar{\varphi}\|^2\|\Bar{\varphi}\|_{H^2}^2\|\textbf{w}_n\|^2.\\
\end{align*}
From \textbf{[A1]} and the preceding estimates, it follows that
\begin{align}
     \frac{1}{2}\frac{d}{dt}\|\textbf{w}_n\|^2 + 2\eta_0\|\nabla\textbf{w}_n\|^2 &\leq 4\epsilon\|\nabla\textbf{w}_n\|^2 + 4\epsilon\|\psi_n\|_{H^2}^2
     + C\|\textbf{w}_n\|^2\Big(\|\nabla\Bar{\textbf{u}}\|^2 + \|\nabla\Bar{\varphi}\|^2\|\Bar{\varphi}\|_{H^2}^2\Big)\nonumber \\&\hspace{.5cm}+ C_1\|\psi_n\|^2\Big(1+\|\nabla\Bar{\textbf{u}}\|^4+ \|\Delta\Bar{\varphi}\|^8 \Big).\label{eq2.16}
\end{align}
Now add \eqref{eq2.15}, \eqref{eq2.16} and apply the elliptic estimate for $\varphi$. A choice of $\epsilon= \min\{\frac{\eta_0}{5}, \frac{m_0}{26}\}$ results in the following inequality.
\begin{align}
     \frac{1}{2}\frac{d}{dt}(\|\psi_n\|^2 &+ \|\textbf{w}_n\|^2 ) + \frac{m_0}{2}\|\Delta\psi_n\|^2 + \eta_0\|\nabla \textbf{w}_n\|^2 \leq C\|h\|^2 + C\Lambda_1(\|\psi_n\|^2 + \|\textbf{w}_n\|^2 ), 
\end{align}
where $\Lambda_1 = ( 1+\|\Bar{\varphi}\|_{L^\infty}^{2r} + (1+\|\Bar{\varphi}\|_{L^\infty}^{2r})\|\nabla\Bar{\varphi}\|\|\Bar{\varphi}\|_{H^2} + \|\Bar{\varphi}\|_{H^3}^2+ \|\nabla\Bar{\varphi}\|^4\|\Bar{\varphi}\|_{H^2}^4 +\|\nabla\Bar{\textbf{u}}\|^4+ \|\Delta\Bar{\varphi}\|^8 +\|\nabla\Bar{\textbf{u}}\|^2 + \|\nabla\Bar{\varphi}\|^2\|\Bar{\varphi}\|_{H^2}^2)$. 
Since the above relation holds true for each $t\in [0,T]$, we can apply Gronwall's lemma,  and obtain,
\begin{align}\label{eq2.18}
     \sup_{t\in[0,T]}(\|\psi_n\|^2 + \|\textbf{w}_n\|^2 ) + \|\Delta\psi_n\|_{L^2(0,T;H)}^2 + \|\textbf{w}_n\|_{L^2(0,T;\mathbb{V}_{div})}^2 \leq C\big(\|h\|_{L^2(0,T:H)}^2 + \|\psi_n(0)\|^2 + \|\textbf{w}_n(0)\|^2 \big).
\end{align}
 Since the R.H.S. of \eqref{eq2.18} is independent of $t$, we have a uniform bound for the solutions and hence we can extend the solution of the finite dimensional problem to the entire interval $[0,T]$. 
Using the elliptic estimate for $\psi_n$, we have from \eqref{eq2.18},
\begin{equation}
    \|\psi_n\|_{L^2(0,T;H^2)}^2\leq C.
\end{equation}
\\To achieve higher regularity of $\psi_n$, test \eqref{eq2.9} with the test function $-\Delta\psi_n \in H$, we obtain,
\begin{align}
    \frac{1}{2}\frac{d}{dt}\|\nabla\psi_n\|^2 + (\textbf{w}_n\cdot\nabla\Bar{\varphi},-\Delta\psi_n) + \big(m(\Bar{\varphi})\nabla(-\Delta \psi_n + F''(\Bar{\varphi})\psi_n), \nabla(-\Delta\psi_n)\big)& \nonumber\\+ \big(m'(\Bar{\varphi})\psi_n\nabla(-\Delta \Bar{\varphi}+ F'(\Bar{\varphi})), \nabla(-\Delta\psi_n)\big) &= (h,-\Delta\psi_n),
\end{align}
Estimating each term by term by applying similar techniques we get,
\begin{align*}
    |(\textbf{w}_n\cdot\nabla\Bar{\varphi},-\Delta\psi_n)| &\leq \|\textbf{w}_n\|_{L^4}\|\nabla\Bar{\varphi}\|_{L^4}\|\Delta\psi_n\|\\
    &\leq \|\nabla\textbf{w}_n\|\|\nabla\Bar{\varphi}\|^{1/2}\|\Bar{\varphi}\|_{H^2}^{1/2}\|\Delta\psi_n\|\\
    &\leq  \|\nabla\textbf{w}_n\|^2 + \|\nabla\Bar{\varphi}\|\|\Bar{\varphi}\|_{H^2}\|\Delta\psi_n\|^2,\\
    |\big(m(\Bar{\varphi})\nabla(-\Delta \psi_n), \nabla(-\Delta\psi_n)\big)|&\geq m_0\|\nabla^3\psi_n\|^2,\\
    |\big(m(\Bar{\varphi})\nabla(F''(\Bar{\varphi})\psi_n), \nabla(-\Delta\psi_n)\big)| &\leq |\big(m(\Bar{\varphi})(F'''(\Bar{\varphi})\nabla\Bar{\varphi}\psi_n + F''(\Bar{\varphi})\nabla\psi_n ), \nabla(-\Delta\psi_n)\big)|\\
    &\leq m_1\|F'''(\Bar{\varphi})\|_{L^\infty}\|\nabla\Bar{\varphi}\|_{L^4}\|\psi_n\|_{L^4}\|\nabla^3\psi_n\| \\&\hspace{.5cm} + m_1\|F''(\Bar{\varphi})\|_{L^\infty}\|\nabla\psi_n\|\| \nabla^3\psi_n\|\\ 
    &\leq C(1+\|\Bar{\varphi}\|_{L^\infty}^{r-1})\|\nabla\Bar{\varphi}\|^{1/2}\|\Bar{\varphi}\|_{H^2}^{1/2}\|\psi_n\|^{1/2}\|\psi_n\|_{V}^{1/2}\|\nabla^3\psi_n\|\\&\hspace{.5cm}+ C(1+\|\Bar{\varphi}\|_{L^\infty}^{r})\|\nabla\psi_n\|\| \nabla^3\psi_n\|\\
    &\leq \frac{m_0}{4} \|\nabla^3\psi_n\|^2+ C(1+\|\Bar{\varphi}\|_{L^\infty}^{2r-2})\|\nabla\Bar{\varphi}\|\|\Bar{\varphi}\|_{H^2}\|\psi_n\|\|\psi_n\|_{V} \\&\hspace{.5cm}+ C(1+\|\Bar{\varphi}\|_{L^\infty}^{2r})\|\nabla\psi_n\|^2\\
    &\leq \frac{m_0}{4} \|\nabla^3\psi_n\|^2+ C(1+\|\Bar{\varphi}\|_{L^\infty}^{4r-4})\|\nabla\Bar{\varphi}\|^2\|\Bar{\varphi}\|_{H^2}^2\|\psi_n\|^2 \\&\hspace{.5cm}+ C(1+\|\Bar{\varphi}\|_{L^\infty}^{2r})\|\psi_n\|_{V}^2.
\end{align*}
Observe that,
\begin{align*}
     |\big(m'(\Bar{\varphi})\psi_n\nabla(-\Delta \Bar{\varphi}+ F'(\Bar{\varphi})), \nabla(-\Delta\psi_n)\big)| &\leq  \|m'\|_{L^\infty}\|\psi_n\|_{L^\infty}\|\nabla^3\Bar{\varphi}\|\|\nabla^3\psi_n\| \\&\hspace{.5cm}+  \|m'\|_{L^\infty}\|\psi_n\|_{L^4}\|F''(\Bar{\varphi})\|_{L^\infty}\|\nabla\Bar{\varphi}\|_{L^4}\|\nabla^3\psi_n\|\\
     &\leq  C\|\nabla^3\Bar{\varphi}\|\|\psi_n\|^{1/2}\|\psi_n\|_{H^2}^{1/2}\|\nabla^3\psi_n\| \\&\hspace{.5cm}+  C(1+ \|\Bar{\varphi}\|_{L^\infty}^r)\|\nabla\Bar{\varphi}\|^{1/2}\|\Bar{\varphi}\|_{H^2}^{1/2}\|\psi_n\|^{1/2}\|\psi_n\|_{V}^{1/2}\|\nabla^3\psi_n\|\\
      &\leq \frac{m_0}{4}\|\nabla^3\psi_n\|^2+ C\|\nabla^3\Bar{\varphi}\|^2\|\psi_n\|\|\psi_n\|_{H^2} \\&\hspace{.5cm}+  C(1+ \|\Bar{\varphi}\|_{L^\infty}^{2r})\|\nabla\Bar{\varphi}\|\|\Bar{\varphi}\|_{H^2}\|\psi_n\|\|\psi_n\|_{V}\\
      &\leq \frac{m_0}{4}\|\nabla^3\psi_n\|^2+ C\|\nabla^3\Bar{\varphi}\|^2\|\psi_n\|_{H^2}^2 \\&\hspace{.5cm}+  C(1+ \|\Bar{\varphi}\|_{L^\infty}^{2r})\|\nabla\Bar{\varphi}\|\|\Bar{\varphi}\|_{H^2}\|\psi_n\|\|\psi_n\|_{V},\\
      |(h,(-\Delta\psi_n))| &\leq \|h\|\|\Delta\psi_n\|\leq C(\|h\|^2+ \|\Delta\psi_n\|^2).
\end{align*}
Therefore by combining all above estimates, we get for fixed t,
\begin{align}
    \frac{1}{2}\frac{d}{dt}\|\nabla\psi_n\|^2 + \frac{m_0}{2}\|\nabla^3\psi_n\|^2 &\leq C\|h\|^2+ \|\nabla\textbf{w}_n\|^2 + C \big( 1+\|\nabla\Bar{\varphi}\|\|\Bar{\varphi}\|_{H^2} + \|\nabla^3\Bar{\varphi}\|^2 \big)\|\Delta\psi_n\|^2 + C(1+\|\Bar{\varphi}\|_{L^\infty}^{2r})\|\nabla\psi_n\|^2 \nonumber \\&\hspace{.5cm} + C\big(1+\|\Bar{\varphi}\|_{L^\infty}^{4r-4} \|\nabla\Bar{\varphi}\|^2\|\Bar{\varphi}\|_{H^2}^2+ \|\Bar{\varphi}\|_{L^\infty}^{4r} \|\nabla\Bar{\varphi}\|^2\|\Bar{\varphi}\|_{H^2}^2 + \|\nabla^3\Bar{\varphi}\|^2 \big)\|\psi_n\|^2. 
\end{align}
holds for each $t\in [0,T]$. Hence by exploiting Gronwall's lemma and applying \eqref{eq2.18} and uniform estimates for the strong solution, we get,
\begin{align}\label{eq2.22}    
\sup_{t\in[0,T]}\|\psi_n(t)\|_{V}^2 + \|\psi_n\|_{L^2(0,T:H^3)}^2 &\leq C\big(\|h\|_{L^2(0,T:H)}^2 +\|\psi_n(0)\|_{V}^2 + \|\textbf{w}_n(0)\|^2 \big).  
\end{align}
Moreover, for $\zeta \in V$,
\begin{align}
\langle\psi_n',\zeta\rangle + (\Bar{\textbf{u}}\cdot\nabla\psi_n,\zeta)  + (\textbf{w}_n\cdot\nabla\Bar{\varphi},\zeta) + \big(m(\Bar{\varphi})\nabla(-\Delta \psi_n + F''(\Bar{\varphi})\psi_n), \nabla\zeta\big)& \nonumber\\+ \big(m'(\Bar{\varphi})\psi_n\nabla(-\Delta \Bar{\varphi}+ F'(\Bar{\varphi})), \nabla\zeta\big) &= (h,\zeta),    
\end{align}
Observe that,
\begin{align*}
|\langle\psi_n',\zeta\rangle| &\leq \big(\|\Bar{\textbf{u}}\|_{L^4}\|\psi_n\|_{L^4}  + \|\textbf{w}_n\|\|\Bar{\varphi}\|_{L^\infty} + m_1(\|\nabla^3\psi_n\| + \|F'''(\Bar{\varphi})\|_{L^\infty}\|\nabla\Bar{\varphi}\|_{L^4}\|\psi_n\|_{L^4}+ \|F''(\Bar{\varphi})\|_{L^\infty}\|\nabla\psi_n\|) \\&\hspace{.5cm}+ |m'\|_{L^\infty}(\|\psi_n\|_{L^\infty}\|\nabla^3\Bar{\varphi}\|+ \|F''(\Bar{\varphi})\|_{L^\infty}\|\nabla\Bar{\varphi}\|_{L^4}\|\psi_n\|_{L^4})+ \|h\|\big)\|\zeta\|_{V}\\
&\leq \big(\|\nabla\Bar{\textbf{u}}\|\|\psi_n\|^{1/2}\|\psi_n\|_{V}^{1/2}  + \|\textbf{w}_n\|\|\Bar{\varphi}\|_{L^\infty} \\&\hspace{.5cm}+ C(\|\psi_n\|_{H^3} + (1+ \|\Bar{\varphi}\|_{L^\infty}^{r-1})\|\nabla\Bar{\varphi}\|^{1/2}\|\Bar{\varphi}\|_{H^2}^{1/2}\|\psi_n\|^{1/2}\|\psi_n\|_{V}^{1/2}+ (1+ \|\Bar{\varphi}\|_{L^\infty}^{r})\|\nabla\psi_n\|) \\&\hspace{.5cm}+ |m'\|_{L^\infty}(\|\psi_n\|^{1/2}\|\psi_n\|_{H^2}^{1/2}\|\nabla^3\Bar{\varphi}\|+ \|F''(\Bar{\varphi})\|_{L^\infty}\|\nabla\Bar{\varphi}\|^{1/2}\|\Bar{\varphi}\|_{H^2}^{1/2}\|\psi_n\|^{1/2}\|\psi_n\|_{V}^{1/2})+ \|h\|\big)\|\zeta\|_{V}\\
&\leq \Big(\|\nabla\Bar{\textbf{u}}\|^2\|\psi_n\|+ \|\psi_n\|_{V}  + \|\textbf{w}_n\|\|\Bar{\varphi}\|_{L^\infty} \\&\hspace{.5cm}+ C(\|\psi_n\|_{H^3} + (1+ \|\Bar{\varphi}\|_{L^\infty}^{2r-2})\|\nabla\Bar{\varphi}\|\|\Bar{\varphi}\|_{H^2}\|\psi_n\|+  (1+ \|\Bar{\varphi}\|_{L^\infty}^{r})\|\psi_n\|_{V}) \\&\hspace{.5cm}+ C(\|\Bar{\varphi}\|_{H^3}^2\|\psi_n\|+ \|\psi_n\|_{H^2}+ (1+ \|\Bar{\varphi}\|_{L^\infty}^{2r})\|\nabla\Bar{\varphi}\|\|\Bar{\varphi}\|_{H^2}\|\psi_n\|+ \|h\|)\Big)\|\zeta\|_{V}.
\end{align*}
Using uniform estimates \eqref{eq2.18}, \eqref{eq2.22} we can conclude that,
\begin{equation}
    \|\psi_n'\|_{L^2(0,T:V')} \leq C.\label{eq3.24}
\end{equation}
From $\eqref{eq2.10}$,
\begin{align*}
    |\langle\textbf{w}_n', \bm{\nu}\rangle| &\leq  |((\Bar{\textbf{u}}\cdot\nabla)\textbf{w}_n,\bm{\nu})| + |((\textbf{w}_n\cdot\nabla)\Bar{\textbf{u}},\bm{\nu})|+ 2|(\eta(\Bar{\varphi})D\textbf{w}_n, D\bm{\nu})| + 2|(\eta'(\Bar{\varphi})\psi_n D\Bar{\textbf{u}}, \bm{\nu})|  \\&\hspace{.5cm}+ |((\Delta\Bar{\varphi})\nabla\psi_n, \bm{\nu})| +| ((\Delta\psi_n)\nabla\Bar{\varphi}, \bm{\nu})| \\ 
    &\leq \|\Bar{\textbf{u}}\|_{L^4}\|\nabla\textbf{w}_n\|\|\bm{\nu}\|_{L^4} + \|\textbf{w}_n\|_{L^4}\|\nabla\Bar{\textbf{u}}\|\|\bm{\nu}\|_{L^4}+ 2\eta_1\|\nabla\textbf{w}_n\|\|\nabla\bm{\nu}\| + 2\|\eta'\|_{L^\infty}\|\psi_n\|_{L^4}\|\nabla\Bar{\textbf{u}}\|\|\bm{\nu}\|_{L^4}  \\&\hspace{.5cm}+ \|\Delta\Bar{\varphi}\|_{L^4}\|\nabla\psi_n\|\|\bm{\nu}\|_{L^4} + \|\Delta\psi_n\|\|\nabla\Bar{\varphi}\|_{L^4} \|\bm{\nu}\|_{L^4}\\
    &\leq \|\nabla\Bar{\textbf{u}}\|\|\nabla\textbf{w}_n\|\|\nabla\bm{\nu}\| + \|\nabla\textbf{w}_n\|\|\nabla\Bar{\textbf{u}}\|\|\nabla\bm{\nu}\|+ 2\eta_1\|\nabla\textbf{w}_n\|\|\nabla\bm{\nu}\| + C\|\psi_n\|^{1/2}\|\psi_n\|_{V}^{1/2}\|\nabla\Bar{\textbf{u}}\|\|\nabla\bm{\nu}\|  \\&\hspace{.5cm}+ \|\nabla^3\Bar{\varphi}\|^{3/4}\|\Bar{\varphi}\|_{L^\infty}^{1/4}\|\nabla\psi_n\|\|\nabla\bm{\nu}\| + \|\Delta\psi_n\|\|\nabla\Bar{\varphi}\|^{1/2}\|\Bar{\varphi}\|_{H^2}^{1/2} \|\nabla\bm{\nu}\|\\
    &\leq C\Big[\|\nabla\Bar{\textbf{u}}\|\|\nabla\textbf{w}_n\| + \|\nabla\textbf{w}_n\| + \|\nabla\Bar{\textbf{u}}\|\|\psi_n\|_{V} + \|\Bar{\varphi}\|_{H^3}\|\nabla\psi_n\| + \|\Bar{\varphi}\|_{H^2}\|\Delta\psi_n\|\Big] \|\nabla\bm{\nu}\|.
\end{align*}
This holds for all $\bm{\nu}\in \mathbb{V}_{div}$. Consequently,
\begin{equation}  \|{\textbf{w}_n}'\|_{L^2(0,T;\mathbb{V}_{div}')} \leq C.\label{eq3.25}
\end{equation}
\\\textbf{Passing to the limit $n\rightarrow \infty$}: Thanks to the derived uniform estimates, \eqref{eq2.18}, \eqref{eq2.22}, \eqref{eq3.24} and \eqref{eq3.25}, we obtain a weak limit of $(\psi_n, \textbf{w}_n)$, say $(\psi, \textbf{w})$ which solves the linearized system. In particular, we have the following convergence results,
\begin{align*}
    \psi_n &\overset{\ast}{\rightharpoonup} \psi, \text{ in } L^\infty(0,T;V),
    \\\psi_n &\rightharpoonup \psi, \text{ in } L^2(0,T;H^3),
    \\\psi'_n &\rightharpoonup \psi', \text{ in } L^2(0,T;V'),
    \\\textbf{w}_n &\overset{\ast}{\rightharpoonup} \textbf{w}, \text{ in } L^\infty(0,T;\mathbb{G}_{div}),
    \\\textbf{w}_n &\rightharpoonup \textbf{w}, \text{ in } L^2(0,T;\mathbb{V}_{div}),
    \\\textbf{w}_n' &\rightharpoonup \textbf{w}', \text{ in } L^2(0,T;\mathbb{V}_{div}').
\end{align*}
Above convergence results enable us to pass to the limit as $n\rightarrow \infty$ in the weak formulation of finite dimensional problem, \eqref{eq2.9}-\eqref{eq2.12}, and thereby obtain a weak solution to the linearised problem \eqref{eq2.3}-\eqref{eq2.8}. Since the system is linear, uniqueness of the solution can be proved using the uniform a priori estimates derived. 
\end{proof}

\begin{remark}
 By Gagliardo - Nirenberg inequality, the solution of the linearised system, $\psi \in L^2(0,T;H^2(\Omega)) \hookrightarrow L^2(0,T; C(\Bar{\Omega}))$. This regularity of $\psi$ is necessary to derive the first order optimality conditions. Similarly for the velocity, by exploiting the Aubin-Lions lemma, we deduce the regularity, $\textbf{w} \in L^\infty(0,T;\mathbb{G}_{div})\cap H^1(0,T;\mathbb{V}_{div}') \hookrightarrow C([0,T]; \mathbb{G}_{div})$.
\end{remark}

\subsubsection{Differentiability of the control to state operator}

\begin{theorem}\label{theorem3.4}
    Let all the hypotheses of theorem \ref{theorem1} be satisfied. In addition, assume $F \in C^4(\mathbb{R})$. The control to state operator $\mathcal{S}$ as a map from $\mathcal{U}$ to $\mathcal{Z}$ is Fr\'echet differentiable and the derivative of $\mathcal{S}$ at $\Bar{U}$ in the direction of $U$ is given by,
    \begin{equation*}
        \mathcal{S}'(\Bar{U})(U) = (\psi_U, \textbf{w}_U),
    \end{equation*}
 where $(\psi_{U}, \textbf{w}_{U})$ solves the linearized system \eqref{eq2.3}- \eqref{eq2.8} with forcing term $U$.  
\end{theorem}
\begin{proof}
Let  $h\in L^2(0,T; V\cap L^\infty(\Omega))$, we denote, $\mathcal{S}(\Bar{U}+h) = (\varphi_h,\textbf{u}_h)$, $\mathcal{S}(\Bar{U})= (\Bar{\varphi}, \Bar{\textbf{u}})$. We define 
$$ \rho = \varphi_h-\Bar{\varphi}-\psi_h, \hspace{.5cm} \textbf{v}= \textbf{u}_h-\Bar{\textbf{u}}-\textbf{w}_h,$$
where $(\psi_h,\textbf{w}_h),$ is the solution to the linearised system \eqref{eq2.3}- \eqref{eq2.8} with forcing term $h$.
For the computational purpose we denote $\xi= \varphi_h-\Bar{\varphi}$ and $\bm{\tau} = \textbf{u}_h-\Bar{\textbf{u}}$. Then $(\rho, \textbf{v})$ satisfies,

\begin{align}
    \rho' + \bm{\tau}\cdot\nabla\xi + \Bar{\textbf{u}}\cdot\nabla\rho + \textbf{v}\cdot\nabla\Bar{\varphi} &= \nabla\cdot(m(\Bar{\varphi})\nabla(-\Delta\rho))+ \nabla\cdot \big(m(\Bar{\varphi})\nabla(F'(\varphi_h)-F'(\Bar{\varphi})-F''(\Bar{\varphi})\psi) \big)\nonumber\\&\hspace{.5cm} \nabla\cdot(m'(\Bar{\varphi})\xi\nabla(-\Delta\xi))+ \nabla\cdot \big(m'(\Bar{\varphi})\xi\nabla(F'(\varphi_h)-F'(\Bar{\varphi})) \big)\nonumber\\&\hspace{.5cm}
    \nabla\cdot \big((m(\varphi_h)-m(\Bar{\varphi})-m'(\Bar{\varphi})\psi)\nabla(-\Delta\Bar{\varphi}+F'(\Bar{\varphi}))\big), \text{ in } Q, \label{eq2.30}\\
    \textbf{v}'+ (\bm{\tau}\cdot\nabla)\bm{\tau}+ (\Bar{\textbf{u}}\cdot\nabla)\textbf{v} + (\textbf{v}\cdot\nabla)\Bar{\textbf{u}}+ \nabla p^\ast &= \nabla\cdot(2\eta(\Bar{\varphi})D\textbf{v}) + \nabla\cdot(2\eta'(\Bar{\varphi})\xi D\bm{\tau}) + (-\Delta\Bar{\varphi})\nabla\rho+ (-\Delta\xi)\nabla\xi  \nonumber \\
    &\hspace{.5cm} +\nabla\cdot(2(\eta(\varphi_h)-\eta(\Bar{\varphi})-\eta'(\Bar{\varphi})\psi)D\Bar{\textbf{u}}) + (-\Delta\rho)\nabla\Bar{\varphi},  \text{ in } Q, \label{eq2.31}\\
    \nabla\cdot\textbf{v} &= 0, \text{ in } Q, \label{eq2.32}\\
    \textbf{v}&=\textbf{0},  \text{ on } \Sigma, \label{eq2.33}\\
    \frac{\partial\rho}{\partial \textbf{n}} =0,\hspace{.25cm} \frac{\partial(\Delta\rho)}{\partial \textbf{n}} &=0, \text{ on } \Sigma,\label{eq2.34}\\
    \rho(x,0) = \rho_0(x),\hspace{.25cm} \textbf{v}(x,0) &= \textbf{v}_0(x), \text{ in } \Omega.\label{eq2.35}
    \end{align}
It is worth noting that, 
\begin{align}\label{DS1}
\frac{\|(\mathcal{S}(\Bar{U}+ h )- \mathcal{S}(\Bar{U})- (\psi_{h}, \textbf{w}_{h}))\|_{\mathcal{Z}}}{\|h\|_{\mathcal{U}}} =  \frac{\|( (\varphi_h, \textbf{u}_h)- (\Bar{\varphi},\Bar{\textbf{u}})- (\psi_{h}, \textbf{w}_{h}))\|_{\mathcal{Z}}}{\|h\|_{\mathcal{U}}} =  \frac{\|(\rho, \textbf{v})\|_{\mathcal{Z}}}{\|h\|_{\mathcal{U}}}
\end{align}
In order to prove the differentiability of $\mathcal{S}$, it is enough to prove that the solution of the above system, $(\rho, \textbf{v})$ satisfies,
\begin{align}\label{DS}
    \frac{\|(\rho, \textbf{v})\|_{\mathcal{Z}}}{\|h\|_{\mathcal{U}}} \rightarrow 0, \hspace{.5cm}\text{ as }\|h\|_{\mathcal{U}}\rightarrow 0.
\end{align}
With this aim, we derive estimates for $\rho$ and $\bm{\tau}$ in $\mathcal{Z}$. First, we test \eqref{eq2.30} with $\rho$ which results in,
\begin{align}\label{eq234}
    \frac{1}{2}\frac{d}{dt}\|\rho\|^2 + (\bm{\tau}\cdot\nabla\xi,\rho) + (\Bar{\textbf{u}} \cdot\nabla\rho,\rho) + (\textbf{v}\cdot\nabla\Bar{\varphi}, \rho) &= -(m(\Bar{\varphi})\nabla(-\Delta\rho), \nabla\rho)- (m(\Bar{\varphi})\nabla(F'(\varphi_h)-F'(\Bar{\varphi})-F''(\Bar{\varphi})\psi), \nabla\rho )\nonumber\\&\hspace{.5cm} -(m'(\Bar{\varphi})\xi\nabla(-\Delta\xi), \nabla\rho)-  (m'(\Bar{\varphi})\xi\nabla(F'(\varphi_h)-F'(\Bar{\varphi})), \nabla\rho)\nonumber\\&\hspace{.5cm}
    -((m(\varphi_h)-m(\Bar{\varphi})-m'(\Bar{\varphi})\psi)\nabla(-\Delta\Bar{\varphi}+F'(\Bar{\varphi})), \nabla\rho).
\end{align}
Estimating each term by term by applying H\"olders, Gagliardo-Nirenberg, Ladyzhenskaya, Young's inequalities we get,
\begin{align*}
    |(\bm{\tau}\cdot\nabla\xi,\rho)| &\leq \|\bm{\tau}\|_{L^4}\|\nabla\xi\|_{L^4}\|\rho\|\\
    &\leq \frac{1}{2}\|\bm{\tau}\|\|\nabla\bm{\tau}\|\|\nabla\xi\|\|\xi\|_{H^2}+ \frac{1}{2}\|\rho\|^2,
    \\(\Bar{\textbf{u}}\cdot\nabla\rho,\rho) &= \int_\Omega\Bar{\textbf{u}}\cdot\nabla\big(\frac{\rho^2}{2}\big)= 0,\\    
    |(\textbf{v}\cdot\nabla\Bar{\varphi}, \rho)| &\leq \|\textbf{v}\|_{L^4}\|\nabla\Bar{\varphi}\|_{L^4}\|\rho\| \leq \|\nabla\textbf{v}\|\|\nabla\Bar{\varphi}\|^{1/2}\|\Bar{\varphi}\|_{H^2}^{1/2}\|\rho\|\\
    &\leq \epsilon\|\nabla\textbf{v}\|^2 + \|\nabla\Bar{\varphi}\|\|\Bar{\varphi}\|_{H^2}\|\rho\|^2,\\
    (m(\Bar{\varphi})\nabla(-\Delta\rho), \nabla\rho) &= (m(\Bar{\varphi})(\Delta\rho), \Delta\rho) - (m'(\Bar{\varphi})\nabla\Bar{\varphi}(-\Delta\rho), \nabla\rho).
\end{align*}
In the above expression, estimating each of its terms separately, we obtain,
\begin{align*}
   (m(\Bar{\varphi})\Delta\rho, \Delta\rho) &\geq m_0\|\Delta\rho\|^2,\\
   |(m'(\Bar{\varphi})\nabla\Bar{\varphi}(-\Delta\rho), \nabla\rho)| &\leq \|m'\|_{L^\infty}\|\nabla\Bar{\varphi}\|_{L^4}\|\nabla\rho\|_{L^4}\|\Delta\rho\|\\
   &\leq C\|\nabla\Bar{\varphi}\|^{1/2}\|\Bar{\varphi}\|_{H^2}^{1/2}\|\rho\|_{H^2}^{3/4}\|\rho\|^{1/4}\|\Delta\rho\|\\
   &\leq \epsilon\|\Delta\rho\|^2 + C\|\nabla\Bar{\varphi}\|\|\Bar{\varphi}\|_{H^2}\|\rho\|_{H^2}^{3/2}\|\rho\|^{1/2}\\
   &\leq 2\epsilon\|\rho\|_{H^2}^2 + C\|\nabla\Bar{\varphi}\|^2\|\Bar{\varphi}\|_{H^2}^2\|\rho\|_{H^2}\|\rho\|\\
   &\leq 3\epsilon\|\rho\|_{H^2}^2 + C\|\nabla\Bar{\varphi}\|^4\|\Bar{\varphi}\|_{H^2}^4\|\rho\|^2.
\end{align*}
Using Taylor's expansion we have,
\begin{align*}
    F'(\varphi_h) &= F'(\Bar{\varphi}) + F''(\Bar{\varphi})(\varphi_h-\Bar{\varphi}) + F'''(r\varphi_h + (1-r)\Bar{\varphi})(\varphi_h-\Bar{\varphi})^2,
\end{align*}
for some $r \in [0,1]$. Observe further that  $F'(\varphi_h)-F'(\Bar{\varphi})-F''(\Bar{\varphi})\psi$ can be written in the following way that enables us to derive an expression for its derivative. We follow similar methods as in \cite[Theorem 3.4]{FRS}.
\begin{align}\label{F11}
F'(\varphi_h)-F'(\Bar{\varphi})-F''(\Bar{\varphi})\psi &= \xi\int\limits_{0}^{1} [F''(r\varphi_h +(1-r)\Bar{\varphi})- F''(\Bar{\varphi})]dr + F''(\Bar{\varphi})\rho.
\end{align}
This leads to  the following expression for $\nabla\big( F'(\varphi_h)-F'(\Bar{\varphi})-F''(\Bar{\varphi})\psi \big)$,
\begin{align*}
  &\nabla\Big( F'(\varphi_h)-F'(\Bar{\varphi})-F''(\Bar{\varphi})\psi \Big) \\&= \nabla\xi\int\limits_{0}^{1} [F''(r\varphi_h +(1-r)\Bar{\varphi})- F''(\Bar{\varphi})]dr + \xi\int\limits_{0}^{1} [F'''(r\varphi_h +(1-r)\Bar{\varphi})(r\nabla\varphi_h +(1-r)\nabla\Bar{\varphi})- F'''(\Bar{\varphi})\nabla\Bar{\varphi}]dr \\&\hspace{.5cm}+ F'''(\Bar{\varphi})\nabla\Bar{\varphi}\rho + F''(\Bar{\varphi})\nabla\rho  \\
  &= \nabla\xi\int\limits_{0}^{1}\int\limits_{0}^{1} [F'''(s(r\varphi_h +(1-r)\Bar{\varphi}) + (1-s)\Bar{\varphi})(r\varphi_h+(1-r)\Bar{\varphi} - \Bar{\varphi})]dr + \xi\int\limits_{0}^{1} [F'''(r\varphi_h +(1-r)\Bar{\varphi})(r\nabla\varphi_h \\&\hspace{.5cm}+(1-r)\nabla\Bar{\varphi}-\nabla\Bar{\varphi})]dr  + \xi\int\limits_{0}^{1} [(F'''(r\varphi_h+ (1-r)\Bar{\varphi})- F'''(\Bar{\varphi}))\nabla\Bar{\varphi}]dr + F'''(\Bar{\varphi})\nabla\Bar{\varphi}\rho + F''(\Bar{\varphi})\nabla\rho \\
   &= \nabla\xi\int\limits_{0}^{1}\int\limits_{0}^{1} [ F'''(s(r\varphi_h +(1-r)\Bar{\varphi}) + (1-s)\Bar{\varphi})r\xi ]dr + \xi\int\limits_{0}^{1} [F'''(r\varphi_h +(1-r)\Bar{\varphi}) r\nabla\xi]dr \\
  &\hspace{.5cm}+ \xi\int\limits_{0}^{1} [F^{(4)}(s(r\varphi_h +(1-r)\Bar{\varphi}) + (1-s)\Bar{\varphi})r\xi\nabla\Bar{\varphi}]dr + F'''(\Bar{\varphi})\nabla\Bar{\varphi}\rho + F''(\Bar{\varphi})\nabla\rho 
\end{align*} 
 Hence, 
\begin{equation}\label{F12}
    \nabla\big( F'(\varphi_h)-F'(\Bar{\varphi})-F''(\Bar{\varphi})\psi \big) =  P_h\xi\nabla\xi + Q_h\xi^2\nabla\Bar{\varphi}+ F'''(\Bar{\varphi})\nabla\Bar{\varphi}\rho + F''(\Bar{\varphi})\nabla\rho 
\end{equation}
Where, 
\begin{align*}
P_h &= \int\limits_{0}^{1}\int\limits_{0}^{1} [ F'''(s(r\varphi_h +(1-r)\Bar{\varphi}) + (1-s)\Bar{\varphi})r + F'''(r\varphi_h +(1-r)\Bar{\varphi}) r] ds\ dr,\\
Q_h &= \int\limits_{0}^{1} [F^{(4)}(s(r\varphi_h +(1-r)\Bar{\varphi}) + (1-s)\Bar{\varphi})r]dr
\end{align*}
are bounded in $L^\infty(\Omega)$ using the uniform global bound for the strong solution $\Bar{\varphi}$. Therefore, 
\begin{equation*}
   (m(\Bar{\varphi})\nabla(F'(\varphi_h)-F'(\Bar{\varphi})-F''(\Bar{\varphi})\psi), \nabla\rho ) = (m(\Bar{\varphi})(P_h\xi\nabla\xi + Q_h\xi^2\nabla\Bar{\varphi} + F''(\Bar{\varphi})\nabla\rho + \rho F'''(\Bar{\varphi})\nabla\Bar{\varphi}), \nabla\rho )
\end{equation*}
Since $m$ and $F''$ are bounded below by a positive constant, there exists a constant $c_0 > 0$ such that,
\begin{equation*}
    (m(\Bar{\varphi})F''(\Bar{\varphi})\nabla\rho  , \nabla\rho) \geq c_0\|\nabla\rho\|^2,
\end{equation*}
The remaining terms can be estimated as follows: 
\begin{align*}
    |(m(\Bar{\varphi})(P_h\xi\nabla\xi + Q_h\xi^2\nabla\Bar{\varphi} + \rho F'''(\Bar{\varphi})\nabla\Bar{\varphi}), \nabla\rho )| 
    &\leq C\|\xi\|_{L^4}\|\nabla\xi\|\|\nabla\rho\|_{L^4} + \|\nabla\Bar{\varphi}\|_{L^4}\|\xi\|_{L^4}^2\|\nabla\rho\|_{L^4} + \|\nabla\Bar{\varphi}\|\|\rho\|_{L^4}\|\nabla\rho\|_{L^4} \\
    &\leq C\|\xi\|_{V}^{1/2}\|\xi\|^{1/2}\|\nabla\xi\|\|\rho\|_{H^2}^{3/4}\|\rho\|^{1/4} + \|\Bar{\varphi}\|_{H^2}^{1/2}\|\nabla\Bar{\varphi}\|^{1/2}\|\xi\|_{V}\|\xi\|\|\rho\|_{H^2}^{3/4}\|\rho\|^{1/4} \\
    &\hspace{.5cm}+ \|\nabla\Bar{\varphi}\|\|\rho\|_{H^2}^{1/4}\|\rho\|^{3/4}\|\rho\|_{H^2}^{3/4}\|\rho\|^{1/4}\\
    &\leq C\|\xi\|_{V}\|\xi\|\|\nabla\xi\|^2 + \|\rho\|_{H^2}^{3/2}\|\rho\|^{1/2} + \|\Bar{\varphi}\|_{H^2}\|\nabla\Bar{\varphi}\|\|\xi\|_{V}^2\|\xi\|^2  \\
    &\hspace{.5cm}+ \|\nabla\Bar{\varphi}\|\|\rho\|_{H^2}\|\rho\|\\
    &\leq \epsilon\|\rho\|_{H^2}^2 + C\|\xi\|_{V}^3\|\xi\| + \|\rho\|_{H^2}\|\rho\| + \|\Bar{\varphi}\|_{H^2}^2\|\xi\|_{V}^2\|\xi\|^2 + \|\nabla\Bar{\varphi}\|^2\|\rho\|^2\\
    &\leq 2\epsilon\|\rho\|_{H^2}^2 + \|\xi\|_{V}^3\|\xi\| + \|\Bar{\varphi}\|_{H^2}^2\|\xi\|_{V}^2\|\xi\|^2 + C(1 +\|\nabla\Bar{\varphi}\|^2 )\|\rho\|^2
\end{align*}

By similar techniques,
\begin{align*}
  |(m'(\Bar{\varphi})\xi\nabla(-\Delta\xi), \nabla\rho)| &\leq \|m'\|_{L^\infty}\|\xi\|_{L^4}\|\nabla^3\xi\|\|\nabla\rho\|_{L^4}\\
    &\leq C\|\xi\|^{1/2}\|\xi\|_{V}^{1/2}\|\nabla^3\xi\|\|\rho\|_{H^2}^{3/4}\|\rho\|^{1/4}\\
    &\leq C\|\xi\|\|\xi\|_{V}\|\nabla^3\xi\|^2 + C\|\rho\|_{H^2}^{3/2}\|\rho\|^{1/2}\\
    &\leq \epsilon\|\rho\|_{H^2}^2 + C\|\xi\|\|\xi\|_{V}\|\nabla^3\xi\|^2 + C\|\rho\|_{H^2}\|\rho\|\\&\leq 2\epsilon\|\rho\|_{H^2}^2 + C\|\xi\|\|\xi\|_{V}\|\xi\|_{H^3}^2 + C\|\rho\|^2.
\end{align*}

\begin{align*}
    |(m'(\Bar{\varphi})\xi\nabla(F'(\varphi_h)-F'(\Bar{\varphi})), \nabla\rho)| &= |(m'(\Bar{\varphi})\xi((F''(\varphi_h)-F''(\Bar{\varphi}))\nabla\varphi_h + F''(\Bar{\varphi})\nabla\xi), \nabla\rho)| \\
    &\leq \|m'\|_{L^\infty}\|F'''(\Hat{\varphi})\|_{L^\infty}\|\xi\|_{L^4}^2\|\nabla\varphi_h\|_{L^4}\|\nabla\rho\|_{L^4} + \|m'\|_{L^\infty}\|F''(\Bar{\varphi})\|_{L^\infty}\|\xi\|_{L^4}\|\nabla\xi\|\|\nabla\rho\|_{L^4}\\
    &\leq C\|\varphi_h\|_{H^2}^{1/2}\|\nabla\varphi_h\|^{1/2}\|\xi\|_{V}\|\xi\|\|\rho\|_{H^2}^{3/4}\|\rho\|^{1/4} + C\|\xi\|_{V}^{1/2}\|\xi\|^{1/2}\|\nabla\xi\|\|\rho\|_{H^2}^{3/4}\|\rho\|^{1/4}\\
    &\leq C\|\varphi_h\|_{H^2}\|\varphi_h\|\|\xi\|_{V}^2\|\xi\|^2 + \|\rho\|_{H^2}^{3/2}\|\rho\|^{1/2}+ C\|\xi\|_{V}\|\xi\|\|\nabla\xi\|^2\\
    &\leq \epsilon\|\rho\|_{H^2}^2 + C\|\varphi_h\|_{H^2}^{2}\|\xi\|_{V}^2\|\xi\|^2 + C\|\rho\|_{H^2}\|\rho\|+ C\|\xi\|_{V}^3\|\xi\|\\
    &\leq 2\epsilon\|\rho\|_{H^2}^2 + C\|\varphi_h\|_{H^2}^2\|\xi\|_{V}^2\|\xi\|^2 + C\|\rho\|^2+ C\|\xi\|_{V}^3\|\xi\|
\end{align*}
The second inequality is obtained by substituting for $(F''(\varphi_h)-F''(\Bar{\varphi}))$ from  Taylor's formula. Further, we used assumptions \eqref{F1}, \eqref{eqF4} and the regularity of $\varphi$ to get a $L^\infty$ bound for $F''(\bar{\varphi})$ and $F'''(\Bar{\varphi})$.
Similarly the last term of \eqref{eq234} can be bounded as follows,

\begin{align*}
   |((m(\varphi_h)-m(\Bar{\varphi})-m'(\Bar{\varphi})\psi)\nabla\Bar{\mu}, \nabla\rho)| &= |(m''(\Hat{\varphi})\xi^2 + m'(\Bar{\varphi})\rho)\nabla\Bar{\mu}, \nabla\rho)|
   \\&\leq \|m''\|_{L^\infty}\|\nabla\Bar{\mu}\|_{L^4}\|\xi\|_{L^4}^2\|\nabla\rho\|_{L^4} + \|m'\|_{L^\infty}\|\nabla\Bar{\mu}\|\|\rho\|_{L^4}\|\nabla\rho\|_{L^4} \\
   &\leq C\|\Bar{\mu}\|_{H^2}^{1/2}\|\nabla\Bar{\mu}\|^{1/2}\|\xi\|_{V}\|\xi\|\|\rho\|_{H^2}^{3/4}\|\rho\|^{1/4} \\&\hspace{.5cm}+ C\|\nabla\Bar{\mu}\|\|\rho\|_{H^2}^{1/4}\|\rho\|^{3/4}\|\rho\|_{H^2}^{3/4}\|\rho\|^{1/4}\\
   &\leq C\|\Bar{\mu}\|_{H^2}\|\nabla\Bar{\mu}\|\|\xi\|_{V}^2\|\xi\|^2 + \|\rho\|_{H^2}^{3/2}\|\rho\|^{1/2} + C\|\nabla\Bar{\mu}\|\|\rho\|_{H^2}\|\rho\|\\
   &\leq 2\epsilon\|\rho\|_{H^2}^{2} + C\|\Bar{\mu}\|_{H^2}\|\nabla\Bar{\mu}\|\|\xi\|_{V}^2\|\xi\|^2 + C\|\rho\|_{H^2}\|\rho\| + C\|\nabla\Bar{\mu}\|^2\|\rho\|^2\\
   &\leq 3\epsilon\|\rho\|_{H^2}^{2}  + C\|\Bar{\mu}\|_{H^2}\|\nabla\Bar{\mu}\|\|\xi\|_{V}^2\|\xi\|^2 + C\|\rho\|^2 + C\|\nabla\Bar{\mu}\|^2\|\rho\|^2
\end{align*}   

By combining all preceding estimates, we obtain,

\begin{align}\label{eq2.37}
    \frac{1}{2}\frac{d}{dt}\|\rho\|^2 + m_0\|\Delta\rho\|^2 + c_0\|\nabla\rho\|^2&\leq 12\epsilon\|\rho\|_{H^2}^2  + \epsilon\|\nabla\textbf{v}\|^2 + \|\bm{\tau}\|\|\nabla\bm{\tau}\|\|\nabla\xi\|\|\xi\|_{H^2} + C\|\xi\|_{V}^3\|\xi\| \nonumber \\&\hspace{.5cm} + C\Big[ 1+ \|\nabla\Bar{\varphi}\|\|\Bar{\varphi}\|_{H^2} + \|\nabla\Bar{\varphi}\|^4\|\Bar{\varphi}\|_{H^2}^4 + \|\nabla\Bar{\varphi}\|^2 + \|\nabla\Bar{\mu}\|^2\Big]\|\rho\|^2 \nonumber \\&\hspace{.5cm}+ C\Big[ \|\Bar{\varphi}\|_{H^2}^2 + \|\Bar{\mu}\|_{H^2}\|\nabla\Bar{\mu}\| + \|\varphi_h\|_{H^2}^2 \Big]\|\xi\|_{V}^2\|\xi\|^2  +  C\|\xi\|\|\xi\|_{V}\|\xi\|_{H^3}^2 
\end{align}

Now we test \eqref{eq2.31} with $\textbf{v}$. This yields,
\begin{align*}
    \frac{1}{2}\frac{d}{dt}\|\textbf{v}\|^2+ ((\bm{\tau}\cdot\nabla)\bm{\tau}, \textbf{v})+ ((\Bar{\textbf{u}}\cdot\nabla)\textbf{v},\textbf{v}) + ((\textbf{v}\cdot\nabla)\Bar{\textbf{u}}, \textbf{v}) + 2(\eta(\Bar{\varphi})D\textbf{v}, D\textbf{v})+ 2(\eta'(\Bar{\varphi})\xi D\bm{\tau}, D\textbf{v}) &\nonumber \\+ 2((\eta(\varphi_h)-\eta(\Bar{\varphi})-\eta'(\Bar{\varphi})\psi)D\Bar{\textbf{u}}, D\textbf{v})
    +((\Delta\Bar{\varphi})\nabla\rho, \textbf{v}) + ((\Delta\xi)\nabla\xi, \textbf{v}) + ((\Delta\rho)\nabla\Bar{\varphi}, \textbf{v}) &= 0.
\end{align*}
We estimate each term by  applying standard techniques/ inequalities.
\begin{align*}
    |((\bm{\tau}\cdot\nabla)\bm{\tau}, \textbf{v})|&=  |((\bm{\tau}\cdot\nabla)\textbf{v}, \bm{\tau})| \leq \|\nabla\textbf{v}\|\|\bm{\tau}\|_{L^4}^2 \leq \|\nabla\textbf{v}\|\|\bm{\tau}\|\|\nabla\bm{\tau}\|\\ &\leq \epsilon \|\nabla\textbf{v}\|^2 + C\|\bm{\tau}\|^2\|\nabla\bm{\tau}\|^2, \\
    |((\Bar{\textbf{u}}\cdot\nabla)\textbf{v},\textbf{v})| &\leq \|\Bar{\textbf{u}}\|_{L^4}\|\nabla\textbf{v}\|\|\textbf{v}\|_{L^4} \leq \|\Bar{\textbf{u}}\|^{1/2}\|\nabla\Bar{\textbf{u}}\|^{1/2}\|\nabla\textbf{v}\|\|\textbf{v}\|^{1/2}\|\nabla\textbf{v}\|^{1/2}\\
    &\leq \epsilon\|\nabla\textbf{v}\|^2 + C\|\Bar{\textbf{u}}\|\|\nabla\Bar{\textbf{u}}\|\|\textbf{v}\|\|\nabla\textbf{v}\|\\
    &\leq 2\epsilon\|\nabla\textbf{v}\|^2 + C\|\Bar{\textbf{u}}\|^2\|\nabla\Bar{\textbf{u}}\|^2\|\textbf{v}\|^2,\\
    |((\textbf{v}\cdot\nabla)\Bar{\textbf{u}}, \textbf{v})| &\leq \|\nabla\Bar{\textbf{u}}\|\|\textbf{v}\|_{L^4}^2 \leq \|\nabla\Bar{\textbf{u}}\|\|\textbf{v}\|\|\nabla\textbf{v}\|\\
    &\leq \epsilon \|\nabla\textbf{v}\|^2 + C\|\nabla\Bar{\textbf{u}}\|^2\|\textbf{v}\|^2,\\
    (\eta(\Bar{\varphi})D\textbf{v}, D\textbf{v}) &\geq \eta_0\|\nabla\textbf{v}\|^2,\\
  |(\eta'(\Bar{\varphi})\xi D\bm{\tau}, D\textbf{v})| &\leq \|\eta'\|_{L^\infty}\|\xi\|_{L^\infty}\|\nabla\bm{\tau}\|\|\nabla\textbf{v}\| \leq C\|\xi\|^{1/2}\|\xi\|_{H^2}^{1/2}\|\nabla\bm{\tau}\|\|\nabla\textbf{v}\|\\ &\leq \epsilon\|\nabla\textbf{v}\|^2 + C\|\xi\|\|\xi\|_{H^2}\|\nabla\bm{\tau}\|^2, 
\end{align*}
\begin{align*}
    |((\eta(\varphi_h)-\eta(\Bar{\varphi})-\eta'(\Hat{\varphi})\psi)D\Bar{\textbf{u}}, D\textbf{v})| &\leq |((\eta''(\Bar{\varphi})\xi^2+ \eta'(\Bar{\varphi})\rho) D\Bar{\textbf{u}}, D\textbf{v})|\\
    &\leq \|\eta''\|_{L^\infty}\|\xi\|_{L^8}^2\|\nabla\Bar{\textbf{u}}\|_{L^4}\|\nabla\textbf{v}\| + \|\eta'\|_{L^\infty}\|\rho\|_{L^4}\|\nabla\Bar{\textbf{u}}\|_{L^4}\|\nabla\textbf{v}\|\\
    &\leq C\|\xi\|_{V}^{3/2}\|\xi\|^{1/2}\|\nabla\Bar{\textbf{u}}\|^{1/2}\|\Bar{\textbf{u}}\|_{\textbf{H}^2}^{1/2}\|\nabla\textbf{v}\| + C\|\rho\|_{V}^{1/2}\|\rho\|^{1/2}\|\nabla\Bar{\textbf{u}}\|^{1/2}\|\Bar{\textbf{u}}\|_{\textbf{H}^2}^{1/2}\|\nabla\textbf{v}\|\\
    &\leq 2\epsilon \|\nabla\textbf{v}\|^2 + C\|\nabla\Bar{\textbf{u}}\|\|\Bar{\textbf{u}}\|_{\textbf{H}^2}\|\xi\|_{V}^3\|\xi\| + \|\rho\|_{V}\|\rho\|\|\nabla\Bar{\textbf{u}}\|\|\Bar{\textbf{u}}\|_{\textbf{H}^2}\\
    &\leq 2\epsilon \|\nabla\textbf{v}\|^2 + \epsilon\|\rho\|_{V}^2 + C\|\nabla\Bar{\textbf{u}}\|\|\Bar{\textbf{u}}\|_{\textbf{H}^2}\|\xi\|_{V}^3\|\xi\| + \|\rho\|^2\|\nabla\Bar{\textbf{u}}\|^2\|\Bar{\textbf{u}}\|_{\textbf{H}^2}^2,\\
   |(\nabla\Bar{\varphi}\Delta\rho, \textbf{v})| &\leq \|\nabla\Bar{\varphi}\|_{L^4}\|\Delta\rho\|\|\textbf{v}\|_{L^4} \\& \leq \|\nabla\Bar{\varphi}\|^{1/2}\|\Bar{\varphi}\|_{H^2}^{1/2}\|\Delta\rho\|\|\textbf{v}\|^{1/2}\|\nabla\textbf{v}\|^{1/2}
    \\& \leq \epsilon\|\Delta\rho\|^2 + C\|\nabla\Bar{\varphi}\|\|\Bar{\varphi}\|_{H^2}\|\textbf{v}\|\|\nabla\textbf{v}\|
    \\& \leq \epsilon\|\Delta\rho\|^2 + \epsilon \|\nabla\textbf{v}\|^2 + C\|\nabla\Bar{\varphi}\|^2\|\Bar{\varphi}\|_{H^2}^2\|\textbf{v}\|^2,\\|(\Delta\Bar{\varphi}\nabla\rho, \textbf{v})|= |\big(\nabla(\nabla\Bar{\varphi}\cdot\nabla\rho) - \nabla\Bar{\varphi}\Delta\rho , \textbf{v}\big)| &= |(\nabla\Bar{\varphi}\Delta\rho, \textbf{v})| \\&\leq \epsilon\|\Delta\rho\|^2 + \epsilon \|\nabla\textbf{v}\|^2 + C\|\nabla\Bar{\varphi}\|^2\|\nabla\Bar{\varphi}\|_{H^2}^2\|\textbf{v}\|^2.
\end{align*}
Observe that, $(\Delta\xi\nabla\xi, \textbf{v}) = 0$. Then by combining all above estimates we get,
\begin{align}\label{eq2.38}
   \frac{1}{2}\frac{d}{dt}\|\textbf{v}\|^2 +  \eta_0\|\nabla\textbf{v}\|^2 &\leq 10\epsilon \|\nabla\textbf{v}\|^2 + 2\epsilon\|\Delta\rho\|^2 + \epsilon\|\nabla\rho\|^2+ C\|\nabla\bm{\tau}\|^2\|\bm{\tau}\|^2 \nonumber\\&\hspace{.5cm}+ C\|\nabla\Bar{\textbf{u}}\|\|\Bar{\textbf{u}}\|_{\textbf{H}^2} \|\xi\|_{V}^3\|\xi\| + C\|\xi\|\|\xi\|_{H^2}\|\nabla\bm{\tau}\|^2 \nonumber\\
   & \hspace{.5cm}+ C\big( 1+ \|\nabla\Bar{\textbf{u}}\|^2+ \|\Bar{\textbf{u}}\|^2\|\nabla\Bar{\textbf{u}}\|^2 + \|\nabla\Bar{\varphi}\|^2\|\Bar{\varphi}\|_{H^2}^2 \big)\|\textbf{v}\|^2 
   \nonumber\\
   & \hspace{.5cm}+ C( 1+ \|\nabla\Bar{\textbf{u}}\|^2\|\Bar{\textbf{u}}\|_{\textbf{H}^2}^2)\|\rho\|^2.
\end{align}
Adding \eqref{eq2.37} and \eqref{eq2.38}, and for $\epsilon = \min\{ \frac{m_0}{28}, \frac{\eta_0}{24}, \frac{c_0}{2}\}$, we get,
\begin{align}
    \frac{1}{2}\frac{d}{dt}(\|\rho\|^2 + \|\textbf{v}\|^2) + \frac{m_0}{2}\|\Delta\rho\|^2 + \frac{\eta_0}{2}\|\nabla\textbf{v}\|^2  &\leq \Lambda_1 + \Lambda(\|\rho\|^2 + \|\textbf{v}\|^2),
\end{align}
where $\Lambda = C\big( 1+ \|\nabla\Bar{\varphi}\|\|\Bar{\varphi}\|_{H^2} +  \|\nabla\Bar{\varphi}\|^2\|\Bar{\varphi}\|_{H^2}^2+ \|\nabla\Bar{\varphi}\|^4\|\Bar{\varphi}\|_{H^2}^4 +\|\nabla\Bar{\textbf{u}}\|^2+ \|\Bar{\textbf{u}}\|^2\|\nabla\Bar{\textbf{u}}\|^2 + \|\nabla\Bar{\textbf{u}}\|^2\|\Bar{\textbf{u}}\|_{\textbf{H}^2}^2 + \|\nabla\Bar{\varphi}\|^2 +  \|\nabla\Bar{\mu}\|^2 \big)$ and 
\begin{align*}
\Lambda_1 &= C \big( \|\nabla\xi\|\|\xi\|_{H^2}\|\bm{\tau}\|\|\nabla\bm{\tau}\| + \|\xi\|\|\xi\|_{H^2}\|\nabla\bm{\tau}\|^2+ \|\xi\|\|\xi\|_{V}\|\xi\|_{H^3}^2 + \|\bm{\tau}\|^2\|\nabla\bm{\tau}\|^2 \big) \\
    &\hspace{.5cm}+ C(\|\Bar{\varphi}\|_{H^2}^2 + \|\Bar{\mu}\|_{H^2}\|\nabla\Bar{\mu}\| + \|\varphi_h\|_{H^2}^2)\|\xi\|^2\|\xi\|_{V}^2 + C(1+ \|\nabla\Bar{\textbf{u}}\|\|\Bar{\textbf{u}}\|_{\textbf{H}^2})\|\xi\|_{V}^{3}\|\xi\| 
\end{align*}
It is worth noting that using the estimates, \eqref{cd1}, \eqref{cd2}, and the regularity of linearized variables $(\psi, \textbf{w})$, it follows that $\int\limits_{0}^{T}\Lambda \leq C$ and $\int\limits_{0}^{T}\Lambda_1 \leq C\|h\|_{L^2(0,T;H)}^4$. By exploiting Gronwall's lemma and an elliptic estimate for $\rho$, we deduce,
\begin{align}\label{eq3.43}
    \sup_{t\in[0,T]}(\|\rho\|^2+ \|\textbf{v}\|^2) + \|\rho\|_{L^2(0,T;H^2)}^2 + \|\textbf{v}\|_{L^2(0,T;\mathbb{V}_{div})}^2 &\leq C\|h\|_{L^2(0,T;H)}^4 \leq C\|h\|_{\mathcal{U}}^4
\end{align}
Furthermore, for the regularity of velocity in time, test $\textbf{v}'$ with an arbitrary test function $\bm{\nu} \in V$, we obtain, 
\begin{align*}
    |\langle\textbf{v}', \bm{\nu}\rangle| &\leq |((\bm{\tau}\cdot\nabla)\bm{\tau}, \bm{\nu})|+ |((\Bar{\textbf{u}}\cdot\nabla)\textbf{v},\bm{\nu})|+ |((\textbf{v}\cdot\nabla)\Bar{\textbf{u}}, \bm{\nu})| + 2|(\eta(\Bar{\varphi})D\textbf{v}, D\bm{\nu})|+ 2|(\eta'(\Bar{\varphi})\xi D\bm{\tau}, D\bm{\nu})|  \\ &\hspace{.5cm}+ 2|((\eta(\varphi_h)-\eta(\Bar{\varphi})-\eta'(\Bar{\varphi})\psi)D\Bar{\textbf{u}}, D\bm{\nu})|
    + |((\Delta\Bar{\varphi})\nabla\rho, \bm{\nu})| + |((\Delta\xi)\nabla\xi, \bm{\nu})| + |((\Delta\rho)\nabla\Bar{\varphi}, \bm{\nu})|
    \\ &\leq \|\bm{\tau}\|_{L^4}^2\|\nabla\bm{\nu}\| + \|\Bar{\textbf{u}}\|_{L^4}\|\textbf{v}\|_{L^4}\|\nabla\bm{\nu}\| + \|\textbf{v}\|_{L^4}\|\Bar{\textbf{u}}\|_{L^4}\|\nabla\bm{\nu}\| + 2\eta_1\|\nabla\textbf{v}\|\|\nabla\bm{\nu}\| + 2\|\eta'(\Bar{\varphi})\|_{L^\infty}\|\xi\|_{L^\infty}\|\nabla\bm{\tau}\|\|\nabla\bm{\nu}\|  \\ &\hspace{.5cm}+ \|\eta''(\Hat{\varphi})\|_{L^\infty}\|\xi\|_{L^8}^2 \|\nabla\Bar{\textbf{u}}\|_{L^4}\|\nabla\bm{\nu}\| +  \|\eta'(\Bar{\varphi})\|_{L^\infty}\|\rho\|_{L^4}\|\nabla\Bar{\textbf{u}}\|_{L^4}\|\nabla\bm{\nu}\|
    + \|\Delta\Bar{\varphi}\|_{L^4}\|\nabla\rho\|\|\bm{\nu}\|_{L^4}  \\ &\hspace{.5cm}+ \|\Delta\rho\|\|\nabla\Bar{\varphi}\|_{L^4}\|\bm{\nu}\|_{L^4}
    \\ &\leq C\Big[\|\nabla\bm{\tau}\|\|\bm{\tau}\|\|\nabla\bm{\nu}\| + \|\nabla\textbf{v}\|\|\nabla\Bar{\textbf{u}}\|\|\nabla\bm{\nu}\| + \|\nabla\textbf{v}\|\|\nabla\bm{\nu}\| + \|\xi\|_{H^2}^{1/2}\|\xi\|^{1/2}\|\nabla\bm{\tau}\|\|\nabla\bm{\nu}\|  \\ &\hspace{.5cm}+ \|\xi\|_{V}^{3/2}\|\xi\|^{1/2} \|\Bar{\textbf{u}}\|_{H^2}^{1/2}\|\nabla\Bar{\textbf{u}}\|^{1/2}\|\nabla\bm{\nu}\| +  \|\rho\|_{V}^{1/2}\|\rho\|^{1/2}\|\Bar{\textbf{u}}\|_{H^2}^{1/2}\|\nabla\Bar{\textbf{u}}\|^{1/2}\|\nabla\bm{\nu}\|
    + \|\Bar{\varphi}\|_{H^3}^{3/4}\|\Bar{\varphi}\|_{L^\infty}^{1/4}\|\nabla\rho\|\|\nabla\bm{\nu}\|  \\ &\hspace{.5cm}+ \|\Delta\rho\|\|\Bar{\varphi}\|_{H^2}^{1/2}\|\nabla\Bar{\varphi}\|^{1/2}\|\nabla\bm{\nu}\|\Big]
\\ &\leq C\Big[\|\nabla\bm{\tau}\|\|\bm{\tau}\| + \|\nabla\Bar{\textbf{u}}\|\|\nabla\textbf{v}\|  + \|\nabla\textbf{v}\| + \|\xi\|_{H^2}\|\nabla\bm{\tau}\| + \|\xi\|_{V}^2 + \|\xi\|_{V}\|\xi\|\|\Bar{\textbf{u}}\|_{H^2}\|\nabla\Bar{\textbf{u}}\| \\ &\hspace{.5cm}+  \|\rho\|_{V} +\|\rho\|\|\Bar{\textbf{u}}\|_{H^2}\|\nabla\Bar{\textbf{u}}\| + \|\Bar{\varphi}\|_{H^3}^{3/4}\|\Bar{\varphi}\|_{L^\infty}^{1/4}\|\nabla\rho\|  + \|\Delta\rho\|\|\Bar{\varphi}\|_{H^2}^{1/2}\|\nabla\Bar{\varphi}\|^{1/2}\Big]\|\nabla\bm{\nu}\|
\end{align*}
Combining the regularity for the strong solution, the stability estimates \eqref{cd1}-\eqref{cd3} and \eqref{eq3.43} we can conclude that,
\begin{align*}
    \|\textbf{v}'\|_{L^2(0,T;H)} &\leq C\|h\|_{L^2(0,T;H)}^2.
\end{align*}
Consequently,
\begin{align}\label{eq2.38a}
    \frac{\|(\mathcal{S}(\Bar{U}+ h )- \mathcal{S}(\Bar{U})- (\psi_{h}, \textbf{w}_{h}))\|_{\mathcal{Z}}}{\|h\|_{\mathcal{U}}}  = \frac{\|(\rho, \textbf{v})\|_{\mathcal{Z}}}{\|h\|_{\mathcal{U}}} \leq \|h\|_{\mathcal{U}}\rightarrow 0, \hspace{.5cm}\text{ as }\|h\|_{\mathcal{U}}\rightarrow 0.
\end{align}
Hence the differentiability of the control to state operator follows. Moreover, from \eqref{eq2.38a}, we can conclude that the derivative of the control to state operator at $\Bar{U}$ in the direction of $h$, $\mathcal{S}'(\Bar{U})(h) = (\psi_h, \textbf{w}_h)$.
\end{proof}

\subsection{First order necessary optimality condition}

\begin{theorem} \label{theorem2.5} 
Let $\Bar{U}$ be an optimal control, and $(\Bar{\varphi}, \Bar{\textbf{u}})$ be the associated state. Then, the following inequality holds:  
\begin{align} \label{foc} 
\int\limits_{0}^{T} \sum\limits_{i=1}^{k} (\Bar{\varphi}(x_i,t) - \Phi^i(t)) \psi_{U-\Bar{U}}(x_i,t) + \int\limits_{0}^{T} \int_\Omega (\Bar{\textbf{u}} - \textbf{u}_d)\cdot \textbf{w}_{U-\Bar{U}} &+ \int_\Omega (\Bar{\textbf{u}}(x,T) - \textbf{u}_d(T))\cdot \textbf{w}_{U-\Bar{U}}(T) \nonumber\\ + \int\limits_{0}^{T} \int_\Omega (U - \Bar{U}) \Bar{U} &\geq 0, \hspace{.5cm} \forall U \in \mathcal{U}_{ad}.
\end{align}  
 Here, $(\psi_{U-\Bar{U}}, \textbf{w}_{U-\Bar{U}})$ is the solution of the linearized system \eqref{eq2.3}-\eqref{eq2.8} corresponding to the source term $U - \Bar{U}$.  
\end{theorem} 

\begin{proof}  
Define $j(U) = \mathcal{J}_1(\mathcal{S}(U), U)$. If $j$ attains a minimum value at $\Bar{U}$ in $\mathcal{U}_{ad}$, then the following optimality condition must hold for all $U \in \mathcal{U}_{ad}$:  
\begin{equation}\label{eq2.42}  
    j'(\Bar{U})(U - \Bar{U}) \geq 0.  
\end{equation}  
For example refer  \cite[Lemma 2.21]{TRO}. We denote $\mathcal{S}(\Bar{U}) = (\Bar{\varphi}, \Bar{\textbf{u}})$ and $\mathcal{S}(\Bar{U} + \beta U) = (\varphi_\beta, \textbf{u}_\beta)$. Observe that, 
\begin{align*}
    j(\Bar{U}+\beta U)- j(\Bar{U}) &= \frac{1}{2}\int\limits_{0}^{T}\sum\limits_{i=1}^{k}(\varphi_\beta(x_i,t)-\Phi^i(t))^2 - (\Bar{\varphi}(x_i,t)-\Phi^i(t))^2 
     +\frac{1}{2}\int\limits_{0}^{T}\int_\Omega|\textbf{u}_\beta-\textbf{u}_d|^2-|\Bar{\textbf{u}}-\textbf{u}_d|^2 \\&+\frac{1}{2}\int_\Omega|\textbf{u}_\beta(x,T)-\textbf{u}_d(T)|^2-|\Bar{\textbf{u}}(T)-\textbf{u}_d(T)|^2 + \frac{1}{2}\int\limits_{0}^{T}\int_\Omega(\Bar{U}+\beta U)^2-(\Bar{U})^2\\
    &= \frac{1}{2}\int\limits_{0}^{T}\sum\limits_{i=1}^{k}(\varphi_\beta(x_i,t)+ \Bar{\varphi}(x_i,t)-2\Phi^i(t))(\varphi_\beta(x_i,t)-\Bar{\varphi}(x_i,t))
     \\&+\frac{1}{2}\int\limits_{0}^{T}\int_\Omega (\textbf{u}_\beta+ \Bar{\textbf{u}}-2\textbf{u}_d)\cdot(\textbf{u}_\beta-\Bar{\textbf{u}})  + \frac{1}{2}\int\limits_{0}^{T}\int_\Omega\beta^2U^2+ 2\beta U\Bar{U}\\& +\frac{1}{2}\int_\Omega (\textbf{u}_\beta(x,T)+ \Bar{\textbf{u}}(x,T)-2\textbf{u}_d(T))\cdot(\textbf{u}_\beta(x,T)-\Bar{\textbf{u}}(x,T)),
\end{align*}
\begin{align*}
  \frac{j(\Bar{U}+\beta U)- j(\Bar{U})}{\beta} 
   &= \frac{1}{2}\int\limits_{0}^{T}\sum\limits_{i=1}^{k} (\varphi_\beta(x_i,t)+ \Bar{\varphi}(x_i,t)-2\Phi^i(t)) \frac{(\varphi_\beta(x_i,t)-\Bar{\varphi}(x_i,t))}{\beta}  \\& + \frac{1}{2}\int\limits_{0}^{T}\int_\Omega(\textbf{u}_\beta+ \Bar{\textbf{u}}-2\textbf{u}_d) \cdot\frac{(\textbf{u}_\beta-\Bar{\textbf{u}})}{\beta} + \frac{1}{2}\int\limits_{0}^{T}\int_\Omega\beta|U|^2+ \int\limits_{0}^{T}\int_\Omega U\Bar{U}\\& + \frac{1}{2}\int_\Omega(\textbf{u}_\beta(x,T)+ \Bar{\textbf{u}}(x,T)-2\textbf{u}_d(T)) \cdot\frac{(\textbf{u}_\beta(x,T)-\Bar{\textbf{u}}(x,T))}{\beta}.
\end{align*}

Using the continuous dependence of solution \eqref{cd1}, \eqref{cd2}and \eqref{cd3} we have, as $\beta \rightarrow 0$, $\varphi_\beta \rightarrow \Bar{\varphi}$ in $L^2(0,T;H^2(\Omega))\hookrightarrow L^2(0,T; C(\Bar{\Omega}))$. Similarly, $\textbf{u}_\beta\rightarrow \Bar{\textbf{u}}$ in $C([0,T];\mathbb{G}_{div})$ as $\beta \rightarrow 0$. Furthermore, by proposition \ref{theorem3.4}, it follows that,
\begin{equation*}
    \frac{\|(\mathcal{S}(\Bar{U}+ \beta U)- \mathcal{S}(\Bar{U})- (\psi_{\beta U}, \textbf{w}_{\beta U}))\|_{\mathcal{Z}}}{\|\beta U\|_{\mathcal{U}}} =  \frac{\|( (\varphi_\beta, \textbf{u}_\beta)- (\Bar{\varphi},\Bar{\textbf{u}})- (\psi_{\beta U}, \textbf{w}_{\beta U}))\|_{\mathcal{Z}}}{\|\beta U\|_{\mathcal{U}}} \leq \|\beta U\|_{\mathcal{U}},
\end{equation*}
where $(\psi_{\beta U}, \textbf{w}_{\beta U})$ solves the linearized system around $\Bar{U}$, corresponding to the source term $\beta U$. Consequently, $( \frac{(\varphi_\beta-\Bar{\varphi})}{\beta}, \frac{(\textbf{u}_\beta-\Bar{\textbf{u}})}{\beta}) 
\rightarrow (\psi_U, \textbf{w}_U ) $ in 
$L^2(0,T;H^2(\Omega)) \times L^\infty(0,T; \mathbb{G}_{div})\cap H^1(0,T; \mathbb{V}_{div}') $ and we use the embedding $ L^2(0,T;H^2(\Omega)) \times L^\infty(0,T; \mathbb{G}_{div})\cap H^1(0,T; \mathbb{V}_{div}')   \hookrightarrow L^2(0,T; C(\Bar{\Omega})) \times C([0,T]; \mathbb{G}_{div})$ to pass to the limit as $\beta\rightarrow 0$. Thus we get,
\begin{align*}
   j'(\Bar{U})(U) &=\int\limits_{0}^{T}\sum_{i=1}^{k}(\Bar{\varphi}(x_i,t)- \Phi^i(t))\psi_{U}(x_i,t)  +\int\limits_{0}^{T}\int_\Omega(\Bar{\textbf{u}}-\textbf{u}_d)\cdot\textbf{w}_{U} + \int_\Omega (\Bar{\textbf{u}}(x,T) - \textbf{u}_d(T))\cdot \textbf{w}_{U}(T)+ \int\limits_{0}^{T}\int_\Omega U\Bar{U}. 
\end{align*}
This expression along with \eqref{eq2.42} gives the result. 
\end{proof}

\subsection{Characterisation of an optimal control}
In this section we will introduce an adjoint system which will be used to characterise the optimal control. First we  discuss the wellposedness of the adjoint system.

\subsubsection{The Adjoint system}
Let $(\bar{\varphi}, \bar{\textbf{u}})$ denote the optimal state associated with an optimal control $\bar{U}$. Consider the adjoint system given by,
\begin{align}
    -\xi'-\Bar{\textbf{u}}\cdot\nabla\xi + 2\eta'(\Bar{\varphi})D\Bar{\textbf{u}}:D\textbf{v}+\Delta(m(\Bar{\varphi})(\Delta\xi)) -m(\Bar{\varphi})F''(\Bar{\varphi})\Delta\xi & \nonumber\\+ \Delta(m'(\Bar{\varphi})\nabla\Bar{\varphi}\cdot\nabla\xi) - m'(\Bar{\varphi})\nabla(\Delta\Bar{\varphi})\cdot\nabla\xi -\nabla\cdot(\textbf{v}\Delta\Bar{\varphi})+ \Delta(\textbf{v}\cdot\nabla\Bar{\varphi}) &= R,\text{ in } Q, \label{eq2.44}\\
    -\textbf{v}'-(\Bar{\textbf{u}}\cdot\nabla)\textbf{v} + (\textbf{v}\cdot\nabla)\Bar{\textbf{u}} + \nabla q -\nabla\cdot(2\eta(\Bar{\varphi})D\textbf{v}) + \xi\nabla\Bar{\varphi} &= (\Bar{\textbf{u}}-\textbf{u}_d), \text{ in } Q,\label{eq2.45}\\
    \nabla\cdot \textbf{v} &= 0, \text{ in } Q,\label{eq2.46}\\
      \textbf{v}&= \textbf{0}, \text{ on } \Sigma,\label{eq2.47}\\
    \frac{\partial\xi}{\partial\textbf{n}} = \frac{\partial\Delta\xi}{\partial\textbf{n}} &= 0, \text{ on } \Sigma,\label{eq2.48}\\
    \xi(x,T)= 0 , \hspace{.5cm} \textbf{v}(x,T) &= (\Bar{\textbf{u}}(T)-\textbf{u}_d(T)), \text{ in } \Omega.\label{eq2.49}
\end{align}
Here, $R := \sum\limits_{i=1}^{k}(\Bar{\varphi}(x_i,t)-\Phi^i(t))$. Observe that $\Phi^i$ and $u_d$ are target functions as defined in the cost $\mathcal{J}_1$ given by \eqref{eqJ1}.
Note that we define the solution of \eqref{eq2.44}-\eqref{eq2.49} using the transposition method due to  low regularity of the source term, $R$.  

\begin{definition}[Solution of the adjoint system by transposition]\label{def2}
We say $(\xi, \textbf{v}) \in L^2(0,T;H) \times L^2(0,T; \mathbb{G}_{div})$ solve the system \eqref{eq2.44}-\eqref{eq2.49} for $R = \sum\limits_{i=1}^{k}(\Bar{\varphi}(x_i,t)-\Phi^i(t))$, if it satisfies,
\begin{align}\label{eq2.50}
    \int\limits_{0}^{T}\int_\Omega \xi U = \int\limits_{0}^{T}\sum\limits_{i=1}^{k}(\Bar{\varphi}(x_i,t)-\Phi^i(t))\psi(x_i, t) dt + \int\limits_{0}^{T}\int_\Omega (\Bar{\textbf{u}}-\textbf{u}_d)\cdot\textbf{w}dxdt + \int_\Omega (\Bar{\textbf{u}}(x,T)-\textbf{u}_d(x,T))\cdot\textbf{w}(x,T)dx,
\end{align}
for every $U\in L^2(0,T;H)$, where $(\psi, \textbf{w})$ is the solution of the system \eqref{eq2.3}-\eqref{eq2.8} corresponding to the source term $U$ and initial data $(\psi_0, \textbf{w}_0) = (0,\textbf{0})$.
\end{definition}

The main idea of the proof is to approximate the term $R = \sum\limits_{i=1}^{k}(\Bar{\varphi}(x_i,t)-\Phi^i(t))$ by functions $R_\epsilon$ in $L^2(0, T; H)$ and passing to the limit to obtain a solution of the system.  First we state a useful lemma about the existence result for a regular source term.

\begin{lemma}\label{lemma2.6}
     Consider the system \eqref{eq2.44}-\eqref{eq2.49} with the forcing term in \eqref{eq2.44} given by $\tilde R$, where $ \tilde R \in L^2(0,T;H)$, and terminal data, $\xi (x, T) = \tilde{\xi_{T}} \in H$. Assume that $\Bar{\textbf{u}}-\textbf{u}_d \in L^2(0,T; \mathbb{G}_{div})$,  and $\Bar{\textbf{u}}(T)-\textbf{u}_d(T) \in \mathbb{G}_{div}$. In addition, assumptions, \emph{\textbf{[A1]}}-\emph{\textbf{[A5]}} are satisfied. Then there exist a weak solution, $(\xi, \textbf{v})$ to the system \eqref{eq2.44}-\eqref{eq2.49} which satisfies,
    \begin{align*}
        \xi &\in L^2(0,T;V) \cap L^\infty(0,T;H) \cap H^1(0,T; V'),\\
        \textbf{v} &\in L^2(0,T;\mathbb{V}_{div}) \cap L^\infty(0,T;\mathbb{G}_{div}) \cap H^1(0,T; \mathbb{V}_{div}'),
    \end{align*}
    and the following weak formulation.
\begin{align}
    -(\xi',\chi) -(\Bar{\textbf{u}}\cdot\nabla\xi,\chi) + 2(\eta'(\Bar{\varphi})D\Bar{\textbf{u}}:D\textbf{v}, \chi)+ (\Delta(m(\Bar{\varphi})(\Delta\xi)),\chi) -(m(\Bar{\varphi})F''(\Bar{\varphi})\Delta\xi&,\chi) \nonumber\\+ (\Delta(m'(\Bar{\varphi})\nabla\Bar{\varphi}\cdot\nabla\xi),\chi) - (m'(\Bar{\varphi})\nabla(\Delta\Bar{\varphi})\cdot\nabla\xi,\chi) -(\nabla\cdot(\textbf{v}\Delta\Bar{\varphi}),\chi)+ (\Delta(\textbf{v}\cdot\nabla\Bar{\varphi}),\chi) &= (\tilde R, \chi),\label{eq2.51}\\
    -(\textbf{v}', \bm{\nu})-((\Bar{\textbf{u}}\cdot\nabla)\textbf{v},\bm{\nu}) + ((\textbf{v}\cdot\nabla)\Bar{\textbf{u}},\bm{\nu}) + 2(\eta(\Bar{\varphi})D\textbf{v},D\bm{\nu}) + (\xi\nabla\Bar{\varphi},\bm{\nu}) &= (\Bar{\textbf{u}}-\textbf{u}_d, \bm{\nu}),\label{eq2.52}
\end{align}
for all $\chi\in V$ and $\bm{\nu}\in \mathbb{V}_{div}$. Furthermore,
\begin{align}
    \xi(T) = \tilde{\xi_{T}}, \hspace{.5cm} \textbf{v}(T)= \Bar{\textbf{u}}(T)-\textbf{u}_d(T), \hspace{.5cm} \text{ a.e in } \Omega.
\end{align}
\end{lemma}
The proof can be done using the Galerkin approximation method. Since the system is linear, the uniqueness of the solution follows by similar arguments as in the proof of theorem \ref{theorem2}.

\begin{theorem}[\textit{Existence of solution to the adjoint system}]\label{theorem3.6}
Let assumptions \emph{\textbf{[A3]}}-\emph{\textbf{[A5]}} are satisfied. Then there exists a solution to the system \eqref{eq2.44}-\eqref{eq2.49} in the sense of definition \ref{def2}, where the source term,  $R = \sum\limits_{i=1}^{k}(\Bar{\varphi}(x_i,t)-\Phi^i(t))$.
\end{theorem}
\begin{proof}
    The proof is based on an approximation technique. Let $R_\epsilon := \sum\limits_{i=1}^{k}\frac{1}{|B(x_i,\epsilon)|}\chi_{B(x_i,\epsilon)}(\Bar{\varphi}(x,t)-\Phi^i(t))$. we choose an $\epsilon >0$ sufficiently small so that the balls $\{ B(x_i, \epsilon)\}_i, 1\leq i\leq k$ do not overlap. Observe that $R_\epsilon\in L^2(0,T;H)$, therefore, by Lemma \ref{lemma2.6}, there exists a weak solution to the adjoint system corresponding to $R_\epsilon$, say $(\xi_\epsilon, \textbf{v}_\epsilon)$.

    Let $U\in \mathcal{U}$ and $(\psi, \textbf{w})$ be the solution of the linearized system \eqref{eq2.3}-\eqref{eq2.8} corresponding to the source term $U-\Bar{U}$ and the initial data, $(\psi_0, \textbf{w}_0) = (0, \textbf{0})$. For a.e. $t \in [0,T]$,  $\psi(t) \in V$ and $\textbf{w}(t) \in \mathbb{V}_{div} $. Using \eqref{eq2.51}-\eqref{eq2.52}, for $\chi = \psi$, $\bm{\nu} = \textbf{w}$ and adding both the equations we get,
\begin{align*}
    -(\xi_{\epsilon}',\psi) -(\Bar{\textbf{u}}\cdot\nabla\xi_{\epsilon},\psi) + 2(\eta'(\Bar{\varphi})D\Bar{\textbf{u}}:D\textbf{v}_{\epsilon}, \psi)+ (\Delta(m(\Bar{\varphi})\Delta\xi_{\epsilon}),\psi) -(m(\Bar{\varphi})F''(\Bar{\varphi})\Delta\xi_{\epsilon},\psi)& \nonumber\\+ (\Delta(m'(\Bar{\varphi})\nabla\Bar{\varphi}\cdot\nabla\xi_{\epsilon}),\psi) - (m'(\Bar{\varphi})\nabla(\Delta\Bar{\varphi})\cdot\nabla\xi_{\epsilon},\psi) -(\nabla\cdot(\textbf{v}_{\epsilon}\Delta\Bar{\varphi}),\psi) + (\Delta(\textbf{v}_{\epsilon}\cdot\nabla\Bar{\varphi}),\psi)& \\
    -(\textbf{v}_{\epsilon}', \textbf{w})-((\Bar{\textbf{u}}\cdot\nabla)\textbf{v}_{\epsilon},\textbf{w}) + ((\textbf{v}_{\epsilon}\cdot\nabla)\Bar{\textbf{u}},\textbf{w}) + 2(\eta(\Bar{\varphi})D\textbf{v}_{\epsilon},D\textbf{w}) + (\xi_{\epsilon}\nabla\Bar{\varphi},\textbf{w}) &= (R_\epsilon, \psi)+ (\Bar{\textbf{u}}-\textbf{u}_d, \textbf{w})
\end{align*}
Applying integration by parts and rearranging, we deduce,
\begin{align*}
    (\xi_{\epsilon}&,\psi') + (\xi_{\epsilon},\Bar{\textbf{u}}\cdot\nabla\psi) + (\textbf{v}_{\epsilon}, \nabla\cdot(2\eta'(\Bar{\varphi})\psi D\Bar{\textbf{u}}))+ (\xi_{\epsilon},\Delta(m(\Bar{\varphi})\Delta\psi)) -(\xi_{\epsilon},\Delta(m(\Bar{\varphi})F''(\Bar{\varphi})\psi)) \nonumber\\&- (\xi_{\epsilon},\nabla\cdot(m'(\Bar{\varphi})\nabla\Bar{\varphi}\Delta\psi)) + (\xi_{\epsilon},\nabla\cdot(m'(\Bar{\varphi})\nabla(\Delta\Bar{\varphi})\psi)) +(\textbf{v}_{\epsilon},(\Delta\Bar{\varphi})\nabla\psi) + (\textbf{v}_{\epsilon},(\Delta\psi)\nabla\Bar{\varphi}) 
    +(\textbf{v}_{\epsilon}, \textbf{w}')\\ &- \int_\Omega (\Bar{\textbf{u}}(T)-\textbf{u}_d(T))\textbf{w}(T)dx+ (\textbf{v}_{\epsilon}, (\Bar{\textbf{u}}\cdot\nabla)\textbf{w}) + (\textbf{v}_{\epsilon},(\textbf{w}\cdot\nabla)\Bar{\textbf{u}}) - (\textbf{v}_{\epsilon},\nabla\cdot(2\eta(\Bar{\varphi})D\textbf{w})) + (\xi_{\epsilon},\textbf{w}\cdot\nabla\Bar{\varphi}) \\&= (R_\epsilon, \psi)+ (\Bar{\textbf{u}}-\textbf{u}_d, \textbf{w})
\end{align*}
After rearranging and applying Equations \eqref{eq2.3} and \eqref{eq2.4}, we obtain,
\begin{align*}
   &\Big(\xi_{\epsilon},\psi' + \Bar{\textbf{u}}\cdot\nabla\psi + \textbf{w}\cdot\nabla\Bar{\varphi}+ \Delta\big(m(\Bar{\varphi})(\Delta\psi - F''(\Bar{\varphi})\psi)\big) - \nabla\cdot(m'(\Bar{\varphi})\nabla\Bar{\varphi}\Delta\psi) + \nabla\cdot(m'(\Bar{\varphi})\nabla(\Delta\Bar{\varphi})\psi) \Big) \\
     & \hspace{.5cm}+ \big(\textbf{v}_{\epsilon}, \textbf{w}'+ (\Bar{\textbf{u}}\cdot\nabla)\textbf{w}+ (\textbf{w}\cdot\nabla)\Bar{\textbf{u}} - \nabla\cdot(2\eta(\Bar{\varphi})D\textbf{w}) - \nabla\cdot(2\eta'(\Bar{\varphi})\psi D\Bar{\textbf{u}}) + (\Delta\Bar{\varphi})\nabla\psi + (\Delta\psi)\nabla\Bar{\varphi} \big)  \\ 
     &= \big(\xi_{\epsilon}, U-\bar{U} \big) \\
     &=(R_\epsilon, \psi)+ (\Bar{\textbf{u}}-\textbf{u}_d,  \textbf{w}) + \int_\Omega (\Bar{\textbf{u}}(T)-\textbf{u}_d(T))\textbf{w}(T)dx
\end{align*}
It follows that,
\begin{align}\label{eq2.54}
     (\xi_\epsilon, U-\bar{U})= (R_\epsilon, \psi)+ (\Bar{\textbf{u}}-\textbf{u}_d,  \textbf{w}) + \int_\Omega (\Bar{\textbf{u}}(T)-\textbf{u}_d(T))\textbf{w}(T)dx
\end{align}
Besides,
\begin{align*}
     \Big|\int\limits_{0}^{T}(\xi_\epsilon, U-\bar{U}) dt \Big| &\leq  \int\limits_{0}^{T}|(R_\epsilon ,\psi)|dt + \int\limits_{0}^{T}|(\Bar{\textbf{u}}-\textbf{u}_d, \textbf{w})| dt  + |(\Bar{\textbf{u}}(T)-\textbf{u}_d(T), \textbf{w}(T))|
     \\&\leq \sum\limits_{i=1}^{k}\frac{1}{|B(x_i,\epsilon)|}\int\limits_{0}^{T}\int_\Omega|\chi_{B(x_i,\epsilon)}(\Bar{\varphi}(x,t)-\Phi^i(t))\psi|dxdt + \int\limits_{0}^{T}\|\Bar{\textbf{u}}-\textbf{u}_d\|\|\textbf{w}\| \\&\hspace{.5cm} + \|\Bar{\textbf{u}}(x,T)-\textbf{u}_d(T)\|\|\textbf{w}(x,T)\|\\
     &\leq \sum\limits_{i=1}^{k}\int\limits_{0}^{T}\|\Bar{\varphi}(x,t)-\Phi^i(t)\|_{L^\infty(\Omega)}\|\psi\|_{L^\infty(\Omega)} + \int\limits_{0}^{T}\|\Bar{\textbf{u}}-\textbf{u}_d\|\|\textbf{w}\| \\&\hspace{.5cm} + \|\Bar{\textbf{u}}(x,T)-\textbf{u}_d(T))\|_{L^2(\Omega)}\|\textbf{w}(x,T)\|_{L^2(\Omega)}
\end{align*}
The last inequality is obtained by combining the regularity results for $(\psi, \textbf{w})$ and the embedding, $L^2(0,T;H^2(\Omega))\hookrightarrow L^2(0,T; C(\Bar{\Omega}))$. Now using the energy estimate \eqref{eq2.18} and \eqref{eq3.25}, we have, 
\begin{align*}
     \int\limits_{0}^{T} |(\xi_\epsilon, U-\bar{U})| dt 
     &\leq C\Big( \sum\limits_{i=1}^{k}\int\limits_{0}^{T}\|\Bar{\varphi}(x,t)-\Phi^i(t)\|_{L^\infty(\Omega)}^2 + \int\limits_{0}^{T}\|\Bar{\textbf{u}}-\textbf{u}_d\|^2 \\&\hspace{.5cm}+ \sum\limits_{i=1}^{k}\|\Bar{\varphi}(x,T)-\Phi^i(T)\|_{L^\infty(\Omega)} + \|\Bar{\textbf{u}}(x,T)-\textbf{u}_d(T))\|_{L^2(\Omega)} \Big) \|U-\bar{U}\|_{L^2(0,T;H)}
\end{align*}
which holds for all $U\in L^2(0,T;V)$. This implies $\|\xi_{\epsilon}\|_{L^2(0,T;H)}\leq C$ , where $C$ is a positive constant independent of $\epsilon$.  
To derive estimates for $\textbf{v}_\epsilon$, we use a test function $\bm{\nu}= \textbf{v}_\epsilon$ in \eqref{eq2.52}, for $\epsilon >0$.We get,
\begin{align*}
   -\frac{1}{2}\frac{d}{dt}\|\textbf{v}_{\epsilon}\|^2- ((\Bar{\textbf{u}}\cdot\nabla)\textbf{v}_{\epsilon},\textbf{v}_{\epsilon}) + ((\textbf{v}_{\epsilon}\cdot\nabla)\Bar{\textbf{u}},\textbf{v}_{\epsilon}) + 2(\eta(\Bar{\varphi})D\textbf{v}_{\epsilon},D\textbf{v}_{\epsilon}) + (\xi\nabla\Bar{\varphi},\textbf{v}_{\epsilon}) &= (\Bar{\textbf{u}}-\textbf{u}_d, \textbf{v}_{\epsilon}) 
\end{align*}
Estimating each of its terms by applying H\"older's, Ladyzhenskaya, Gagliardo- Nirenberg and Young's inequalities, we obtain,
\begin{align*}
    ((\Bar{\textbf{u}}\cdot\nabla)\textbf{v}_{\epsilon},\textbf{v}_{\epsilon}) &= 0\\
    2(\eta(\Bar{\varphi})D\textbf{v}_{\epsilon},D\textbf{v}_{\epsilon}) &\geq 2\eta_0\|\nabla\textbf{v}_\epsilon\|^2\\
    |((\textbf{v}_{\epsilon}\cdot\nabla)\Bar{\textbf{u}},\textbf{v}_{\epsilon})| &\leq \|\nabla\Bar{\textbf{u}}\|\|\textbf{v}_{\epsilon}\|_{L^4}^2 \leq \|\nabla\Bar{\textbf{u}}\|\|\textbf{v}_{\epsilon}\|\|\nabla\textbf{v}_{\epsilon}\|\\ &\leq \frac{\eta_0}{2}\|\nabla\textbf{v}_{\epsilon}\|^2 + C\|\nabla\Bar{\textbf{u}}\|^2\|\textbf{v}_{\epsilon}\|^2\\
    |(\xi\nabla\Bar{\varphi},\textbf{v}_{\epsilon})| &\leq \|\xi\|\|\nabla\Bar{\varphi}\|_{L^4}\|\textbf{v}_{\epsilon}\|_{L^4} \leq \|\nabla\Bar{\varphi}\|^{1/2}\|\Bar{\varphi}\|_{H^2}^{1/2}\|\xi\|\|\nabla\textbf{v}_{\epsilon}\|\\& \leq \frac{\eta_0}{2}\|\nabla\textbf{v}_{\epsilon}\|^2 + C\|\nabla\Bar{\varphi}\|\|\Bar{\varphi}\|_{H^2}\|\xi\|^2 \\
    |(\Bar{\textbf{u}}-\textbf{u}_d, \textbf{v}_{\epsilon})| &\leq \|\Bar{\textbf{u}}-\textbf{u}_d\|\|\textbf{v}_{\epsilon}\| \leq \|\Bar{\textbf{u}}-\textbf{u}_d\|^2 + \|\textbf{v}_{\epsilon}\|^2
\end{align*}
We deduce from the above calculations that,
\begin{align*}
    -\frac{1}{2}\frac{d}{dt}\|\textbf{v}_{\epsilon}\|^2 + \eta_0\|\nabla\textbf{v}_\epsilon\|^2 \leq C(\|\nabla\Bar{\varphi}\|\|\Bar{\varphi}\|_{H^2}\|\xi\|^2 + \|\Bar{\textbf{u}}-\textbf{u}_d\|^2) +  C( 1+ \|\nabla\Bar{\textbf{u}}\|^2) \|\textbf{v}_{\epsilon}\|^2 
\end{align*}
Now integrating over the interval, $[t,T]$ and applying Gronwall's inequality in backward form, we obtain,
\begin{align}\label{eq2.55}
    \sup_{t\in [0,T]}\|\textbf{v}_{\epsilon}(t)\|^2 + \eta_0\int\limits_{0}^{T}\|\textbf{v}_\epsilon\|_{\mathbb{V}_{div}}^2 \leq C\Big[\|\Bar{\textbf{u}}(x,T)-\textbf{u}_d(T)\|_{\mathbb{G}_{div}}^2+ \int\limits_{0}^{T}(\|\nabla\Bar{\varphi}\|\|\Bar{\varphi}\|_{H^2}\|\xi\|^2 + \|\Bar{\textbf{u}}-\textbf{u}_d\|^2)dt\Big] 
\end{align}
It follows that $\|\textbf{v}_\epsilon\|_{L^2(0,T;\mathbb{V}_{div})\cap L^\infty(0,T;\mathbb{G}_{div})} \leq C$, where $C$ is a constant independent of $\epsilon$. Using the uniform estimates derived, we have, renamed subsequences, 
\begin{align*}
    \xi_\epsilon &\rightharpoonup \xi, \hspace{1cm}\text{in } L^2(0,T;H),\\
    \textbf{v}_\epsilon &\overset{\ast}{\rightharpoonup} \textbf{v}, \hspace{1cm}\text{in } L^\infty(0,T;\mathbb{G}_{div}),\\
    \textbf{v}_\epsilon &\rightharpoonup \textbf{v}, \hspace{1cm}\text{in } L^2(0,T;\mathbb{V}_{div}).
\end{align*}
Since $\Bar{\varphi}, \psi \in L^2(0,T; H^2(\Omega)) \hookrightarrow L^2(0,T;C(\Bar{\Omega}))$, we have, for each $i$, $1\leq i\leq k$, the following distributional convergences as $\epsilon \rightarrow 0$.  
\begin{align*}
    \frac{1}{|B(x_i,\epsilon)|}\int\limits_{0}^{T}\int_\Omega|\chi_{B(x_i,\epsilon)}(\Bar{\varphi}(x,t)-\Phi^i(t))\psi| &\rightarrow \int\limits_{0}^{T}(\Bar{\varphi}(x_i,t)-\Phi^i(t))\psi(x_i,t)dt, 
\end{align*}
Therefore we can pass to the limit in \eqref{eq2.54}, we get,
\begin{align}\label{eq3.56}
    \int\limits_{0}^{T}\int_\Omega \xi (U-\bar{U}) &= \int\limits_{0}^{T}\sum\limits_{i=1}^{k}(\Bar{\varphi}(x_i,t)-\Phi^i(t))\psi(x_i, t) dt + \int\limits_{0}^{T}\int_\Omega (\Bar{\textbf{u}}-\textbf{u}_d)\cdot\textbf{w}dxdt  + \int_\Omega (\Bar{\textbf{u}}(x,T)-\textbf{u}_d(x,T))\cdot\textbf{w}(x,T)dx,
\end{align}
hold true for all $U\in L^2(0,T;V)$. Hence the existence of a solution of \eqref{eq2.44}-\eqref{eq2.49} in the sense of definition \ref{def2}. Moreover, the uniqueness of the solution follows due to linearity of the system.
\end{proof}

\begin{theorem}\label{theorem3.8}
Let $\Bar{U}$ be an optimal control and $(\Bar{\varphi}, \Bar{\textbf{u}})$ be the associated state. Further, $(\xi, \textbf{v})$ be the adjoint variable as defined in definition \ref{def2} corresponding to $(U-\Bar{U})$. Then the following variational inequality holds true. 
\begin{align} \label{choc} 
\int\limits_{0}^{T} \int_\Omega (\xi+ \Bar{U})(U - \Bar{U}) \geq 0, \hspace{.5cm} \forall U \in \mathcal{U}_{ad}.
\end{align}  
\end{theorem}
\begin{proof}
    This results follows directly by comparing the optimality condition, \eqref{foc} and \eqref{eq3.56}. 
\end{proof}

\section{ Terminal Time Pointwise Tracking   Optimal Control }

We introduce a new cost functional which tracks the concentration pointwise at terminal time.
\begin{align}\label{J2}
\mathcal{J}_2(\varphi,\textbf{u}, U) &= \frac{1}{2}\sum\limits_{i=1}^{k}\int\limits_{0}^{T}(\varphi(x_i,t)-\Phi^i(t))^2 dt +\frac{1}{2}\int\limits_{0}^{T}\int_\Omega(\textbf{u}-\textbf{u}_d)^2 dxdt + \frac{1}{2}\sum\limits_{i=1}^{k}(\varphi(x_i,T)-\Phi^i(T))^2 \nonumber\\ &\hspace{.5cm}+ \frac{1}{2}\int_\Omega(\textbf{u}(x,T)-\textbf{u}_d(x,T))^2 dx
    + \frac{1}{2}\int\limits_{0}^{T}\int_\Omega|U|^2 dxdt,
\end{align}
where $\{ \Phi^i(t)\}_{1\leq i\leq k}$ denotes the desired trajectory at prescribed points $x_i, 1\leq i\leq k$ and $\textbf{u}_d$ is the desired velocity. Particularly, we focus on a constant trajectory, namely $\Phi^i(t)=\varphi^i$ for $1\leq i\leq k$, with $\varphi^i\in\mathbb{R}$.
We optimize above cost functional subject to the state system \eqref{eq8}–\eqref{eq14} and the set of admissible controls,
\begin{equation*}
\mathcal{U}_{ad}:= \{ U\in L^2(0,T;V\cap L^\infty(\Omega)): U_1(x,t) \leq U(x,t) \leq U_2(x,t), \text{ for a.e. } (x,t)\in \Omega\times(0,T)\}.
\end{equation*}

The corresponding optimal control problem is formulated as
$$\min_{U\in \mathcal{U}_{ad}}\{ \mathcal{J}_2(\varphi,\textbf{u},U) | \; (\varphi, \textbf{u}) \text{is the strong solution of \eqref{eq8} -\eqref{eq14} with control U} \}\hspace{1cm} \textbf{[OCP\RomanNumeralCaps{2}]}$$ 
This problem is a variant of the optimal control problem \textbf{[OCP\RomanNumeralCaps{1}]} introduced in Section 3 with an additional terminal time pointwise tracking term.
Note that for every $U\in\mathcal{U}_{ad}$, the existence and uniqueness of a strong solution, $(\varphi,\mathbf{u})$ to \eqref{eq8}–\eqref{eq14} follows from Propositions \ref{prop2} and Proposition \ref{prop3}. Moreover, by Aubin–Lions lemma, $\varphi \in C([0,T];C^{0,\gamma}(\Bar{\Omega}))$ for some $0< \gamma < 1$, ensuring that the pointwise tracking terms are well defined and thus $\mathcal{J}_2$ is well defined. As in Section 3, we define the control to state map, $\mathcal{S} : \mathcal{U} \rightarrow\mathcal{V}$, and reformulate \textbf{[OCP\RomanNumeralCaps{2}]} as follows.
$$\min_{U\in \mathcal{U}_{ad}}\{ \mathcal{J}_2(\mathcal{S}(U),U) \}.$$
 
We follow the similar strategy as in section 3, first we prove existence of optimal control, followed by differentiability of control to state map and then characterization of the control.

\begin{theorem}[Existence of an optimal control]\label{EOC2}
  Let $\varphi_0 \in H^3(\Omega)$, $\textbf{u}_0 \in \mathbb{V}_{div}$, $F$ satisfies \eqref{F5} and assumptions \emph{\textbf{[A1]}-\textbf{[A5]}} are satisfied. . Then the optimal control problem, \emph{\textbf{[OCP\RomanNumeralCaps{2}]}} admits a solution.  
\end{theorem}
\begin{proof}
The analysis follows the same strategy as in theorem \ref{theorem1}: Since the cost functional $\mathcal{J}_2$ is bounded from below, define,
$$j= \inf_{U\in \mathcal{U}_{ad}} \mathcal{J}_2(\varphi, \textbf{u}, U).$$ 
Let $\{U_n\}\subset\mathcal{U}_{ad}$ be a minimizing sequence such that $\mathcal{J}_2(\varphi_n,\mathbf{u}_n,U_n)\rightarrow j$, where $(\varphi_n,\mathbf{u}_n)$ denotes the strong solution corresponding to $U_n$. Using the boundedness of $\mathcal{U}_{ad}$ and the a priori estimates for the state system, we can extract a subsequence such that
$U_n \rightharpoonup \Bar{U},\,\,\text{ in } \mathcal{U}_{ad}$, and the associated state variables satisfy $(\varphi_n,\mathbf{u}_n)\rightarrow (\Bar{\varphi},\Bar{\mathbf{u}})$ in appropriate weak and strong topologies, where $(\Bar{\varphi},\Bar{\mathbf{u}})$ denotes the state corresponding to control $\Bar{U}$. In particular, by Aubin–Lions Lemma, $\varphi_n \to \Bar{\varphi} \,\,\text{ strongly in } C([0,T];C(\Bar{\Omega}))$.

Now to prove $j$ is the minimum value of $\mathcal{J}_2$, consider the cost functional,
\begin{align*}
   \mathcal{J}_2(\varphi,\textbf{u}, U) &=  \underbrace{\frac{1}{2}\sum\limits_{i=1}^{k}(\varphi(x_i,T)-\varphi^i(T))^2}_{\mathcal{J}^*} + \mathcal{J}_1(\varphi,\textbf{u}, U)
\end{align*}
Using the strong convergence of $\varphi_n$ in $ C([0,T];C(\Bar{\Omega}))$, which follows from the embedding $L^\infty(0,T; H^2(\Omega))\cap L^2(0,T;V') \hookrightarrow C([0,T];C(\Bar{\Omega}))$, we get that, $\underset{k\rightarrow \infty}{\lim} \mathcal{J}^*(\varphi_k, \textbf{u}_k, U_k) = \mathcal{J}^*(\Bar{\varphi}, \Bar{\textbf{u}}, \Bar{U})$. Further using the arguments as in the proof of theorem \ref{theorem1} for $\mathcal{J}_1(\varphi,\textbf{u}, U)$, it follows that, 
\begin{equation*}
    j\leq \mathcal{J}_2(\Bar{\varphi}, \Bar{\textbf{u}}, \Bar{U}) \leq \lim_{k\rightarrow \infty} \mathcal{J}_2(\varphi_k, \textbf{u}_k, U_k) = j
\end{equation*}
Hence the existence of an optimal control.
\end{proof}

\subsection{Differentiability of the control to state operator}

To deduce the necessary optimality conditions for \textbf{[OCP\RomanNumeralCaps{2}]}, we need to study the differentiability properties of the map $\mathcal{S}$. By exploiting theorem \ref{theorem3.4}, we have, $S$ is Fr\'etchet differentiable as a map from $\mathcal{U}$ to $\mathcal{Z}$. Due to the low regularity of pointwise tracking terms in the cost, we require a stronger differentiability result: we prove in the next subsection that the map $\mathcal{S} : \mathcal{U} \rightarrow \mathcal{V}$ is Fr\'etchet differentiable. To this end, we study the strong solution of the linearised system using uniform a priori estimates and compactness results.

Furthermore, due to the low regularity of source term in the adjoint system we study the solution of the system using the transposition method. The solution of the adjoint system using transposition necessitates to prove a higher regularity of the solution of the linearised system.

Under higher regularity of $F$ and  the initial data $ \psi_0$, we prove  the following theorem for existence of a strong solution to the linearised system.

\begin{theorem}[\textit{Existence of a strong solution to the Linearised system}]\label{theorem3.3}
Let all the hypotheses of Theorem \ref{theorem2} be satisfied. Further, assume $\psi_0\in H^2(\Omega)$ and $F\in C^4(\mathbb{R})$. Then a solution of the system \eqref{eq2.3}-\eqref{eq2.8} has the following additional regularity.
  \begin{align*}
      \psi\in L^\infty(0,T;H^2) \cap L^2(0,T;H^4).
  \end{align*}
\end{theorem}
\begin{proof}
We will derive a-priori estimates for the solution of finite dimensional problems as a continuation of the previous proof of Theorem \ref{theorem2}. Later we pass to the limit as $n\rightarrow \infty$ to prove the result. We use a test function $\Delta^2\psi_n$ in \eqref{eq2.9}. Hence, 
\begin{align}\label{eq2.030}
     \frac{1}{2}\frac{d}{dt}\|\Delta\psi_n\|^2 + (\Bar{\textbf{u}}\cdot\nabla\psi_n,\Delta^2\psi_n)  + (\textbf{w}_n\cdot\nabla\Bar{\varphi},\Delta^2\psi_n) &= \big(\nabla\cdot(m(\Bar{\varphi})\nabla\Tilde{\mu}_n), \Delta^2\psi_n\big)+ \big(\nabla\cdot( m'(\Bar{\varphi})\psi_n\nabla\Bar{\mu}, \Delta^2\psi_n\big) \nonumber\\&\hspace{.5cm}+ (h,\Delta^2\psi_n),
\end{align}
Where $\tilde{\mu}_n := (-\Delta \psi_n + F''(\Bar{\varphi})\psi_n)$ and $\Bar{\mu} := (-\Delta \Bar{\varphi}+ F'(\Bar{\varphi}))$. Observe that, from proposition \ref{prop2} and Theorem \ref{theorem2}, it follows that $\Tilde{\mu}_n \in L^2(0,T;V)$ and $\Bar{\mu}\in L^2(0,T;H^2(\Omega))\cap L^\infty(0,T;V)$.
We estimate each of the terms of \eqref{eq2.030} by applying  H\"older's, Gagliardo- Nirenberg, Poincar\'e, sobolev, and Young's inequalities.
\begin{align*}
(\Bar{\textbf{u}}\cdot\nabla\psi_n,\Delta^2\psi_n) &= 0,\\
|(\textbf{w}_n\cdot\nabla\Bar{\varphi},\Delta^2\psi_n)| &\leq \|\textbf{w}_n\|_{L^4}\|\nabla\Bar{\varphi}\|_{L^4}\|\Delta^2\psi_n\| \leq \|\nabla\textbf{w}_n\|\|\nabla\Bar{\varphi}\|^{1/2}\|\Bar{\varphi}\|_{H^2}^{1/2}\|\Delta^2\psi_n\| \\&\leq \epsilon\|\Delta^2\psi_n\|^2 + C\|\nabla\Bar{\varphi}\|\|\Bar{\varphi}\|_{H^2}\|\nabla\textbf{w}_n\|^2,\\
    \big(\nabla\cdot(m(\Bar{\varphi})\nabla\Tilde{\mu}_n), \Delta^2\psi_n\big) &= \big(m'(\Bar{\varphi})\nabla\Bar{\varphi}\nabla\Tilde{\mu}_n, \Delta^2\psi_n\big) + \big(m(\Bar{\varphi})\Delta\Tilde{\mu}_n, \Delta^2\psi_n\big), \end{align*}
    The first term on R.H.S. of the above equation can be estimated as 
    \begin{align*}
|\big(m'(\Bar{\varphi})\nabla\Bar{\varphi}\nabla\Tilde{\mu}_n, \Delta^2\psi_n\big)| &\leq C\|\nabla\Bar{\varphi}\|_{L^\infty}\|\nabla\Tilde{\mu}_n\|\|\Delta^2\psi_n\| \leq C\|\Bar{\varphi}\|_{H^3}\|\nabla\Tilde{\mu}_n\|\|\Delta^2\psi_n\| \\&\leq \epsilon\|\Delta^2\psi_n\|^2 + C\|\Bar{\varphi}\|_{H^3}^2\|\nabla\Tilde{\mu}_n\|^2,
\end{align*}
The second term on R.H.S. can be written as 
\begin{align*}
    \big(m(\Bar{\varphi})\Delta\Tilde{\mu}_n, \Delta^2\psi_n\big) &= \big(m(\Bar{\varphi})\Delta(-\Delta \psi_n + F''(\Bar{\varphi})\psi_n), \Delta^2\psi_n\big)\\
    &= -\big(m(\Bar{\varphi})\Delta^2\psi_n, \Delta^2\psi_n\big) + \big(m(\Bar{\varphi})F^{(4)}(\Bar{\varphi})(\nabla\Bar{\varphi})^2\psi_n, \Delta^2\psi_n\big)  \\& +  \big(m(\Bar{\varphi})F'''(\Bar{\varphi})\Delta\Bar{\varphi}\psi_n, \Delta^2\psi_n\big) +  2\big(m(\Bar{\varphi})F'''(\Bar{\varphi})\nabla\Bar{\varphi}\nabla\psi_n, \Delta^2\psi_n\big)  \\&+  \big(m(\Bar{\varphi})F''(\Bar{\varphi})\Delta\psi_n, \Delta^2\psi_n\big).\\
    &= \Sigma_{j=1}^{5} I_j
\end{align*}
By direct computation it follows that,
\begin{align*}
    -I_1 &\geq m_0\|\Delta^2 \psi_n\|^2,\\
    |I_2| &\leq \|\nabla\Bar{\varphi}\|_{L^8}^2\|\psi_n\|_{L^4}\|\Delta^2\psi_n\| \leq C\|\Delta\Bar{\varphi}\|^{3/2}\|\Bar{\varphi}\|_{L^\infty}^{1/2}\|\psi_n\|^{1/2}\|\psi_n\|_{V}^{1/2}\|\Delta^2\psi_n\| \\&\leq \epsilon\|\Delta^2\psi_n\|^2 + C\|\Delta\Bar{\varphi}\|^3\|\Bar{\varphi}\|_{L^\infty}\|\psi_n\|\|\psi_n\|_{V}\\& \leq \epsilon\|\Delta^2\psi_n\|^2 + C\|\Bar{\varphi}\|_{H^2}^4\|\psi_n\|_{V}^2,\\
    |I_3| &\leq C\|\Delta\Bar{\varphi}\|_{L^4}\|\psi_n\|_{L^4}\|\Delta^2\psi_n\| \leq C\|\Bar{\varphi}\|_{H^3}^{3/4}\|\Bar{\varphi}\|_{L^\infty}^{1/4}\|\psi_n\|^{1/2}\|\psi_n\|_{V}^{1/2}\|\Delta^2\psi_n\|\\
    &\leq \epsilon\|\Delta^2\psi_n\|^2 + C\|\Bar{\varphi}\|_{H^3}^{3/2}\|\Bar{\varphi}\|_{L^\infty}^{1/2}\|\psi_n\|\|\psi_n\|_{V}\\ &\leq \epsilon\|\Delta^2\psi_n\|^2 + C\|\Bar{\varphi}\|_{H^3}^2\|\psi_n\|_{V}^2,
 \end{align*}
 \begin{align*}   
    |I_4| &\leq C\|\nabla\Bar{\varphi}\|_{L^4}\|\nabla\psi_n\|_{L^4}\|\Delta^2\psi_n\| \leq C\|\Bar{\varphi}\|_{H^2}^{3/4}\|\Bar{\varphi}\|^{1/4}\|\psi_n\|_{H^2}^{3/4}\|\psi_n\|^{1/4}\|\Delta^2\psi_n\| \\ &\leq \epsilon\|\Delta^2\psi_n\|^2 + C\|\Bar{\varphi}\|_{H^2}^{3/2}\|\Bar{\varphi}\|^{1/2}\|\psi_n\|_{H^2}^{3/2}\|\psi_n\|^{1/2}\\
    &\leq \epsilon\|\Delta^2\psi_n\|^2 + C\|\Bar{\varphi}\|_{H^2}^2\|\psi_n\|_{H^2}^2 + C\|\Bar{\varphi}\|_{H^2}\|\Bar{\varphi}\|\|\psi_n\|_{H^2}\|\psi_n\|\\
    &\leq \epsilon\|\Delta^2\psi_n\|^2 + C\|\Bar{\varphi}\|_{H^2}^2\|\psi_n\|_{H^2}^2 + C\|\Bar{\varphi}\|^2\|\psi_n\|^2,\\
    |I_5| &\leq C\|\Delta\psi_n\|\|\Delta^2\psi_n\| \leq \epsilon\|\Delta^2\psi_n\|^2 + C\|\Delta\psi_n\|^2.
\end{align*}
Note that, we have used Sobolev embeddings to estimate the terms $I_2$ and $I_3$.
Further, note that the second term on R.H.S. of \eqref{eq2.030} can be expanded as follows,
\begin{align*}
  \big(\nabla\cdot( m'(\Bar{\varphi})\psi_n\nabla\Bar{\mu}), \Delta^2\psi_n\big) &=  \big(m''(\Bar{\varphi})\nabla\Bar{\varphi}\psi_n\nabla\Bar{\mu}, \Delta^2\psi_n\big) +  \big( m'(\Bar{\varphi})\nabla\psi_n\nabla\Bar{\mu}, \Delta^2\psi_n\big) +  \big( m'(\Bar{\varphi})\psi_n\Delta\Bar{\mu}, \Delta^2\psi_n\big).
\end{align*}
    We denote the terms on R.H.S. of the above equation by $J_1, J_2$ and $J_3$. 
\begin{align*}
    |J_1| &\leq \|\nabla\Bar{\varphi}\|_{L^8}\|\psi_n\|_{L^8}\|\nabla\Bar{\mu}\|_{L^4}\|\Delta^2\psi_n\| \\&\leq \|\Bar{\varphi}\|_{H^2}^{3/4}\|\Bar{\varphi}\|_{L^\infty}^{1/4}\|\psi_n\|_{V}^{3/4}\|\psi_n\|^{1/4}\|\Bar{\mu}\|_{H^2}^{3/4}\|\Bar{\mu}\|^{1/4}\|\Delta^2\psi_n\| \\
    &\leq \epsilon\|\Delta^2\psi_n\|^2+ \|\Bar{\varphi}\|_{H^2}^{3/2}\|\Bar{\varphi}\|_{L^\infty}^{1/2}\|\psi_n\|_{V}^{3/2}\|\psi_n\|^{1/2}\|\Bar{\mu}\|_{H^2}^{3/2}\|\Bar{\mu}\|^{1/2}\\
    &\leq \epsilon\|\Delta^2\psi_n\|^2+ \|\Bar{\varphi}\|_{H^2}^2\|\Bar{\mu}\|_{H^2}^2\|\psi_n\|_{V}^2  + \|\Bar{\varphi}\|_{H^2}\|\Bar{\varphi}\|_{L^\infty}\|\psi_n\|_{V}\|\psi_n\|\|\Bar{\mu}\|_{H^2}\|\Bar{\mu}\|\\ &\leq \epsilon\|\Delta^2\psi_n\|^2+ \|\Bar{\varphi}\|_{H^2}^2\|\Bar{\mu}\|_{H^2}^2\|\psi_n\|_{V}^2  + \|\Bar{\varphi}\|_{L^\infty}^2\|\Bar{\mu}\|^2\|\psi_n\|^2,
\end{align*}
\begin{align*}
    |J_2| &\leq \|\nabla\psi_n\|_{L^4}\|\nabla\Bar{\mu}\|_{L^4}\|\Delta^2\psi_n\| \leq \|\psi_n\|_{H^2}^{3/4}\|\psi_n\|^{1/4}\|\Bar{\mu}\|_{H^2}^{3/4}\|\Bar{\mu}\|^{1/4}\|\Delta^2\psi_n\|\\
    &\leq \epsilon\|\Delta^2\psi_n\|^2 + \|\psi_n\|_{H^2}^{3/2}\|\psi_n\|^{1/2}\|\Bar{\mu}\|_{H^2}^{3/2}\|\Bar{\mu}\|^{1/2}\\ &\leq \epsilon\|\Delta^2\psi_n\|^2 + C\|\Bar{\mu}\|_{H^2}^2\|\psi_n\|_{H^2}^2 + \|\Bar{\mu}\|^2\|\psi_n\|^2,\\
    |J_3| &\leq \|\psi_n\|_{L^\infty}\|\Delta\Bar{\mu}\|\|\Delta^2\psi_n\| \leq \|\psi_n\|^{1/2}\|\psi_n\|_{H^2}^{1/2}\|\Delta\Bar{\mu}\|\|\Delta^2\psi_n\| \\
    &\leq \epsilon\|\Delta^2\psi_n\|^2 + C\|\psi_n\|(\|\psi_n\| + \|\Delta^2\psi_n\|)\|\Delta\Bar{\mu}\|^2\\
    &\leq 2\epsilon\|\Delta^2\psi_n\|^2 + C(1+\|\Delta\Bar{\mu}\|^4)\|\psi_n\|^2\\
    |(h,\Delta^2\psi_n)| &\leq \|h\|\|\Delta^2\psi_n\| \leq \epsilon\|\Delta^2\psi_n\|^2 + C\|h\|^2.
\end{align*}
Finally, by combining all above estimates in \eqref{eq2.030}, and for a choice of $\epsilon = \frac{m_0}{22}$, we obtain, 
\begin{align*}
     \frac{1}{2}\frac{d}{dt}\|\Delta\psi_n\|^2 + \frac{m_0}{2}\|\Delta^2\psi_n\|^2 &\leq  C\|\nabla\Bar{\varphi}\|\|\Bar{\varphi}\|_{H^2}\|\nabla\textbf{w}_n\|^2+ C\|\Bar{\varphi}\|_{H^3}^2\|\nabla\Tilde{\mu}_n\|^2 + C\|h\|^2 \\&\hspace{.5cm}+ C\big( 1 + \|\Bar{\varphi}\|^2+ \|\Bar{\varphi}\|_{L^\infty}^2\|\Bar{\mu}\|^2 + \|\Bar{\mu}\|^2 +\|\Delta\Bar{\mu}\|^4 \big)\|\psi_n\|^2\\&\hspace{.5cm} + C \big( \|\Bar{\varphi}\|_{H^2}^4 + \|\Bar{\varphi}\|_{H^3}^2 + \|\Bar{\varphi}\|_{H^2}^2\|\Bar{\mu}\|_{H^2}^2 \big)\|\psi_n\|_{V}^2 \\&\hspace{.5cm} + C\big( 1+ \|\Bar{\varphi}\|_{H^2}^2+ \|\Bar{\mu}\|_{H^2}^2 \big)\|\Delta\psi_n\|^2.
\end{align*}
holds for every $t\in [0,T]$. Then by applying Gronwall's lemma, followed by an application of elliptic estimates \cite{LKW} we deduce, 
\begin{align}\label{eq3.28}
    \sup_{t\in [0,T]}\|\psi_n(t)\|_{H^2}+ \|\psi_n\|_{L^2(0,T;H^4)} \leq C(1+ \|\psi_n(0)\|_{H^2}^2 + \|\textbf{w}_n(0)\|^2 + \|h\|_{L^2(0,T;H)}^2).
\end{align}
Hence, the result follows by passing to the limit along a weakly convergent subsequence by repeating the arguments used in theorem \ref{theorem2}.
\end{proof}

\begin{remark}
 By exploiting Aubin-Lions lemma, $L^\infty(0,T;H^2)\cap H'(0,T;V') \hookrightarrow \hookrightarrow C([0,T];C^{0,\gamma}(\Bar{\Omega}))$, for $0< \gamma< 1$, the strong solution of the linearised system, $\psi \in C([0,T];C^{0,\gamma}(\Bar{\Omega}))$. This regularity of $\psi$ is necessary to derive the local differentiability of the control to state map as well as to define the adjoint variable. Similarly for the velocity, we deduce the regularity, $\textbf{w} \in C([0,T]; \mathbb{G}_{div})$.
\end{remark}

Now we prove that the control to state map  $\mathcal{S}$ can be shown to be differentiable as a mapping from $\mathcal{U} \rightarrow \mathcal{V}$; if  we  assume that the potential, $F\in C^5(\mathbb{R})$ and $m \in C^3(\mathbb{R})$. See the following theorem.
\begin{theorem}[\textit{Differentiability of the control to state operator}]\label{prop4.3}
    Let all the hypotheses of Theorem \ref{theorem3.4} are satisfied. Additionally assume $F\in C^5(\mathbb{R})$ and $m\in C^3(\mathbb{R})$. Then the control to state operator, $\mathcal{S}$ as a map from $\mathcal{U}$ to $\mathcal{V}$ is Fr\'etchet differentiable.  
\end{theorem}
 Let $ \rho$ and $\textbf{v}$ be as defined in  Theorem \ref{theorem3.4} .
 The proof of differentiability is done using the same techniques as in theorem \ref{theorem3.4}. However since we are proving the differentiability of control to state map in $\mathcal{V}$ we need extra higher regular estimates. Using the estimate \eqref{eq3.43} and higher regularity estimates proved in the appendix, we can get, 
\begin{align}\label{eq4.4}
    \frac{\|(\mathcal{S}(\Bar{U}+ h )- \mathcal{S}(\Bar{U})- (\psi_{h}, \textbf{w}_{h}))\|_{\mathcal{V}}}{\|h\|_{\mathcal{U}}}  = \frac{\|(\rho, \textbf{v})\|_{\mathcal{V}}}{\|h\|_{\mathcal{U}}} \leq \|h\|_{\mathcal{U}}\rightarrow 0, \hspace{.5cm}\text{ as }\|h\|_{\mathcal{U}}\rightarrow 0.
\end{align}
Further by applying the Aubin-Lions lemma, we obtain the regularity $(\rho, \textbf{v}) \in C(\bar{\Omega} \times [0,T]) \times C([0,T];\mathbb{G}_{div})$.

\subsection{First order necessary optimality condition}

\begin{theorem} \label{theorem4.4} 
Let $\Bar{U}$ be an optimal control, and $(\Bar{\varphi}, \Bar{\textbf{u}})$ be the associated state. Then, the following inequality holds:  
\begin{align} \label{foc2} 
\int\limits_{0}^{T} \sum\limits_{i=1}^{k} (\Bar{\varphi}(x_i,t) - \Phi^i(t)) \psi_{U-\Bar{U}}(x_i,t) &+ \sum\limits_{i=1}^{k} (\Bar{\varphi}(x_i,T) - \Phi^i(T)) \psi_{U-\Bar{U}}(x_i,T)\nonumber \\  
+ \int\limits_{0}^{T} \int_\Omega (\Bar{\textbf{u}} - \textbf{u}_d)\cdot \textbf{w}_{U-\Bar{U}} &+ \int_\Omega (\Bar{\textbf{u}}(x,T) - \textbf{u}_d(T))\cdot \textbf{w}_{U-\Bar{U}}(T) + \int\limits_{0}^{T} \int_\Omega (U - \Bar{U}) \Bar{U} \geq 0, \hspace{.5cm} \forall U \in \mathcal{U}_{ad}.
\end{align}  
 Here, $(\psi_{U-\Bar{U}}, \textbf{w}_{U-\Bar{U}})$ is the solution of the linearized system \eqref{eq2.3}-\eqref{eq2.8} corresponding to the source term $U - \Bar{U}$.  
\end{theorem} 

\begin{proof} 
Observe that $\mathcal{J}_2(\varphi, \textbf{u}, U) = \mathcal{J}_1(\varphi, \textbf{u}, U) + \frac{1}{2}\sum\limits_{i=1}^{k}(\varphi(x_i,T)-\Phi^i(T))^2 $. For $\mathcal{J}_1$ we have estimated the variation in theorem \ref{theorem2.5}. We need to estimate only the variation of the second pointwise tracking term. 

We denote $j(U) = \mathcal{J}_2(\mathcal{S}(U), U)$, $\mathcal{S}(\Bar{U}) = (\Bar{\varphi}, \Bar{\textbf{u}})$ and $\mathcal{S}(\Bar{U} + \beta U) = (\varphi_\beta, \textbf{u}_\beta)$. Observe that, 

\begin{align*}
    \frac{j(\Bar{U}+\beta U)- j(\Bar{U})}{\beta} &= \frac{1}{2\beta}\int\limits_{0}^{T}\sum\limits_{i=1}^{k}(\varphi_\beta(x_i,t)-\Phi^i(t))^2 - (\Bar{\varphi}(x_i,t)-\Phi^i(t))^2 
    \\& +\frac{1}{2\beta}\int\limits_{0}^{T}\int_\Omega|\textbf{u}_\beta-\textbf{u}_d|^2-|\Bar{\textbf{u}}-\textbf{u}_d|^2 +\frac{1}{2\beta}\int_\Omega|\textbf{u}_\beta(x,T)-\textbf{u}_d(T)|^2-|\Bar{\textbf{u}}(T)-\textbf{u}_d(T)|^2 \\&+ \frac{1}{2\beta}\int\limits_{0}^{T}\int_\Omega(\Bar{U}+\beta U)^2-(\Bar{U})^2 + \frac{1}{2\beta}\sum\limits_{i=1}^{k}(\varphi_\beta(x_i,T)-\Phi^i(T))^2- (\Bar{\varphi}(x_i,T)-\Phi^i(T))^2
\end{align*}
Observe,
\begin{align*}
    \frac{1}{2\beta}\sum\limits_{i=1}^{k}(\varphi_\beta(x_i,T)-\Phi^i(T))^2- (\Bar{\varphi}(x_i,T)-\Phi^i(T))^2 = \frac{1}{2}\sum\limits_{i=1}^{k}(\varphi_\beta(x_i,T)+ \Bar{\varphi}(x_i,T)-2\Phi^i(T)) \frac{(\varphi_\beta(x_i,T)-\Bar{\varphi}(x_i,T))}{\beta}
\end{align*}

Using the continuous dependence estimates, \eqref{cd2} and \eqref{cd3} we have,  as $\beta \rightarrow 0$, $\varphi_\beta \rightarrow \Bar{\varphi}$ in $L^\infty(0,T;H^2(\Omega)) \cap H^1(0,T;V')\hookrightarrow  C(\Bar{\Omega} \times [0,T])$. Furthermore, by theorem \ref{prop4.3} and the regularity results proved for the linearised variable in theorem \ref{theorem3.3}, we have,  $\frac{(\varphi_\beta-\Bar{\varphi})}{\beta} \rightarrow \psi_U$ in  $C(\Bar{\Omega}\times [0,T])$, where $(\psi_{U}, \textbf{w}_{U})$ denotes the solution of the linearized system around $\Bar{U}$, corresponding to the source term $U$. 
Therefore,  combining the convergence above and convergences in terms of $\mathcal{J}_1$,   the limit $\beta\rightarrow 0$, is given by
\begin{align*}
   j'(\Bar{U})(U) &=\int\limits_{0}^{T}\sum\limits_{i=1}^{k}(\Bar{\varphi}(x_i,t)- \Phi^i(t))\psi_{U}(x_i,t)dt +\int\limits_{0}^{T}\int_\Omega(\Bar{\textbf{u}}-\textbf{u}_d)\cdot\textbf{w}_{U}dxdt + \int_\Omega (\Bar{\textbf{u}}(x,T) - \textbf{u}_d(T))\cdot \textbf{w}_{U}(T)dx \\&+ \int\limits_{0}^{T}\int_\Omega U\Bar{U} dxdt + \sum\limits_{i=1}^{k}(\Bar{\varphi}(x_i,T)- \Phi^i(T)) \psi_{U}(x_i,T). 
\end{align*}
This expression along with \eqref{eq2.42} gives the result. 
\end{proof}

\subsection{Characterisation of an optimal control}
In this section we will introduce an adjoint system and discuss its well-posedness. Our idea is to replace the linearised variable in the optimality condition \eqref{foc} by the adjoint variable and characterise the optimal control. 
\subsubsection{The Adjoint system}
Let $\bar{U}$ be an optimal control and $(\bar{\varphi}, \bar{\textbf{u}})$ be the associated optimal state. Consider the following adjoint system.
\begin{align}
    -\xi'-\Bar{\textbf{u}}\cdot\nabla\xi + 2\eta'(\Bar{\varphi})D\Bar{\textbf{u}}:D\textbf{v}+\Delta(m(\Bar{\varphi})(\Delta\xi)) -m(\Bar{\varphi})F''(\Bar{\varphi})\Delta\xi & \nonumber\\+ \Delta(m'(\Bar{\varphi})\nabla\Bar{\varphi}\cdot\nabla\xi) - m'(\Bar{\varphi})\nabla(\Delta\Bar{\varphi})\cdot\nabla\xi -\nabla\cdot(\textbf{v}\Delta\Bar{\varphi})+ \Delta(\textbf{v}\cdot\nabla\Bar{\varphi}) &= R, &\text{ in } Q, \label{eq4.6}\\
    -\textbf{v}'-(\Bar{\textbf{u}}\cdot\nabla)\textbf{v} + (\textbf{v}\cdot\nabla)\Bar{\textbf{u}} + \nabla q -\nabla\cdot(2\eta(\Bar{\varphi})D\textbf{v}) + \xi\nabla\Bar{\varphi} &= (\Bar{\textbf{u}}-\textbf{u}_d), &\text{ in } Q,\label{eq4.7}\\
    \nabla\cdot \textbf{v} &= 0, &\text{ in } Q,\label{eq4.8}\\
      \textbf{v}&= \textbf{0}, &\text{ on } \Sigma,\label{eq4.9}\\
    \frac{\partial\xi}{\partial\textbf{n}} = \frac{\partial\Delta\xi}{\partial\textbf{n}} &= 0, &\text{ on } \Sigma,\label{eq4.10}\\
    \xi(x,T)= \xi_{T} , \hspace{.5cm} \textbf{v}(x,T) &= (\Bar{\textbf{u}}(T)-\textbf{u}_d(T)), &\text{ in } \Omega.\label{eq4.11}
\end{align}
Here, $R := \sum\limits_{i=1}^{k}(\Bar{\varphi}(x_i,t)-\Phi^i(t))$ and $\xi_{T} := \sum\limits_{i=1}^{k}(\Bar{\varphi}(x_i,T)-\Phi^i(T))$. Note that $(\Phi^i, \textbf{u}_d)$ denotes the desired state as defined in the cost $\mathcal{J}_2$ given by \eqref{J2}.
Now we define the solution of \eqref{eq4.6}-\eqref{eq4.11} using the transposition method due to the low regularity of $R$ and $\xi_T$.  
\begin{definition}[Solution of the adjoint system by transposition]\label{def4.1}
We say $(\xi, \textbf{v}) \in L^2(0,T;H) \times L^2(0,T; \mathbb{G}_{div})$ solve the system \eqref{eq4.6}-\eqref{eq4.11} for $R = \sum\limits_{i=1}^{k}(\Bar{\varphi}(x_i,t)-\Phi^i(t))$ and $\xi_{T} = \sum\limits_{i=1}^{k}(\Bar{\varphi}(x_i,T)-\Phi^i(T))$, if it satisfies,
\begin{align}\label{eq4.12}
    \int\limits_{0}^{T}\int_\Omega \xi U &= \int\limits_{0}^{T}\sum\limits_{i=1}^{k}(\Bar{\varphi}(x_i,t)-\Phi^i(t))\psi(x_i, t) dt + \int\limits_{0}^{T}\int_\Omega (\Bar{\textbf{u}}-\textbf{u}_d)\cdot\textbf{w}dxdt + \sum\limits_{i=1}^{k}(\Bar{\varphi}(x_i,T)-\Phi^i(T))\psi(x_i, T) \nonumber\\
    & + \int_\Omega (\Bar{\textbf{u}}(x,T)-\textbf{u}_d(x,T))\cdot\textbf{w}(x,T)dx,
\end{align}
for every $U\in L^2(0,T;H)$, where $(\psi, \textbf{w})$ is the weak  solution of the system \eqref{eq2.3}-\eqref{eq2.8} corresponding to the source term $U$ and initial data $(\psi_0, \textbf{w}_0) = (0,\textbf{0})$. 
\end{definition}

 We prove the existence result  using similar arguments as in the proof of theorem \ref{theorem3.6}. We approximate the source term $R = \sum\limits_{i=1}^{k}(\Bar{\varphi}(x_i,t)-\Phi^i(t))$ by functions $R_\epsilon$ in $L^2(0, T; H)$ and the terminal data $\xi_T$ by $\xi_{T \epsilon} \in H$  and passing to the limit to obtain a solution of the system.

\begin{theorem}[Existence of solution to the adjoint system]
Let assumptions \emph{\textbf{[A3]}}-\emph{\textbf{[A5]}} are satisfied. Then there exists a solution to the system \eqref{eq4.6}-\eqref{eq4.11} in the sense of definition \ref{def4.1}, where $R = \sum\limits_{i=1}^{k}(\Bar{\varphi}(x_i,t)-\Phi^i(t))$ and $\xi_{T} = \sum\limits_{i=1}^{k}(\Bar{\varphi}(x_i,T)-\Phi^i(T))$ .
Further, the uniqueness of the weak solution follows as the adjoint system is linear.
\end{theorem}
\begin{proof}
    The proof is based on an approximation technique. Let 
    $$R_\epsilon := \sum\limits_{i=1}^{k}\frac{1}{|B(x_i,\epsilon)|}\chi_{B(x_i,\epsilon)}(\Bar{\varphi}(x,t)-\Phi^i(t)) \; \; and  \;\; \xi_{T\epsilon}= \sum\limits_{i=1}^{k}\frac{1}{|B(x_i,\epsilon)|}\chi_{B(x_i,\epsilon)}(\Bar{\varphi}(x,T)-\Phi^i(T)).$$ We choose  $\epsilon$ sufficiently small so that the balls $\{ B(x_i, \epsilon)\}_i, 1\leq i\leq k$ do not overlap. Observe that $R_\epsilon\in L^2(0,T;H)$ and $\xi_{T\epsilon}\in H$, therefore, by Lemma \ref{lemma2.6}, there exists a weak solution to the adjoint system corresponding to $(R_\epsilon, \xi_{T\epsilon})$, say $(\xi_\epsilon, \textbf{v}_\epsilon)$.

    Let $U\in \mathcal{U}$ and $(\psi, \textbf{w})$ be the solution of the linearized system \eqref{eq2.3}-\eqref{eq2.8} corresponding to the source term $U-\Bar{U}$ and the initial data, $(\psi_0, \textbf{w}_0) = (0, \textbf{0})$. For each fixed $t$,  $\psi(t) \in V$ and $\textbf{w}(t) \in \mathbb{V}_{div} $. We test \eqref{eq4.6}, \eqref{eq4.7}, with $\psi$, $\textbf{w}$ respectively. Adding both the equations we get,
\begin{align*}
    -(\xi_{\epsilon}',\psi) -(\Bar{\textbf{u}}\cdot\nabla\xi_{\epsilon},\psi) + 2(\eta'(\Bar{\varphi})D\Bar{\textbf{u}}:D\textbf{v}_{\epsilon}, \psi)+ (\Delta(m(\Bar{\varphi})\Delta\xi_{\epsilon}),\psi) -(m(\Bar{\varphi})F''(\Bar{\varphi})\Delta\xi_{\epsilon},\psi)& \nonumber\\+ (\Delta(m'(\Bar{\varphi})\nabla\Bar{\varphi}\cdot\nabla\xi_{\epsilon}),\psi) - (m'(\Bar{\varphi})\nabla(\Delta\Bar{\varphi})\cdot\nabla\xi_{\epsilon},\psi) -(\nabla\cdot(\textbf{v}_{\epsilon}\Delta\Bar{\varphi}),\psi) + (\Delta(\textbf{v}_{\epsilon}\cdot\nabla\Bar{\varphi}),\psi)& \\
    -(\textbf{v}_{\epsilon}', \textbf{w})-((\Bar{\textbf{u}}\cdot\nabla)\textbf{v}_{\epsilon},\textbf{w}) + ((\textbf{v}_{\epsilon}\cdot\nabla)\Bar{\textbf{u}},\textbf{w}) + 2(\eta(\Bar{\varphi})D\textbf{v}_{\epsilon},D\textbf{w}) + (\xi_{\epsilon}\nabla\Bar{\varphi},\textbf{w}) &= (R_\epsilon, \psi)+ (\Bar{\textbf{u}}-\textbf{u}_d, \textbf{w})
\end{align*}
Applying integration by parts and rearranging, we obtain,
\begin{align*}
    (\xi_{\epsilon}&,\psi')- \int_\Omega\xi_{T\epsilon}\psi(T) + (\xi_{\epsilon},\Bar{\textbf{u}}\cdot\nabla\psi) + (\textbf{v}_{\epsilon}, \nabla\cdot(2\eta'(\Bar{\varphi})\psi D\Bar{\textbf{u}}))+ (\xi_{\epsilon},\Delta(m(\Bar{\varphi})\Delta\psi)) -(\xi_{\epsilon},\Delta(m(\Bar{\varphi})F''(\Bar{\varphi})\psi)) \nonumber\\&- (\xi_{\epsilon},\nabla\cdot(m'(\Bar{\varphi})\nabla\Bar{\varphi}\Delta\psi)) + (\xi_{\epsilon},\nabla\cdot(m'(\Bar{\varphi})\nabla(\Delta\Bar{\varphi})\psi)) +(\textbf{v}_{\epsilon},(\Delta\Bar{\varphi})\nabla\psi) + (\textbf{v}_{\epsilon},(\Delta\psi)\nabla\Bar{\varphi}) 
    +(\textbf{v}_{\epsilon}, \textbf{w}')\\ &- \int_\Omega (\Bar{\textbf{u}}(T)-\textbf{u}_d(T))\textbf{w}(T)dx+ (\textbf{v}_{\epsilon}, (\Bar{\textbf{u}}\cdot\nabla)\textbf{w}) + (\textbf{v}_{\epsilon},(\textbf{w}\cdot\nabla)\Bar{\textbf{u}}) - (\textbf{v}_{\epsilon},\nabla\cdot(2\eta(\Bar{\varphi})D\textbf{w})) + (\xi_{\epsilon},\textbf{w}\cdot\nabla\Bar{\varphi}) \\&= (R_\epsilon, \psi)+ (\Bar{\textbf{u}}-\textbf{u}_d, \textbf{w})
\end{align*}
After rearranging and applying Equations \eqref{eq2.3} and \eqref{eq2.4}, we get,
\begin{align*}
   &\Big(\xi_{\epsilon},\psi' + \Bar{\textbf{u}}\cdot\nabla\psi + \textbf{w}\cdot\nabla\Bar{\varphi}+ \Delta\big(m(\Bar{\varphi})(\Delta\psi - F''(\Bar{\varphi})\psi)\big) - \nabla\cdot(m'(\Bar{\varphi})\nabla\Bar{\varphi}\Delta\psi) + \nabla\cdot(m'(\Bar{\varphi})\nabla(\Delta\Bar{\varphi})\psi) \Big) \\
     & \hspace{.5cm}+ \big(\textbf{v}_{\epsilon}, \textbf{w}'+ (\Bar{\textbf{u}}\cdot\nabla)\textbf{w}+ (\textbf{w}\cdot\nabla)\Bar{\textbf{u}} - \nabla\cdot(2\eta(\Bar{\varphi})D\textbf{w}) - \nabla\cdot(2\eta'(\Bar{\varphi})\psi D\Bar{\textbf{u}}) + (\Delta\Bar{\varphi})\nabla\psi + (\Delta\psi)\nabla\Bar{\varphi} \big)  \\ 
     &= \big(\xi_{\epsilon}, U-\bar{U} \big) \\
     &=(R_\epsilon, \psi)+ (\Bar{\textbf{u}}-\textbf{u}_d,  \textbf{w}) + \int_\Omega\xi_{T\epsilon}\psi(T) + \int_\Omega (\Bar{\textbf{u}}(T)-\textbf{u}_d(T))\textbf{w}(T)dx
\end{align*}
It follows that,
\begin{align}\label{eq4.13}
     (\xi_\epsilon, U-\bar{U})= (R_\epsilon, \psi)+ (\Bar{\textbf{u}}-\textbf{u}_d,  \textbf{w}) + \int_\Omega\xi_{T\epsilon}\psi(T) + \int_\Omega (\Bar{\textbf{u}}(T)-\textbf{u}_d(T))\textbf{w}(T)dx
\end{align}
Indeed,
\begin{align*}
     \Big|\int\limits_{0}^{T}(\xi_\epsilon, U-\bar{U}) dt \Big| &\leq  \int\limits_{0}^{T}|(R_\epsilon ,\psi)|dt + \int\limits_{0}^{T}|(\Bar{\textbf{u}}-\textbf{u}_d, \textbf{w})| dt + |(\xi_{T\epsilon}, \psi(T))| + |(\Bar{\textbf{u}}(T)-\textbf{u}_d(T), \textbf{w}(T))|
     \\&\leq \sum\limits_{i=1}^{k}\frac{1}{|B(x_i,\epsilon)|}\int\limits_{0}^{T}\int_\Omega|\chi_{B(x_i,\epsilon)}(\Bar{\varphi}(x,t)-\Phi^i(t))\psi|dxdt + \int\limits_{0}^{T}\|\Bar{\textbf{u}}-\textbf{u}_d\|\|\textbf{w}\| \\&\hspace{.5cm}+ \sum\limits_{i=1}^{k}\frac{1}{|B(x_i,\epsilon)|}\int_{\Omega}|\chi_{B(x_i,\epsilon)}(\Bar{\varphi}(x,T)-\Phi^i(T))\psi(x,T)|dx + \|\Bar{\textbf{u}}(x,T)-\textbf{u}_d(T)\|\|\textbf{w}(x,T)\|\\
     &\leq \sum\limits_{i=1}^{k}\int\limits_{0}^{T}\|\Bar{\varphi}(x,t)-\Phi^i(t)\|_{L^\infty(\Omega)}\|\psi\|_{L^\infty(\Omega)} + \int\limits_{0}^{T}\|\Bar{\textbf{u}}-\textbf{u}_d\|\|\textbf{w}\| \\&\hspace{.5cm}+ \sum\limits_{i=1}^{k}\|\Bar{\varphi}(x,T)-\Phi^i(T)\|_{L^\infty(\Omega)}\|\psi(x,T)\|_{L^\infty(\Omega)} + \|\Bar{\textbf{u}}(x,T)-\textbf{u}_d(T))\|_{L^2(\Omega)}\|\textbf{w}(x,T)\|_{L^2(\Omega)}
\end{align*}
The last inequality is obtained by combining the regularity results for $(\psi, \textbf{w})$ and the embedding, $C([0,T];H^2(\Omega))\hookrightarrow C(\Omega\times[0,T])$. Now using \eqref{eq2.18}, \eqref{eq3.25}, we have, 
\begin{align*}
     \int\limits_{0}^{T} |(\xi_\epsilon, U-\bar{U})| dt 
     &\leq C\Big( \sum\limits_{i=1}^{k}\int\limits_{0}^{T}\|\Bar{\varphi}(x,t)-\Phi^i(t)\|_{L^\infty(\Omega)}^2 + \int\limits_{0}^{T}\|\Bar{\textbf{u}}-\textbf{u}_d\|^2 \\&\hspace{.5cm}+ \sum\limits_{i=1}^{k}\|\Bar{\varphi}(x,T)-\Phi^i(T)\|_{L^\infty(\Omega)} + \|\Bar{\textbf{u}}(x,T)-\textbf{u}_d(T))\|_{L^2(\Omega)} \Big) \|U-\bar{U}\|_{L^2(0,T;H)}
\end{align*}
which holds for all $U\in L^2(0,T;V)$. This implies $\|\xi_{\epsilon}\|_{L^2(0,T;H)}\leq C$ , where $C$ is a positive constant independent of $\epsilon$.

Similarly we derive estimates for $\textbf{v}_\epsilon$, following the same methods as in theorem \ref{theorem3.6}. 
It follows that \\$\|\textbf{v}_\epsilon\|_{L^2(0,T;\mathbb{V}_{div})\cap L^\infty(0,T;\mathbb{G}_{div})} \leq C$, where $C$ is constant independent of $\epsilon$. Using the uniform estimates derived, we have, renamed subsequences, 
\begin{align*}
    \xi_\epsilon &\rightharpoonup \xi, \hspace{1cm}\text{in } L^2(0,T;H),\\
    \textbf{v}_\epsilon &\overset{\ast}{\rightharpoonup} \textbf{v}, \hspace{1cm}\text{in } L^\infty(0,T;\mathbb{G}_{div}),\\
    \textbf{v}_\epsilon &\rightharpoonup \textbf{v}, \hspace{1cm}\text{in } L^2(0,T;\mathbb{V}_{div}).
\end{align*}
Since $\Bar{\varphi}, \psi \in L^\infty(0,T; H^2(\Omega)) \hookrightarrow C(\Bar{\Omega}\times [0,T])$, we have, for each $i$, $1\leq i\leq k$, the following distributional convergences as $\epsilon \rightarrow 0$.  
\begin{align*}
    \frac{1}{|B(x_i,\epsilon)|}\int\limits_{0}^{T}\int_\Omega|\chi_{B(x_i,\epsilon)}(\Bar{\varphi}(x,t)-\Phi^i(t))\psi| &\rightarrow \int\limits_{0}^{T}(\Bar{\varphi}(x_i,t)-\Phi^i(t))\psi(x_i,t)dt, \\
    \frac{1}{|B(x_i,\epsilon)|}\int_{\Omega}\chi_{B(x_i,\epsilon)}(\Bar{\varphi}(x,T)-\Phi^i(T))\psi(x,T) &\rightarrow (\Bar{\varphi}(x_i,T)-\Phi^i(T))\psi(x_i,T)
\end{align*}
Therefore we can pass to the limit in \eqref{eq2.54} to get,
\begin{align}\label{eq4.15}
    \int\limits_{0}^{T}\int_\Omega \xi (U-\bar{U}) &= \int\limits_{0}^{T}\sum\limits_{i=1}^{k}(\Bar{\varphi}(x_i,t)-\Phi^i(t))\psi(x_i, t) dt + \int\limits_{0}^{T}\int_\Omega (\Bar{\textbf{u}}-\textbf{u}_d)\cdot\textbf{w}dxdt + \sum\limits_{i=1}^{k}(\Bar{\varphi}(x_i,T)-\Phi^i(T))\psi(x_i, T) \nonumber\\
    & + \int_\Omega (\Bar{\textbf{u}}(x,T)-\textbf{u}_d(x,T))\cdot\textbf{w}(x,T)dx,
\end{align}
hold true for all $U\in L^2(0,T;V)$. Hence the existence of a solution of \eqref{eq4.6}-\eqref{eq4.11} in the sense of definition \ref{def2}. Moreover, the uniqueness of the solution follows due to linearity of the system.
\end{proof}

\begin{theorem}\label{theorem4.6}
Let $\Bar{U}$ be an optimal control and $(\Bar{\varphi}, \Bar{\textbf{u}})$ be the associated state. Further, $(\xi, \textbf{v})$ be the adjoint variable as defined in definition \ref{def4.1} corresponding to $(U-\Bar{U})$. Then the following variational inequality holds true. 
\begin{align} \label{choc2} 
\int\limits_{0}^{T} \int_\Omega (\xi+ \Bar{U})(U - \Bar{U}) \geq 0, \hspace{.5cm} \forall U \in \mathcal{U}_{ad}.
\end{align}  
\end{theorem}
\begin{proof}
    This result follows directly by comparing the optimality condition, \eqref{foc2} and \eqref{eq4.12}. 
\end{proof}

\section{Local CHNS system with a singular potential}

In most of the real life problems, the potential appears to be singular. Motivated by this, we study the \textbf{[OCP\RomanNumeralCaps{1}]} for system \eqref{eq8}–\eqref{eq14} under a singular potential. The mathematical analysis of this setting has been developed in several works. In \cite{AMT}, the basic well-posedness theory was established. Later, \cite{GMT, JHW} proved the separation property and obtained higher regularity of solutions. These results will be central to our discussion. More recently, the model was extended to include chemotaxis effects in \cite{JWH}.

Consider the following set of assumptions from \cite{ JNG, GMT, JHW}.

\begin{enumerate}[font={\bfseries},label={A\arabic*.}]
\item[{[H1]}] The potential, $F \in C([-1,1]) \cap C^3(-1,1)$ and it can be expressed as
    \begin{align*}
        F(r) = F_1(r) - \frac{\theta_0}{2}r^2,
     \end{align*}
  such that,
  \begin{align*}
    \lim_{r\rightarrow \pm 1} F_1'(r)= \infty ; \hspace{.5cm} F_1''(r) \geq \alpha_0, \,\,r\in(-1,1).
  \end{align*}
  where $\theta_0, \alpha_0 \in \mathbb{R}$ and $\alpha_0 >0$. Furthermore, There exists  $\epsilon_0 > 0$ such that $F_1''$ is non-decreasing in $[1-\epsilon_0,1)$ and non-increasing in $(-1,-1+\epsilon_0].$
 \item[{[H2]}] The convex potential $F_1$ satisfies the following growth condition on its derivatives,
    \begin{align}
        F_1''(r) \leq Ce^{C|F'(r)|},\,\, \forall r \in (-1,1), 
    \end{align}
    for some $ C >0$.
\end{enumerate}

Observe that the logarithmic potential defined by,
\begin{equation}\label{eq09}
F_{\log}(r)=\frac{\theta}{2}((1+r) \log(1+r)+(1-r) \log(1-r))+\frac{\theta_c}{2}\left(1-r^2\right),\quad \text{ for } r \in(-1,1).
\end{equation}
satisfies all above assumptions. We have the strong wellposedness result for the local system with singular potential \cite[Theorem 2.1]{JNG}.


The argument in \cite{JNG} is based on an approximation technique. The singular potential $F$ is approximated by a family of regular functions $F_\epsilon$. For the corresponding approximated solutions $(\varphi_\epsilon, \mathbf{u}_\epsilon)$, $\epsilon$-uniform estimates can be derived. To obtain higher regularity of $\varphi$, the authors employ an elliptic estimate for $\varphi$, which follows from \eqref{eq9} and \eqref{eq12}. The separation property for $\varphi$ allows control over the terms $F'(\varphi)$ and $F''(\varphi)$. Since $\varphi$ remains confined to a compact interval $[-1+\delta, 1-\delta], \,\forall t\geq 0$ for some $\delta>0$, both $F'(\varphi)$ and $F''(\varphi)$ attain sufficient regularity. This effectively reduces the singular potential case to the regular case, and the remaining analysis proceeds analogously to that in Section 3. 
The regularity of $(\varphi, \textbf{u})$,
 \begin{align*}
       \varphi &\in L^\infty(0,T; H^2)\cap L^2(0,T; H^3) \cap H^1(0,T;H), \\
      \textbf{u} &\in L^\infty(0,T;\mathbb{V}_{div}) \cap H^1(0,T; \mathbb{G}_{div}),
   \end{align*}
from \cite[Theorem 2.1]{JNG} is sufficient to define \textbf{[OCP\RomanNumeralCaps{1}]}. Furthermore, the solution satisfy the following instantaneous separation property \cite[Theorem 4.4]{GMT}, \cite[Theorem 2.2]{JHW}. i.e, given $\tau> 0$, there exist $\delta >0$ such that 
\begin{equation}
    \sup_{t\geq \tau} \|\varphi(t)\|_{L^\infty(\Omega)} \leq 1-\delta.
\end{equation}
The separation property is crucial to study the pointwise tracking control problem. 
\begin{theorem}
Consider the system \eqref{eq8}–\eqref{eq14}. Let $\varphi_0 \in H^2(\Omega)$, $\textbf{u}_0 \in \mathbb{V}_{div}$, $U\in L^2(0,T; V\cap L^\infty(\Omega))$  \emph{\textbf{[A4]}, and let $F$ be a singular potential such that the assumptions } \textbf{[H1]}, \textbf{[H2]} are satisfied. Then the optimal control problem \emph{\textbf{[OCP\RomanNumeralCaps{1}]}} admits a solution. 
\end{theorem}

 Sketch of the Proof: The existence of an optimal control follows by similar arguments as in theorem \ref{theorem1}. Due to the low regularity of the cost functional, $\mathcal{J}_1$, we study the adjoint system using the transposition method, which will require a higher regularity of the linearised variable. We need to obtain a strong solution of the linearised system. This can be done by setting a higher regularity assumption on $m$ and $F$ on $(-1,1)$ along with higher regularity of the initial data, $\psi_0$ and $\textbf{w}_0$. Thanks to the separation property, we can adopt the methods used for the regular potential to study the singular case. By repeating the techniques used for the regular case we can get the optimality conditions and a characterisation of an optimal control analogous to the theorem \ref{theorem2.5} and theorem \ref{theorem3.8}.
Observe that the separation property of the system plays a crucial role in this study.

\subsubsection*{Concluding Remarks}
In this work, we have addressed  two different pointwise tracking optimal control problems for the local CHNS system in dimension 2, where  $L^2$ control is applied in the concentration equation. Our analysis has been carried out assuming a mobility function with a bounded derivative and a regular potential. 

A similar investigation can be extended to the nonlocal system under various assumptions on mobility and potential. According to the well-posedness results established in \cite{FGG, FSC, FGK, FGS}, with suitable assumptions on the initial data, potential, and mobility, the solutions satisfy $\varphi \in L^\infty(0,T; H^2)$ and $\mathbf{u} \in L^2(0,T; \mathbb{V}_{div})$. While these regularity results ensure existence, they are generally insufficient to study the differentiability properties of the control-to-state map. In particular, the nonlocal potential complicates the derivation of strong solutions for the linearized system, which is a key step in proving differentiability. This remains an open problem and is an important direction for future research, given the practical relevance of nonlocal CHNS systems.

\section*{Appendix}

\begin{prop}\label{propA}
    Let all the hypotheses of Theorem \ref{theorem3.4} be satisfied. Additionally assume $F\in C^5(\mathbb{R})$ and $m\in C^3(\mathbb{R})$. We prove the following additional regularity of $(\rho, \textbf{v})$, the solution of \eqref{eq2.30}-\eqref{eq2.35}.
    \begin{align}\label{eq4.3}
        \|\rho \|_{L^\infty(0,T;H^2(\Omega)) \cap L^2(0,T;H^4(\Omega)) \cap  H^1(0,T; V')} + \|\textbf{v}\|_{H^1(0,T;\mathbb{V}_{div}')} \leq \|h\|_{L^2(0,T;H)}^2.
    \end{align}
\end{prop}
\begin{proof}

 Continuing from the estimates derived in the proof of Theorem \ref{theorem3.4}, we test \eqref{eq2.30} with $\Delta^2\rho$ to obtain,
\begin{align}
    &\frac{1}{2}\frac{d}{dt}\|\Delta\rho\|^2 + (\bm{\tau}\cdot\nabla\xi,\Delta^2\rho) + (\Bar{u}\cdot\nabla\rho,\Delta^2\rho) + (\textbf{v}\cdot\nabla\Bar{\varphi}, \Delta^2\rho) 
    \nonumber\\&\hspace{.5cm}= (\nabla\cdot(m(\Bar{\varphi})\nabla(-\Delta\rho)), \Delta^2\rho) + (\nabla\cdot(m(\Bar{\varphi})\nabla(F'(\varphi_h)-F'(\Bar{\varphi})-F''(\Bar{\varphi})\psi)), \Delta^2\rho) \nonumber\\&\hspace{.5cm}+ (\nabla\cdot(m'(\Bar{\varphi})\xi\nabla(-\Delta\xi)), \Delta^2\rho) + (\nabla\cdot(m'(\Bar{\varphi})\xi\nabla(F'(\varphi_h)-F'(\Bar{\varphi}))), \Delta^2\rho)\nonumber\\&\hspace{.5cm}
    +(\nabla\cdot((m(\varphi_h)-m(\Bar{\varphi})-m'(\Bar{\varphi})\psi)\nabla(-\Delta\Bar{\varphi}+F'(\Bar{\varphi}))), \Delta^2\rho).
\end{align}
To estimate the terms, we invoke Hölder’s, Gagliardo–Nirenberg, Poincaré and Young’s inequalities. 
\begin{align*}
    |(\bm{\tau}\cdot\nabla\xi,\Delta^2\rho)| &\leq \|\bm{\tau}\|_{L^4}\|\nabla\xi\|_{L^4}\|\Delta^2\rho\| \leq \epsilon\|\Delta^2\rho\|^2 + C\|\bm{\tau}\|\|\nabla\bm{\tau}\|\|\nabla\xi\|\|\xi\|_{H^2},\\
    (\Bar{u}\cdot\nabla\rho,\Delta^2\rho) &= \int_\Omega\Bar{u}\cdot\nabla(\frac{(\nabla\rho)^2}{2}) dx =0, \\|(\textbf{v}\cdot\nabla\Bar{\varphi}, \Delta^2\rho)| &\leq \|\textbf{v}\|_{L^4}\|\nabla\Bar{\varphi}\|_{L^4}\| \Delta^2\rho\| \leq \epsilon\|\Delta^2\rho\|^2 + C\|\textbf{v}\|\|\nabla\textbf{v}\|\|\nabla\Bar{\varphi}\|\|\Bar{\varphi}\|_{H^2}.
\end{align*}
Proceeding further,
\begin{align*}
    (\nabla\cdot(m(\Bar{\varphi})\nabla(-\Delta\rho)), \Delta^2\rho) &= (m'(\Bar{\varphi})\nabla\Bar{\varphi}\nabla(-\Delta\rho), \Delta^2\rho) -(m(\Bar{\varphi})\Delta^2\rho), \Delta^2\rho),\\
    (m(\Bar{\varphi})\Delta^2\rho, \Delta^2\rho) &\geq m_0\|\Delta^2\rho\|^2,\\
    |(m'(\Bar{\varphi})\nabla\Bar{\varphi}\nabla(-\Delta\rho), \Delta^2\rho)| &\leq C\|\nabla\Bar{\varphi}\|_{L^4}\|\nabla^3\rho\|_{L^4}\|\Delta^2\rho\| \leq C\|\nabla\Bar{\varphi}\|^{1/2}\|\Bar{\varphi}\|_{H^2}^{1/2}\|\rho\|_{H^4}^{7/8}\|\rho\|^{1/8}\|\Delta^2\rho\|
    \\&\leq \epsilon\|\Delta^2\rho\|^2+  C\|\nabla\Bar{\varphi}\|\|\Bar{\varphi}\|_{H^2}\|\Delta^2\rho\|^{7/4}\|\rho\|^{1/4} + C\|\rho\|^2
    \\&\leq 2\epsilon\|\Delta^2\rho\|^2+  C\|\nabla\Bar{\varphi}\|^2\|\Bar{\varphi}\|_{H^2}^2\|\Delta^2\rho\|^{3/2}\|\rho\|^{1/2} + C\|\rho\|^2
    \\&\leq 3\epsilon\|\Delta^2\rho\|^2+  C\|\nabla\Bar{\varphi}\|^4\|\Bar{\varphi}\|_{H^2}^4\|\Delta^2\rho\|\|\rho\| + C\|\rho\|^2
     \\&\leq 4\epsilon\|\Delta^2\rho\|^2+  C\|\nabla\Bar{\varphi}\|^8\|\Bar{\varphi}\|_{H^2}^8\|\rho\|^2 + C\|\rho\|^2.
\end{align*}
Observe that, the last term is estimated by a repeated application of Gagliardo-Nirenberg and Young's inequalities. Indeed we employ an elliptic estimate for the $H^4$ norm of $\rho$.

Before estimating the next term, we derive an expression for the term $\Delta\big( F'(\varphi_h)-F'(\Bar{\varphi})-F''(\Bar{\varphi})\psi \big)$ following similar methods as in the proof of theorem \ref{theorem3.4}. Recalling \eqref{F11} \eqref{F12}, which give the expressions for  $F'(\varphi_h)-F'(\Bar{\varphi})-F''(\Bar{\varphi})\psi$ and its gradient, we differentiate further to obtain,

\begin{align}\label{F13}
  &\Delta\big( F'(\varphi_h)-F'(\Bar{\varphi})-F''(\Bar{\varphi})\psi \big) \nonumber\\&= \Delta\xi\int\limits_{0}^{1} \int\limits_{0}^{1} F'''(s(r\varphi_h +(1-r)\Bar{\varphi}) + (1-s)\Bar{\varphi}) r\xi dr + \nabla\xi \int\limits_{0}^{1} F'''(r\varphi_h +(1-r)\Bar{\varphi}) r\nabla\xi dr\nonumber\\
  &\hspace{.5cm}+ \nabla\xi \int\limits_{0}^{1} F^{(4)}(s(r\varphi_h +(1-r)\Bar{\varphi}) + (1-s)\Bar{\varphi}) r\xi\nabla\Bar{\varphi} dr + \nabla\xi\int\limits_{0}^{1} F'''(r\varphi_h +(1-r)\Bar{\varphi}) r\nabla\xi dr \nonumber\\
  &\hspace{.5cm}+ \xi\int\limits_{0}^{1} F'''(r\varphi_h +(1-r)\Bar{\varphi})r\Delta\xi dr  + \xi\int\limits_{0}^{1} F^{(4)}(r\varphi_h +(1-r)\Bar{\varphi})(r\nabla\varphi_h +(1-r)\nabla\Bar{\varphi})r\nabla\xi dr\nonumber\\
  &\hspace{.5cm}+ \nabla\xi\int\limits_{0}^{1} F^{(4)}(s(r\varphi_h +(1-r)\Bar{\varphi}) + (1-s)\Bar{\varphi})r\xi\nabla\Bar{\varphi}dr + \xi\int\limits_{0}^{1} F^{(4)}(s(r\varphi_h +(1-r)\Bar{\varphi}) + (1-s)\Bar{\varphi})r\xi\Delta\Bar{\varphi}dr\nonumber\\
  &\hspace{.5cm}+ \xi\int\limits_{0}^{1} (F^{(4)}(r\varphi_h+ (1-r)\Bar{\varphi})r\nabla\xi\cdot\nabla\Bar{\varphi} dr+ \xi\int\limits_{0}^{1} F^{(5)}(s(r\varphi_h +(1-r)\Bar{\varphi}) + (1-s)\Bar{\varphi})r\xi\nabla\Bar{\varphi}\cdot \nabla\Bar{\varphi} dr
  \nonumber\\&\hspace{.5cm}+ F^{(4)}(\Bar{\varphi})\nabla\Bar{\varphi}\cdot\nabla\Bar{\varphi}\rho + F'''(\Bar{\varphi})\Delta\Bar{\varphi}\rho +  2F'''(\Bar{\varphi})\nabla\Bar{\varphi}\cdot\nabla\rho  + F''(\Bar{\varphi})\Delta\rho. 
\end{align}
Observe that, 
\begin{align*}
    |(\nabla\cdot(m(\Bar{\varphi})\nabla(F'(\varphi_h)-F'(\Bar{\varphi})-F''(\Bar{\varphi})\psi)), \Delta^2\rho)| &\leq |(m'(\Bar{\varphi})\nabla\Bar{\varphi}\nabla(F'(\varphi_h)-F'(\Bar{\varphi})-F''(\Bar{\varphi})\psi), \Delta^2\rho)| \\
    & \hspace{.5cm}+|(m(\Bar{\varphi})\Delta(F'(\varphi_h)-F'(\Bar{\varphi})-F''(\Bar{\varphi})\psi), \Delta^2\rho)|.
\end{align*}
Using expressions \eqref{F12}, \eqref{F13} and standard inequalities and sobolev embeddings, proceeding further we obtain,
\begin{align*}
    |(m'(\Bar{\varphi})\nabla\Bar{\varphi}\nabla(F'(\varphi_h)-F'(\Bar{\varphi})-F''(\Bar{\varphi})\psi), \Delta^2\rho)| &\leq C\Big[ \|\nabla\Bar{\varphi}\|_{L^8}\|\nabla\xi\|_{L^4}\|\xi\|_{L^8}\|\Delta^2\rho\| + \|\nabla\Bar{\varphi}\|_{L^8}^2\|\xi\|_{L^8}^2\|\Delta^2\rho\| \\&\hspace{.5cm}+ \|\nabla\Bar{\varphi}\|_{L^8}^2\|\rho\|_{L^4}\|\Delta^2\rho\| + \|\nabla\Bar{\varphi}\|_{L^4}\|\nabla\rho\|_{L^4}\|\Delta^2\rho\| \Big]\\ 
    &\leq 4\epsilon\|\Delta^2\rho\|^2 + \|\Bar{\varphi}\|_{H^2}^{3/2}\|\Bar{\varphi}\|_{L^\infty}^{1/2}\|\xi\|_{H^2}^{3/2}\|\xi\|^{1/2}\|\xi\|_{V}^{3/2}\|\xi\|^{1/2} \\&\hspace{.5cm}+ \|\Bar{\varphi}\|_{H^2}^{3}\|\Bar{\varphi}\|_{L^\infty}\|\xi\|_{V}^{3}\|\xi\| + \|\Bar{\varphi}\|_{H^2}^{3}\|\Bar{\varphi}\|_{L^\infty}\|\rho\|_{V}\|\rho\| \\&\hspace{.5cm}+ \|\Bar{\varphi}\|_{H^2}^{3/2}\|\Bar{\varphi}\|^{1/2}\|\rho\|_{H^2}^{3/2}\|\rho\|^{1/2}\\
     &\leq 4\epsilon\|\Delta^2\rho\|^2 + \|\Bar{\varphi}\|_{H^2}^{2}\|\xi\|_{H^2}^{2}\|\xi\|_{V}^{2} + \|\Bar{\varphi}\|_{H^2}^{3}\|\Bar{\varphi}\|_{L^\infty}\|\xi\|_{V}^{3}\|\xi\| \\&\hspace{.5cm}+ \|\Bar{\varphi}\|_{H^2}^{3}\|\Bar{\varphi}\|_{L^\infty}\|\rho\|_{V}\|\rho\| + \|\Bar{\varphi}\|_{H^2}^{2}\|\rho\|_{H^2}^{2}.
\end{align*}
Observe that in the previous calculations we have used the fact that $\Bar{\varphi} \in C([0,T]; C(\Bar{\Omega}))$ and $F\in C^5(\mathbb{R})$. In addition, the regularity results for the strong solution together with the stability estimates \eqref{cd1}- \eqref{cd3} justify these computations. Similarly,
\begin{align*}
    |(m(\Bar{\varphi})\Delta(F'(\varphi_h)-F'(\Bar{\varphi})-F''(\Bar{\varphi})\psi), \Delta^2\rho)| &\leq \big[ \|\Delta\xi\|\|\xi\|_{L^\infty} + \|\nabla\xi\|_{L^4}^2 + \|\nabla\Bar{\varphi}\|_{L^8}\|\xi\|_{L^8}\|\nabla\xi\|_{L^4} + \|\Delta\Bar{\varphi}\|_{L^4}\|\xi\|_{L^8}^2  \\
    &\hspace{.5cm}+ \|\nabla\Bar{\varphi}\|_{L^8}^2\|\xi\|_{L^8}^2 + \|\nabla\Bar{\varphi}\|_{L^8}^2\|\rho\|_{L^4} + \|\Delta\Bar{\varphi}\|_{L^4}\|\rho\|_{L^4} \\
    &\hspace{.5cm}+ \|\nabla\Bar{\varphi}\|_{L^4}\|\nabla\rho\|_{L^4} + \|\Delta\rho\| \big]\|\Delta^2\rho\| \\
    &\leq \big[ \|\xi\|_{H^2}^2 + \|\xi\|_{H^2}\|\xi\|_{V} + \|\Bar{\varphi}\|_{H^2}^{3/4}\|\Bar{\varphi}\|_{L^\infty}^{1/4}\|\xi\|_{V}^{3/4}\|\xi\|^{1/4}\|\xi\|_{H^2}^{3/4}\|\xi\|^{1/4} \\
    &\hspace{.5cm}+ \|\Bar{\varphi}\|_{H^3}^{3/4}\|\Bar{\varphi}\|_{L^\infty}^{1/4}\|\xi\|_{V}^{3/2}\|\xi\|^{1/2}  + \|\Bar{\varphi}\|_{H^2}^{3/2}\|\Bar{\varphi}\|_{L^\infty}^{1/2}\|\xi\|_{V}^{3/2}\|\xi\|^{1/2} \\
    &\hspace{.5cm}+ \|\Bar{\varphi}\|_{H^2}^{3/2}\|\Bar{\varphi}\|_{L^\infty}^{1/2}\|\rho\|_{V}^{1/2}\|\rho\|^{1/2} + \|\Bar{\varphi}\|_{H^3}^{3/4}\|\Bar{\varphi}\|_{L^\infty}^{1/4}\|\rho\|_{V}^{1/2}\|\rho\|^{1/2} \\
    &\hspace{.5cm}+ \|\Bar{\varphi}\|_{H^2}^{3/4}\|\Bar{\varphi}\|^{1/4}\|\rho\|_{H^2}^{3/4}\|\rho\|^{1/4} + \|\Delta\rho\| \big]\|\Delta^2\rho\| \\
    &\leq \epsilon\|\Delta^2\rho\| ^2 + C \big[ \|\xi\|_{H^2}^4 + \|\xi\|_{H^2}^2\|\xi\|_{V}^2 + \|\Bar{\varphi}\|_{H^2}^{4}\|\xi\|_{V}^{2}\|\xi\|_{H^2}^{2} \\
    &\hspace{.5cm}+ \|\Bar{\varphi}\|_{H^3}^{2}\|\xi\|_{V}^{3}\|\xi\| + \|\Bar{\varphi}\|_{H^2}^{4}\|\xi\|_{V}^{3}\|\xi\| + \|\Bar{\varphi}\|_{H^2}^{4}\|\rho\|_{V}\|\rho\| 
    \\ &\hspace{.5cm}+ \|\Bar{\varphi}\|_{H^3}^{2}\|\rho\|_{V}\|\rho\| + \|\Bar{\varphi}\|_{H^2}^{2}\|\rho\|_{H^2}^{2} + \|\Delta\rho\|^2 \big].
\end{align*}
The next term can be expressed as,
\begin{align*}
    (\nabla\cdot(m'(\Bar{\varphi})\xi\nabla(-\Delta\xi)), \Delta^2\rho) &= (m''(\Bar{\varphi})\nabla\Bar{\varphi}\xi\nabla(-\Delta\xi), \Delta^2\rho) + (m'(\Bar{\varphi})\nabla\xi\nabla(-\Delta\xi), \Delta^2\rho) + (m'(\Bar{\varphi})\xi\Delta(-\Delta\xi), \Delta^2\rho),
\end{align*}
Using similar techniques, we obtain
\begin{align*}
    |(m''(\Bar{\varphi})\nabla\Bar{\varphi}\xi\nabla(-\Delta\xi), \Delta^2\rho)| &\leq C\|\nabla\Bar{\varphi}\|_{L^8}\|\xi\|_{L^8}\|\nabla^3\xi\|_{L^4} \|\Delta^2\rho\|\\
    &\leq \epsilon\|\Delta^2\rho\|^2 + C\|\Bar{\varphi}\|_{H^2}^{3/2}\|\Bar{\varphi}\|_{L^\infty}^{1/2}\|\xi\|_{V}^{3/2}\|\xi\|^{1/2}\|\xi\|_{H^4}^{7/4}\|\xi\|^{1/4}\\
    &\leq \epsilon\|\Delta^2\rho\|^2 + C\|\Bar{\varphi}\|_{H^2}^2\|\xi\|_{V}^2\|\xi\|_{H^4}^2 + C\|\Bar{\varphi}\|_{H^2}\|\Bar{\varphi}\|_{L^\infty}\|\xi\|_{V}\|\xi\|_{H^4}^{3/2}\|\xi\|^{3/2}\\
     &\leq \epsilon\|\Delta^2\rho\|^2 + C\|\Bar{\varphi}\|_{H^2}^2\|\xi\|_{V}^2\|\xi\|_{H^4}^2 + C\|\Bar{\varphi}\|_{L^\infty}^2\|\xi\|_{H^4}\|\xi\|^3,\\
     |(m'(\Bar{\varphi})\nabla\xi\nabla(-\Delta\xi), \Delta^2\rho)| &\leq C\|\nabla\xi\|_{L^4}\|\nabla^3\xi\|_{L^4}\|\Delta^2\rho\| \leq \epsilon\| \Delta^2\rho\|^2 + C\|\xi\|_{H^2}^{3/2}\|\xi\|^{1/2}\|\xi\|_{H^4}^{7/4}\|\xi\|^{1/4}\\ 
     &\leq \epsilon\|\Delta^2\rho\|^2 + C\|\xi\|_{H^2}^2\|\xi\|_{H^4}^2+ C\|\xi\|_{H^2}\|\xi\|_{H^4}^{3/2}\|\xi\|^{3/2}\\ 
     &\leq \epsilon\|\Delta^2\rho\|^2 + C\|\xi\|_{H^2}^2\|\xi\|_{H^4}^2+ C\|\xi\|_{H^4}\|\xi\|^3,\\
     |(m'(\Bar{\varphi})\xi\Delta(-\Delta\xi), \Delta^2\rho)| &\leq C\|\xi\|_{L^\infty}\|\Delta^2\xi\|\|\Delta^2\rho\| \leq \epsilon\|\Delta^2\rho\|^2 + C\|\xi\|\|\xi\|_{H^2}\|\xi\|_{H^4}^2.
\end{align*}
By Taylor’s formula, we obtain the following expression.
\begin{equation*}
\nabla(F'(\varphi_h)-F'(\Bar{\varphi})) = F''(\varphi_h)\nabla\varphi_h -F''(\Bar{\varphi})\nabla\Bar{\varphi} = F'''(\Hat{\varphi})\nabla\varphi_h + F''(\Bar{\varphi})\nabla\xi,
\end{equation*}
where $\Hat{\varphi} = r\varphi_h + (1-r)\Bar{\varphi}$, $r\in (0,1)$. It follows that, 
\begin{align*}
     |(\nabla\cdot(m'(\Bar{\varphi})\xi\nabla(F'(\varphi_h)-F'(\Bar{\varphi}))), \Delta^2\rho)| &\leq 
     |(\nabla\cdot(m'(\Bar{\varphi})\xi(F'''(\Hat{\varphi})\nabla\varphi_{h}\xi + F''(\Bar{\varphi})\nabla\xi )), \Delta^2\rho)|\\
     &\leq 
     |(m''(\Bar{\varphi})\nabla\Bar{\varphi}\xi(F'''(\Hat{\varphi})\nabla\varphi_{h}\xi + F''(\Bar{\varphi})\nabla\xi ), \Delta^2\rho)| \\&+ |(m'(\Bar{\varphi})\nabla\xi(F'''(\Hat{\varphi})\nabla\varphi_{h}\xi + F''(\Bar{\varphi})\nabla\xi ), \Delta^2\rho)| \\
      &+ |(m'(\Bar{\varphi})\xi\Delta(F'(\varphi_h)-F'(\Bar{\varphi})), \Delta^2\rho)|,
\end{align*}
Now estimating each of the term separately by applying similar techniques, we obtain,
\begin{align*}
    |(m''(\Bar{\varphi})\nabla\Bar{\varphi}\xi(F'''(\Hat{\varphi})\nabla\varphi_{h}\xi + F''(\Bar{\varphi})\nabla\xi ), \Delta^2\rho)| &\leq C\|\nabla\Bar{\varphi}\|_{L^8}\|\nabla\varphi_{h}\|_{L^8}\|\xi\|_{L^8}^2 \|\Delta^2\rho\| + C \|\nabla\Bar{\varphi}\|_{L^8}\|\xi\|_{L^8}\|\nabla\xi\|_{L^4}\|\Delta^2\rho\|\\
     &\leq C\|\Bar{\varphi}\|_{H^2}^{3/4}\|\Bar{\varphi}\|_{L^\infty}^{1/4}\|\varphi_{h}\|_{H^2}^{3/4}\|\varphi_{h}\|_{L^\infty}^{1/4}\|\xi\|_{V}^{3/2}\|\xi\|^{1/2}\|\Delta^2\rho\| \\&\hspace{.5cm}+ C \|\Bar{\varphi}\|_{H^2}^{3/4}\|\Bar{\varphi}\|_{L^\infty}^{1/4}\|\xi\|_{V}^{3/4}\|\xi\|^{1/4}\|\xi\|_{H^2}^{1/2}\|\xi\|_V^{1/2}\|\Delta^2\rho\|\\
     &\leq \epsilon\|\Delta^2\rho\|^2 + C\|\Bar{\varphi}\|_{H^2}^{3/2}\|\Bar{\varphi}\|_{L^\infty}^{1/2}\|\varphi_{h}\|_{H^2}^{3/2}\|\varphi_{h}\|_{L^\infty}^{1/2}\|\xi\|_{V}^{3}\|\xi\|  \\&\hspace{.5cm}+ C \|\Bar{\varphi}\|_{H^2}^{3/2}\|\Bar{\varphi}\|_{L^\infty}^{1/2}\|\xi\|_{V}^{3/2}\|\xi\|^{1/2}\|\xi\|_{H^2}\|\xi\|_{V}\\
      &\leq \epsilon\|\Delta^2\rho\|^2 + C\|\Bar{\varphi}\|_{H^2}^{2}\|\varphi_{h}\|_{H^2}^{2}\|\xi\|_{V}^{3}\|\xi\|  \\&\hspace{.5cm}+ C \|\Bar{\varphi}\|_{H^2}^{2}\|\xi\|_{V}^3\|\xi\|_{H^2},
\end{align*}

\begin{align*}
      |(m'(\Bar{\varphi})\nabla\xi(F'''(\Hat{\varphi})\nabla\varphi_{h}\xi + F''(\Bar{\varphi})\nabla\xi ), \Delta^2\rho)| &\leq C\|\nabla\xi\|_{L^4}\|\nabla\varphi_{h}\|_{L^8}\|\xi\|_{L^8}\|\Delta^2\rho\| + C\|\nabla\xi\|_{L^4}^2\|\Delta^2\rho\|\\
      &\leq C\|\xi\|_{H^2}^{1/2}\|\xi\|_{V}^{1/2}\|\Bar{\varphi}\|_{H^2}^{3/4}\|\Bar{\varphi}\|_{L^\infty}^{1/4}\|\xi\|_{V}^{3/4}\|\xi\|^{1/4}\|\Delta^2\rho\| \\&\hspace{.5cm}+ C\|\xi\|_{H^2}\|\xi\|_{V}\|\Delta^2\rho\|\\
      &\leq 2\epsilon\|\Delta^2\rho\|^2 + C\|\xi\|_{H^2}\|\xi\|_{V}\|\Bar{\varphi}\|_{H^2}^2\|\xi\|_{V}^2+ C\|\xi\|_{H^2}^2\|\xi\|_{V}^2\\
       &\leq 2\epsilon\|\Delta^2\rho\|^2 + C\|\Bar{\varphi}\|_{H^2}^2\|\xi\|_{H^2}\|\xi\|_{V}^3 + C\|\xi\|_{H^2}^2\|\xi\|_{V}^2,
\end{align*}

\begin{align*}
      |(m'(\Bar{\varphi})\xi\nabla(F'''(\Hat{\varphi})\nabla\varphi_{h}\xi + F''(\Bar{\varphi})\nabla\xi ), \Delta^2\rho)| &\leq  |(m'(\Bar{\varphi})\xi(F^{(4)}(\Hat{\varphi})\nabla\Bar{\varphi}\cdot\nabla\varphi_h\xi + F'''(\Hat{\varphi})\Delta\varphi_{h}\xi+ F'''(\Bar{\varphi})\nabla\Bar{\varphi}\cdot\nabla\xi), \Delta^2\rho)| 
      \\&\hspace{.5cm} +|(m'(\Bar{\varphi})\xi(F'''(\Bar{\varphi})\nabla\varphi_{h}\cdot\nabla\xi + F''(\Bar{\varphi})\Delta\xi), \Delta^2\rho)| \\
      &\leq  \|\nabla\Bar{\varphi}\|_{L^8}\|\nabla\varphi_{h}\|_{L^8}\|\xi\|_{L^8}^2\|\Delta^2\rho\| + \|\Delta\varphi_{h}\|_{L^4}\|\xi\|_{L^8}^2\|\Delta^2\rho\| 
      \\&\hspace{.5cm}+ \|\nabla\Bar{\varphi}\|_{L^8}\|\xi\|_{L^8}\|\nabla\xi\|_{L^4}\|\Delta^2\rho\| + \|\nabla\varphi_{h}\|_{L^8}\|\xi\|_{L^8}\|\nabla\xi\|_{L^4}\|\Delta^2\rho\| \\&\hspace{.5cm}+ \|\xi\|_{L^\infty}\|\Delta\xi\|\|\Delta^2\rho\| \\  &\leq  \|\Bar{\varphi}\|_{H^2}\|\varphi_{h}\|_{H^2}\|\xi\|_{V}^{3/2}\|\xi\|^{1/2}\|\Delta^2\rho\| + \|\varphi_{h}\|_{H^3}^{3/4}\|\varphi_{h}\|^{1/4}\|\xi\|_{V}^{3/2}\|\xi\|^{1/2}\|\Delta^2\rho\| 
      \\&\hspace{.5cm}+ \|\Bar{\varphi}\|_{H^2}^{3/4}\|\Bar{\varphi}\|_{L^\infty}^{1/4}\|\xi\|_{V}^{3/4}\|\xi\|^{1/4}\|\xi\|_{H^2}^{1/2}\|\xi\|_{V}^{1/2}\|\Delta^2\rho\| \\&\hspace{.5cm}+ \|\varphi_{h}\|_{H^2}^{3/4}\|\varphi_{h}\|_{L^\infty}^{1/4}\|\xi\|_{V}^{3/4}\|\xi\|^{1/4}\|\xi\|_{H^2}^{1/2}\|\xi\|_{V}^{1/2}\|\Delta^2\rho\|+ \|\xi\|_{H^2}\|\Delta\xi\|\|\Delta^2\rho\|
      \\  &\leq 5\epsilon\|\Delta^2\rho\|^2 + 
      \|\Bar{\varphi}\|_{H^2}^2\|\varphi_{h}\|_{H^2}^2\|\xi\|_{V}^3\|\xi\| + \|\varphi_{h}\|_{H^3}^{2}\|\xi\|_{V}^3\|\xi\|
      \\&\hspace{.5cm}+ \|\Bar{\varphi}\|_{H^2}^{2}\|\xi\|_{H^2}\|\xi\|_{V}^3 + \|\varphi_{h}\|_{H^2}^{2}\|\xi\|_{H^2}\|\xi\|_{V}^3+ \|\xi\|_{H^2}^2\|\Delta\xi\|^2.
\end{align*}
Using Taylor's formula, we have,
\begin{align*}
    m(\varphi_h) = m(\Bar{\varphi}) + m'(\Bar{\varphi})(\varphi_h-\Bar{\varphi}) + m'''(\Hat{\varphi}) (\varphi_h-\Bar{\varphi})^2. 
\end{align*}
\begin{align*}
    |\big(\nabla\cdot((m(\varphi_h)-m(\Bar{\varphi})-m'(\Bar{\varphi})\psi)\nabla\Bar{\mu}), \Delta^2\rho\big)| &\leq  |\big(\nabla(m(\varphi_h)-m(\Bar{\varphi})-m'(\Bar{\varphi})\psi)\nabla\Bar{\mu}, \Delta^2\rho\big)| \\&\hspace{.5cm}+ |\big( (m(\varphi_h)-m(\Bar{\varphi})-m'(\Bar{\varphi})\psi)\Delta\Bar{\mu}, \Delta^2\rho\big)|
\end{align*}
We used Gagliardo-Nirenberg, Agmon's and Sobolev inequalities to estimate the following. 
\begin{align*}
    |\big(\nabla(m(\varphi_h)-m(\Bar{\varphi})-m'(\Bar{\varphi})\psi)\nabla\Bar{\mu}, \Delta^2\rho\big)| &= |\big(\nabla(m'(\Bar{\varphi})\rho + m''(r\varphi_h + (1-r)\Bar{\varphi})\xi^2)\cdot\nabla\Bar{\mu}, \Delta^2\rho\big)|\\
    &\leq  \|\nabla\Bar{\mu}\|_{L^4}\|\xi\|_{L^8}\|\nabla\xi\|_{L^8}\|\Delta^2\rho\| + \|\nabla\Bar{\varphi}\|_{L^8}\|\nabla\Bar{\mu}\|_{L^8}\|\xi\|_{L^8}^2\|\Delta^2\rho\| \\
    &\hspace{.5cm}+ \|\nabla\Bar{\mu}\|_{L^4}\|\nabla\rho\|_{L^4}\|\Delta^2\rho\| + \|\nabla\Bar{\varphi}\|_{L^8}\|\nabla\Bar{\mu}\|_{L^4}\|\rho\|_{L^8}\|\Delta^2\rho\|\\
    &\leq 4\epsilon\|\Delta^2\rho\|^2 + \|\Bar{\mu}\|_{H^2}^{3/2}\|\Bar{\mu}\|^{1/2}\|\xi\|_{V}^{3/2}\|\xi\|^{1/2}\|\xi\|_{H^2}^{3/2}\|\xi\|_{L^\infty}^{1/2} \\&\hspace{.5cm}+ \|\Bar{\varphi}\|_{H^2}^{3/2}\|\Bar{\varphi}\|_{L^\infty}^{1/2}\|\Bar{\mu}\|_{H^2}^{3/2}\|\Bar{\mu}\|_{L^\infty}^{1/2}\|\xi\|_{V}^{3}\|\xi\| + \|\Bar{\mu}\|_{H^2}^{3/2}\|\Bar{\mu}\|^{1/2}\|\rho\|_{H^2}^{3/2}\|\rho\|^{1/2} \\
    &\hspace{.5cm}+ \|\Bar{\varphi}\|_{H^2}^{3/2}\|\Bar{\varphi}\|_{L^\infty}^{1/2}\|\Bar{\mu}\|_{H^2}^{3/2}\|\Bar{\mu}\|^{1/2}\|\rho\|_{V}^{3/2}\|\rho\|^{1/2}\\
    &\leq 4\epsilon\|\Delta^2\rho\|^2 + \|\Bar{\mu}\|_{H^2}^{2}\|\xi\|_{V}^{2}\|\xi\|_{H^2}^{2} + \|\Bar{\varphi}\|_{H^2}^{2}\|\Bar{\mu}\|_{H^2}^{2}\|\xi\|_{V}^{3}\|\xi\| \\&\hspace{.5cm}+ \|\Bar{\mu}\|_{H^2}^{2}\|\rho\|_{H^2}^{2} + \|\Bar{\varphi}\|_{H^2}^{2}\|\Bar{\mu}\|_{H^2}^{2}\|\rho\|_{V}^{2}.
\end{align*}
\begin{align*}
     |\big( (m(\varphi_h)-m(\Bar{\varphi})-m'(\Bar{\varphi})\psi)\Delta\Bar{\mu}, \Delta^2\rho\big)| &\leq   \|\Delta\Bar{\mu}\|\|\xi\|_{L^\infty}^2\|\Delta^2\rho\| + \|\Delta\Bar{\mu}\|\|\rho\|_{L^\infty}\|\Delta^2\rho\| \\
     &\leq 2\epsilon\|\Delta^2\rho\|^2 + \|\Delta\Bar{\mu}\|^2\|\xi\|_{H^2}^2\|\xi\|^2 + \|\Delta\Bar{\mu}\|^2\|\rho\|_{H^2}\|\rho\|.
\end{align*}
Now by combining all above estimates, and for a choice of $\epsilon< \frac{m_0}{40}$ we obtain,

\begin{align*}
    \frac{1}{2}\frac{d}{dt}\|\Delta\rho\|^2 + \frac{m_0}{2}\|\Delta^2\rho\|^2 &\leq \Lambda_1 + C(1+ \Lambda_0) \|\Delta\rho\|^2, 
\end{align*} 
Where $\Lambda_0$ and $\Lambda_1$ are given by,
\begin{align*}
    \Lambda_0 &= 1 + \|\Bar{\varphi}\|_{H^2}^{2} + \|\Bar{\mu}\|_{H^2}^2, \\
    \Lambda_1&= \|\bm{\tau}\|\|\nabla\bm{\tau}\|\|\nabla\xi\|\|\xi\|_{H^2}  + \|\textbf{v}\|\|\nabla\textbf{v}\|\|\nabla\Bar{\varphi}\|\|\Bar{\varphi}\|_{H^2}+ \|\xi\|_{H^2}^4 + \|\Bar{\varphi}\|_{H^2}^2\|\xi\|_{V}^2\|\xi\|_{H^4}^2 \\&\hspace{.5cm}+ (1+ \|\Bar{\varphi}\|_{L^\infty}^2)\|\xi\|_{H^4}\|\xi\|^3+  \|\xi\|_{H^4}^2\|\xi\|_{H^2}^2 + \|\xi\|\|\xi\|_{H^2}\|\xi\|_{H^4}^2 \\
    &\hspace{.5cm}+ \Big[ 1+ \|\Bar{\varphi}\|_{H^2}^{2} + \|\Bar{\varphi}\|_{H^2}^4 + \|\Bar{\mu}\|_{H^2}^2  \Big]\|\xi\|_{H^2}^{2}\|\xi\|_{V}^{2} + (\|\Bar{\varphi}\|_{H^2}^{2} + \|\varphi_{h}\|_{H^2}^{2})\|\xi\|_{V}^3\|\xi\|_{H^2}
    \\&\hspace{.5cm}+ \Big[ \|\Bar{\varphi}\|_{H^3}^2 + \|\Bar{\varphi}\|_{H^2}^4 + \|\Bar{\varphi}\|_{H^2}^2\|\Bar{\mu}\|_{H^2}^2 + \|\Bar{\varphi}\|_{H^2}^2\|\varphi_{h}\|_{H^2}^2 + \|\varphi_{h}\|_{H^3}^{2} \Big]\|\xi\|_{V}^{3}\|\xi\| \\&\hspace{.5cm}+\Big[1 + \|\nabla\Bar{\varphi}\|^8\|\Bar{\varphi}\|_{H^2}^8 + \|\Bar{\varphi}\|_{H^2}^2  + \|\Bar{\mu}\|_{H^2}^2 \Big]\|\rho\|^2 + \|\Bar{\varphi}\|_{H^2}^2\|\Bar{\mu}\|_{H^2}^2\|\rho\|_{V}^2  \\
    &\hspace{.5cm}+ \Big[\|\Bar{\varphi}\|_{H^2}^{3}\|\Bar{\varphi}\|_{L^\infty} + \|\Bar{\varphi}\|_{H^2}^4 + \|\Bar{\varphi}\|_{H^3}^2 \Big]\|\rho\|_{V}\|\rho\|
\end{align*}
Observe that $\int\limits_{0}^{T}\Lambda_0 \leq C$ and $\int\limits_{0}^{T}\Lambda_1 \leq \|h\|_{L^2(0,T;H)}^2$. By exploiting Gronwall's lemma and using an elliptic estimate, we obtain,
\begin{align*}
    \|\rho\|_{L^\infty(0,T;H^2)\cap L^2(0,T;H^4)} \leq C\|h\|_{L^2(0,T;H)}^2.
\end{align*}
Furthermore, consider \eqref{eq2.30} tested with $\chi \in V$, 
\begin{align*}
     \langle\rho', \chi\rangle + (\bm{\tau}\cdot\nabla\xi, \chi) + (\Bar{\textbf{u}}\cdot\nabla\rho, \chi) + (\textbf{v}\cdot\nabla\Bar{\varphi}, \chi) &= (m(\Bar{\varphi})\nabla(-\Delta\rho), \nabla\chi)+ \big(m(\Bar{\varphi})\nabla(F'(\varphi_h)-F'(\Bar{\varphi})-F''(\Bar{\varphi})\psi), \nabla\chi \big)\nonumber\\&\hspace{.5cm} (m'(\Bar{\varphi})\xi\nabla(-\Delta\xi), \nabla\chi)+  \big(m'(\Bar{\varphi})\xi\nabla(F'(\varphi_h)-F'(\Bar{\varphi})), \nabla\chi \big)\nonumber\\&\hspace{.5cm}
    \big((m(\varphi_h)-m(\Bar{\varphi})-m'(\Bar{\varphi})\psi)\nabla(-\Delta\Bar{\varphi}+F'(\Bar{\varphi})), \nabla \chi\big)
\end{align*}

estimating each of its terms by applying H\"olders, Gagliardo-Nirenberg and Young's inequalities, we obtain,
\begin{align*}
    |\langle\rho', \chi\rangle| &\leq C\Big[ \|\bm{\tau}\|_{L^4}\|\xi\|_{L^4} + \|\Bar{\textbf{u}}\|_{L^4}\|\rho\|_{L^4} + \|\textbf{v}\|_{L^4}\|\Bar{\varphi}\|_{L^4}  + \|\nabla^3\rho\| + \|\nabla\Bar{\varphi}\|_{L^4}\|\xi\|_{L^8}^2 \\&\hspace{.5cm}+ \|\xi\|_{L^4}\|\nabla\xi\|_{L^4}  + \|\nabla\Bar{\varphi}\|_{L^4}\|\rho\|_{L^4} +\|\nabla\rho\| +\|\xi\|_{L^4}\|\nabla^3\xi\|_{L^4}  \\
    &\hspace{.5cm}+ \|\rho\|_{L^4}\|\nabla\Bar{\mu}\|_{L^4} + \|\xi\|_{L^8}^2\|\nabla\Bar{\mu}\|_{L^4} \Big]\|\nabla\chi\| \\
     &\leq C\Big[ \|\nabla\bm{\tau}\|\|\xi\|^{1/2}\|\xi\|_{V}^{1/2} + \|\nabla\Bar{\textbf{u}}\|\|\rho\|^{1/2}\|\rho\|_{V}^{1/2} + \|\nabla\textbf{v}\|\|\Bar{\varphi}\|^{1/2}\|\Bar{\varphi}\|_{V}^{1/2}  \\
    &\hspace{.5cm}+ \|\nabla^3\rho\|+ \|\Bar{\varphi}\|_{H^2}^{1/2}\|\nabla\Bar{\varphi}\|^{1/2}\|\xi\|_{V}^{3/2}\|\xi\|^{1/2} + \|\xi\|_{H^2}^{3/4}\|\xi\|^{1/4}\|\xi\|_{H^2}^{3/4}\|\xi\|^{1/4}   \\
    &\hspace{.5cm}+  \|\Bar{\varphi}\|_{H^2}^{1/2}\|\nabla\Bar{\varphi}\|^{1/2}\|\rho\|_{V}^{1/2}\|\rho\|^{1/2} +\|\nabla\rho\| + \|\xi\|_{V}^{1/2}\|\xi\|^{1/2}\|\xi\|_{H^4}^{7/8}\|\xi\|^{1/8}  \\
    &\hspace{.5cm}+ \|\rho\|_{V}^{1/2}\|\rho\|^{1/2}\|\Bar{\mu}\|_{H^2}^{1/2}\|\nabla\Bar{\mu}\|^{1/2} + \|\xi\|_{V}^{3/2}\|\xi\|^{1/2}\|\Bar{\mu}\|_{H^2}^{1/2}\|\nabla\Bar{\mu}\|^{1/2} \Big]\|\nabla\chi\| \\
    &\leq C\Big[ \|\nabla\bm{\tau}\|\|\xi\|_{V} + \|\nabla\Bar{\textbf{u}}\|\|\rho\|_{V} + \|\nabla\textbf{v}\|\|\Bar{\varphi}\|_{V} + \|\rho\|_{H^3} + \|\Bar{\varphi}\|_{H^2}\|\xi\|_{V}^{2} + \|\xi\|_{H^2}\|\xi\|  \\
    &\hspace{.5cm}+  \|\Bar{\varphi}\|_{H^2}\|\rho\|_{V} +\|\rho\|_{V} + \|\xi\|_{V}\|\xi\|_{H^4} + \|\rho\|_{V}\|\Bar{\mu}\|_{H^2} + \|\xi\|_{V}^{2}\|\Bar{\mu}\|_{H^2} \Big]\|\nabla\chi\| \\
\end{align*}

Now integrating above inequality over $[0,T]$ and by applying the regularity results and the stability estimates we get,
\begin{equation*}
    \|\rho'\|_{L^2(0,T; V')} \leq C\|h\|_{L^2(0,T;H)}^2.
\end{equation*}
Hence the result follows. 
\end{proof} 
 \begin{remark}
     Note that by Aubin-Lions $L^\infty(0,T; H^2(\Omega))\cap H^1(0,T; V') \hookrightarrow C( \Bar{\Omega} \times[0,T])$. Hence we get, $\rho \rightarrow 0$ for a.e $(x,t) \in \Omega \times[0,T]$ as $\|h\|$ tends to $0$ using \eqref{eq4.4}. This limit is used to derive the first order optimality condition in theorem \ref{theorem4.4}.
 \end{remark}

\addcontentsline{toc}{chapter}{Bibliography}
\printbibliography
\nocite{SKN}

\section*{Statements and Declarations}

\textbf{Acknowledegment}: Greeshma K would like to thank the Department of Science and Technology (DST), India, for the Innovation in Science Pursuit for Inspired Research (INSPIRE) Fellowship (IF210199). 
\\\\
\textbf{Conflict of Interest}: There is no conflict of interest monetary or otherwise  amongst the two authors.
\\\\
\textbf{Author Contribution}:
Authors would like to declare that both of them have participated equally in designing, concept. calculation and preparation of the manuscript.

\end{document}